\documentclass[11pt]{amsart}
\usepackage{amsmath}
\usepackage{amsfonts}
\usepackage{amssymb}
\usepackage{amsthm}
\usepackage{url}
\usepackage[all]{xy}
\usepackage{graphicx}
\usepackage{caption}
\usepackage{subcaption}
\usepackage{comment} 
\usepackage{hyperref}
\newcommand{\llbracket}{\mathopen{[\![}}
\newcommand{\rrbracket}{\mathclose{]\!]}}
\usepackage{todonotes}
\usepackage{color}
\usepackage{enumerate,tikz-cd}
\usepackage{comment}
\usepackage{mathrsfs}
\allowdisplaybreaks
\usepackage{geometry}
\numberwithin{equation}{section}

\newtheorem{proposition}{Proposition}[section]

\newtheorem{lemma}[proposition]{Lemma}
\newtheorem{theorem}[proposition]{Theorem}
\newtheorem{corollary}[proposition]{Corollary}

\theoremstyle{definition}
\newtheorem{remark}[proposition]{Remark}
\newtheorem{definition}[proposition]{Definition}

\newtheorem{claim}{Claim}

\DeclareMathOperator{\Aut}{Aut}

\DeclareMathOperator{\Proj}{Proj}
\DeclareMathOperator{\Spec}{Spec}

\DeclareMathOperator{\codim}{codim}

\DeclareMathOperator{\rk}{rk}

\DeclareMathOperator{\DF}{DF}

\DeclareMathOperator{\lct}{lct}
\DeclareMathOperator{\Ding}{Ding}

\DeclareMathOperator{\ord}{ord}

\DeclareMathOperator{\coeff}{coeff}

\DeclareMathOperator{\ST}{ST}

\DeclareMathOperator{\PGL}{PGL}

\DeclareMathOperator{\Val}{Val}

\newcommand{\N}{\mathbb{N}}
\newcommand{\R}{\mathbb{R}}
\newcommand{\C}{\mathbb{C}}
\newcommand{\A}{\mathbb{A}}
\newcommand{\Z}{\mathbb{Z}}

\newcommand{\Q}{\mathbb{Q}}
\newcommand{\G}{\mathbb{G}}

\newcommand{\PP}{\mathbb{P}}
\newcommand{\QQ}{\mathbb{Q}}
\newcommand{\OO}{\mathcal{O}}
\newcommand{\GG}{\mathbb{G}}
\renewcommand{\epsilon}{\varepsilon}

\newcommand{\M}{\mathcal{M}}

\newcommand{\I}{\mathcal{I}}

\newcommand{\F}{\mathcal{F}}
\newcommand{\B}{\mathcal{B}}

\newcommand{\E}{\mathcal{E}}
\renewcommand{\L}{\mathcal{L}}
\newcommand{\X}{\mathcal{X}}
\newcommand{\Y}{\mathcal{Y}}

\renewcommand{\phi}{\varphi}

\newcommand\cA{\mathcal{A}}
\newcommand\cB{\mathcal{B}}
\newcommand\cC{\mathcal{C}}

\newcommand\cG{\mathcal{G}}
\newcommand\cH{\mathcal{H}}

\newcommand\cO{\mathcal{O}}

\newcommand\cZ{\mathcal{Z}}

\newcommand\NA{^{\mathrm{NA}}}

\newcommand\red{{\mathrm{red}}}

\title{$\Theta$-reductivity and $S$-completeness for adjoint Fano foliated structures}

\author[T. S. Papazachariou]{Theodoros Stylianos Papazachariou}
\address{Yau Mathematical Sciences Center, Jingzhai, Tsinghua University, Haidian District, Beijing, China.}
\email{tpapazachariou@mail.tsinghua.edu.cn}

\begin{document}

\begin{abstract}
We prove the valuative criteria of $\Theta$-reductivity and $S$-completeness for the moduli problem of $t$-K-semistable adjoint Fano foliated structures. We develop a mixed Ding theory for arbitrary linearly bounded multiplicative filtrations, prove inversion of adjunction with arbitrary ideals for adjoint foliated structures, and establish a relative extraction and finite generation theorem. Together, these results yield the required relative extension theorems for families. As applications, we prove uniqueness of $t$-K-polystable degenerations, reductivity of the automorphism group of $t$-K-polystable adjoint Fano foliated structures, and finiteness of the automorphism group in the $t$-K-stable case.
\end{abstract}

\maketitle

\setcounter{tocdepth}{1}
{\hypersetup{hidelinks}
\tableofcontents}

\section{Introduction}
One of the main achievements of the theory of K-stability for Fano varieties is the construction of a moduli theory for K-semistable Fano varieties that admits a good projective K-moduli space parametrising K-polystable Fanos. The main inputs for the existence of such K-moduli spaces include the verification of the valuative criteria of $\Theta$-reductivity and $S$-completeness \cite{BX19, ABHLX20}, which give sufficient conditions for the existence of good moduli spaces for Artin stacks \cite{AlperHalpernLeistnerHeinloth2023}. 
% This construction built upon prior results on K-stability in families, which showed that the set of K-semistable Fano varieties with bounded volume is bounded \cite{Jiang2020}, and that the locus of K-semistable Fano varieties is Zariski open \cite{BLX22}, and allowed the construction of a K-moduli Artin stack parametrising K-semistable Fano varieties.

In recent work, the author defined a K-stability notion for adjoint Fano foliated structures \cite{Pap26} and showed that K-semistable adjoint Fano foliated structures are klt \cite[Theorem 1.2(3)]{Pap26b}. Adjoint Fano foliated structures are triples $(X,\F,t)$, where $X$ is a normal variety, $\F$ is an algebraically integrable foliation on $X$ and $t\in (0,1)\cap \Q$ is a rational number, and the adjoint anticanonical divisor $-K_{X,\F}^{[t]}$ is ample. Due to recent developments in the MMP and birational geometry of these structures \cite{CHLMSSX24,CHLMSSX25, CLSV26}, these provide a natural framework for extending Fano geometry to the foliated setting, since Fano foliations, i.e. foliations with $-K_\F$ ample, exhibit a number of pathologies. 

In this paper, as an extension of the results in \cite{Pap26, Pap26b}, we prove the valuative criteria for $\Theta$-reductivity and $S$-completeness for the moduli problem of $t$-K-semistable adjoint Fano foliated structures, with respect to essentially of finite type DVRs. This choice is anavoidable since we rely on techniques from the adjoint foliated MMP which are unavailable over every DVR $R$. Although an Artin stack parametrising $t$-K-semistable adjoint Fano foliated structures has not been constructed yet, as openness of the $t$-K-semistable locus requires substantial new input from the adjoint foliated MMP, the above criteria can be verified valuatively by studying specific extensions of families of $t$-K-semistable adjoint Fano foliated structures, and as such do not require the prior construction of the stack itself. 
% We expect that the stack itself is a finite type Artin stack, in which case using \cite[Theorem A, Remark 5.5]{AlperHalpernLeistnerHeinloth2023}, the results of this paper will imply that the stack of K-semistable adjoint Fano foliated structures satisfies the valuative criteria for $S$-completeness and $\Theta$-reductivity for all DVRs $R$. 
We expect that this moduli problem is represented by a finite-type Artin stack with affine diagonal over $k$. Under this additional assumption, \cite[Theorem 5.4]{AlperHalpernLeistnerHeinloth2023} implies that it is enough to verify $\Theta$-reductivity and $S$-completeness for DVRs essentially of finite type over $k$. Consequently, the results of this paper would imply the corresponding valuative criteria for arbitrary DVRs. We prompt the reader to \cite[Remark 2.10, Remark 3.4]{ABHLX20} and the surrounding discussions for more details in the K-moduli setting. The following theorems are the main results of the paper. 

\begin{theorem}[Valuative criterion for $\Theta$-reductivity. See Theorem \ref{thm:adjoint-foliated-theta-extension}]
\label{thm:adjoint-foliated-theta-extension-intro}
Let $R$ be DVR essentially of finite type over $k$, with fraction field $K$ and residue field $\kappa$. Let $(\X,\F,t)\to \Spec R$ be an admissible family of $t$-K-semistable adjoint Fano foliated structures, and set $\L:=-K^{[t]}_{\X/R,\F}$. Let $(\Y_K,\F_{\Y_K},\M_K) \to \A^1_K$ be a special $t$-K-semistable degeneration of the generic fibre $(X_K,\F_K,L_K)$. Then, after replacing $R$ by a finite extension and abusing notation, there exists a unique admissible family $(\Y,\F_{\Y},\M) \to \A^1_R$ such that
\[
        (\Y,\F_{\Y},\M)\times_R K \simeq (\Y_K,\F_{\Y_K},\M_K),
\]
and
\[
        (\Y,\F_{\Y},\M)|_{\{1\}\times \Spec R} \simeq (\X,\F,\L).
\]
Moreover, every geometric fibre of $(\Y,\F_{\Y},\M) \to \A^1_R$ is $t$-K-semistable, and $\M \sim_{\Q,\A^1_R} -K^{[t]}_{\Y/\A^1_R,\F_{\Y}}$.
\end{theorem}

\begin{theorem}[Valuative criterion for $S$-completeness. See Theorem \ref{thm:adjoint-foliated-S-completeness}]
\label{thm:adjoint-foliated-S-completeness-intro}
Let $(\X,\F,t)\to \Spec R$ and $(\X',\F',t)\to \Spec R$ be two admissible families of $t$-K-semistable adjoint Fano foliated structures, where $R$ is a DVR essentially of finite type over $k$. Set $\L:=-K^{[t]}_{\X/R,\F}$ and $\L':=-K^{[t]}_{\X'/R,\F'}$, and assume that the generic fibres are isomorphic $(X_K,\F_K,L_K)\simeq (X'_K,\F'_K,L'_K)$. Then the induced family over $\ST_R^\circ$ extends uniquely to an admissible family $(\cZ,\F_\cZ,\M)\to \ST_R$ whose geometric fibres are $t$-K-semistable.
\end{theorem}

The notion of admissible families is technical and necessary for our results, and is defined in Section \ref{sec: admissible families}. We omit their precise definitions here, noting that they are the foliated analogues of the stability and compatibility conditions imposed on families and test configurations in the usual K-moduli theory (see also \cite[\S 5]{Xu2025} and \cite[\S 3]{Kol23}). We refer the reader to \cite[\S 8.2.1 and \S 8.2.2]{Xu2025} for a detailed explanation on why the above theorems satisfy the valuative criteria of $\Theta$-reductivity and $S$-completeness. 

The proofs of the above theorems follow the ideas presented in \cite{BX19, ABHLX20}, and later streamlined in \cite[\S 8.2.1 and \S 8.2.2]{Xu2025}. In more detail, we first study the notion of $t$-Ding stability for adjoint Fano foliated structures defined in \cite[\S 7]{Pap26}, which we reinterpret in terms of arbitrary linearly bounded multiplicative filtrations. We define the Ding invariant using mixed log canonical slopes and show that $t$-K-semistability implies its non-negativity for every such filtration. The main advantage of the Ding invariant is its birational nature. 

In particular, since Theorems \ref{thm:adjoint-foliated-theta-extension-intro} and \ref{thm:adjoint-foliated-S-completeness-intro} both require finite generation properties, we aim to use Ding stability to derive finite generation using the main results in \cite{CHLMSSX25}. In order to do so we prove an inversion of adjunction statement in families, for arbitrary ideals.

\begin{theorem}[Inversion of adjunction with ideals. See Theorem \ref{thm:F-compatible-IOA-finite}]\label{thm:F-compatible-IOA-finite-intro}
Let $C$ be a smooth pointed curve with closed point $0$, and let $f:(X,\F,t)\to C$ be an admissible family, and let $S=X_0$ be a reduced normal Cartier fibre. Assume:
\begin{enumerate}
\item $T_{\F}\subset T_{X/C}$;
\item $X$ is potentially klt near $S$;
\item the divisor $K^{[t]}_{\X/C,\mathcal F}+(1-t)S$ is $\mathbb Q$-Cartier;
\item $(X,\F,(1-t)S,t)$ is lc near $S$.
\end{enumerate}
Let $\mathfrak a\subset \cO_X$ be any nonzero coherent ideal not vanishing identically along $S$, and let $\F_S$ be the induced foliation on $S$. Then
\[
        \lct^{[t]}_{X,\F;(1-t)S}(\mathfrak a) =  \lct^{[t]}_{S,\F_S}(\mathfrak a_S).
\]
\end{theorem}

The proof of Theorem \ref{thm:F-compatible-IOA-finite-intro} relies on the adjunction and MMP results of \cite{CHLMSSX24, CHLMSSX25}. After principalising an arbitrary ideal, adjunction is applied to its full divisorial transform. The theorem controls the mixed log canonical thresholds needed below and gives upper bounds on the relevant mixed log discrepancies. We then extend \cite[Corollaries 1.68 and 1.70]{Xu2025} to our adjoint foliated setting using the results of \cite{CHLMSSX24,CHLMSSX25}, which allows us to extract divisors in specific birational models preserving the foliations and obtain finite generation of section rings. In particular, we show the following:

\begin{theorem}[Relative non-terminal extraction and finite generation. See Theorem \ref{thm:relative-non-terminal-extraction-fg}]
\label{thm:relative-non-terminal-extraction-fg-intro}
Let $f:(\X,\F,t)\to \Spec R$ be an admissible DVR family. Assume:
\begin{enumerate}
\item $\X$ is potentially klt;
\item $T_{\F}\subset T_{\X/R}$;
\item $\F$ is algebraically integrable;
\item $L:=-K^{[t]}_{\X/R,\F}$ is $f$-ample and
$\Q$-Cartier, and $L_r=rL$ is Cartier;
\item $\cA_0=(\X,\F,(1-t)\X_\kappa,t)$ is lc with
$\X_\kappa$ as its unique lc place;
\item $D\geq 0$ is an effective $\Q$-Cartier divisor with
$D\sim_{\Q,R}qL_r$;
\item the adjoint foliated structure
\[
        \cA_D  = \left(\X,\F, (1-t)\X_\kappa+\frac{1}{rq}D, t\right)
\]
is lc near $\X_\kappa$;
\item $G$ is an exceptional prime divisor over $\X$ such that
\[
 A_{\cA_D}(G)<t\varepsilon_{\F}(G)+(1-t), \qquad A_{\cA_0}(G)>0.
\]
\end{enumerate}
Then there exists a projective birational morphism
$\rho:\cZ\to\X$ such that:
\begin{enumerate}
\item $\cZ$ is normal and $\Q$-factorial;
\item $G$ appears as a $\rho$-exceptional prime divisor on $\cZ$;
\item $G$ is the unique $\rho$-exceptional prime divisor;
\item $-G$ is $\rho$-ample;
\item the induced foliation $\F_{\cZ}$ is algebraically
integrable and relative over $\Spec R$;
\item the bigraded algebra
\[
        \mathcal R_G(\X,L_r):=\bigoplus_{a\geq 0}\bigoplus_{p\geq 0} H^0\!\left(\cZ, \cO_{\cZ}(a\rho^*L_r-pG)\right)
\]
is finitely generated, where the divisorial sheaves are understood reflexively.
\end{enumerate}
\end{theorem}

For the proof of Theorems \ref{thm:adjoint-foliated-theta-extension-intro} and \ref{thm:adjoint-foliated-S-completeness-intro}, we then study the filtrations that arise naturally in the study of morphisms to $\Theta^{\circ}$ and $\ST_R^{\circ}$ as in \cite{BX19, ABHLX20} and \cite[\S 8.2.1 and \S 8.2.2]{Xu2025}. Using the K-semistability of the fibres of each family we reduce to the assumptions of Theorem \ref{thm:relative-non-terminal-extraction-fg-intro}; this allows us to prove the existence and uniqueness of the necessary extensions over $\Theta$ and $\ST_R$.

Furthermore, using Theorem \ref{thm:adjoint-foliated-S-completeness-intro} we are able to study the K-polystable degenerations of adjoint Fano foliated structures and show that they are unique. This extends the main result of \cite{BX19} in the adjoint Fano foliated setting. In particular, we show the following.

\begin{theorem}[Uniqueness of $t$-K-polystable degenerations. See Theorem \ref{thm:uniqueness-polystable-degeneration}] \label{thm:uniqueness-polystable-degeneration-intro}
Let $(X,\F,t)$ be a $t$-K-semistable adjoint Fano foliated structure, with $L\sim_{\Q}-K^{[t]}_{X,\F}$ an ample $\Q$-Cartier polarisation. Let $(\X_1,\F_1,\L_1)\to \A^1$ and $(\X_2,\F_2,\L_2)\to \A^1$ be two $\F$-compatible special test configurations of $(X,\F,L)$ with $\DF^{[t]}(\X_i,\F_i,\L_i)=0$ and with $t$-K-polystable central fibres $(Y_i,\F_{Y_i},M_i)$. Then
\[
        (Y_1,\F_{Y_1},M_1) \simeq (Y_2,\F_{Y_2},M_2)
\]
as polarised foliated varieties.
\end{theorem}

Using this result and the argument via Iwahori decompositions detailed in \cite[\S 4]{ABHLX20} we also prove that K-polystable adjoint Fano foliated structures have reductive automorphism groups.

\begin{theorem}[Reductivity of the automorphism group. See Theorem \ref{thm:aut-reductive-adjoint-foliated}] \label{thm:aut-reductive-adjoint-foliated-intro}
Let $(X,\F,t)$ be a $t$-K-polystable adjoint Fano foliated structure. Then $\Aut(X,\F)$ is reductive.
\end{theorem}

Furthermore, by studying K-stable degenerations we also show that they are unique, and that they admit a finite automorphism group.

\begin{theorem}[See Theorem \ref{thm:stable-separatedness-adjoint-foliated} and Corollary \ref{cor:finite-aut-t-K-stable}] \label{thm:stable-separatedness-adjoint-foliated-intro}
Let $C$ be a smooth pointed curve with closed point $0$, and let $(\X,\F,t)\to C$ and $(\X',\F',t)\to C$ be admissible families of $t$-K-semistable adjoint Fano foliated structures. Suppose there is an isomorphism over the punctured curve $C^\circ=C\setminus\{0\}$:
\[
        (\X,\F)|_{C^\circ}\simeq (\X',\F')|_{C^\circ}.
\]
If the central fibre $(X_0,\F_0,t)$ is $t$-K-stable, then the isomorphism over $C^\circ$ extends uniquely to an isomorphism
\[
        (\X,\F,t)\simeq (\X',\F',t)
\]
over $C$. As a consequence, if $(X,\F,t)$ is $t$-K-stable, then $\Aut(X,\F)$ is finite. 
\end{theorem}

Together with the boundedness result of \cite[Theorem 1.6]{Pap26}, the results of this paper provide several of the key valuative inputs expected in the construction of a K-moduli space of $t$-K-polystable objects. The main remaining inputs, as mentioned before, is the openness of the $t$-K-semistable locus, which is expected to follow using new developments in the MMP of adjoint Fano foliated structures, as well as the construction and representability of the relative moduli stack.

\subsection*{Structure of the paper}
In Section \ref{sec:preliminaries} we recall the basic notions of foliations and adjoint foliated structures and their birational invariants such as mixed log discrepancies and mixed log canonical thresholds. We also recall the key K-stability definitions from \cite{Pap26} and their consequences for the singularities of adjoint Fano foliated structures presented in \cite{Pap26b}. We also recall some information on filtrations that will be used throughout.

In Section \ref{sec: admissible families} we define admissible families, which naturally extend Koll\'ar's notion of stable families to the adjoint foliated setting. In Section \ref{sec: F-compatible filtrations} we study the filtrations, base ideals, and flag ideals that play a central role in the proof, and we define the mixed Ding functional for arbitrary linearly bounded multiplicative filtrations.

In Section \ref{sec:F-compatible-IOA} we prove Theorem \ref{thm:F-compatible-IOA-finite-intro}, which is the first main technical result of the paper. Then, in Section \ref{sec:relative-extraction-fg} we extend some results of \cite{CHLMSSX24,CHLMSSX25} and we prove Theorem \ref{thm:relative-non-terminal-extraction-fg-intro}. In this section we also detail how the filtrations arising in the study of $\Theta$-reductivity and $S$-completeness satisfy our main technical assumptions.

Having obtained the main technical results we need, we prove Theorem \ref{thm:adjoint-foliated-theta-extension-intro} in Section \ref{sec: theta red}, and Theorem \ref{thm:adjoint-foliated-S-completeness-intro} in Section \ref{sec: s-completeness}. In particular, in both sections we study how filtrations give us finite generation and unique extensions necessary to prove the main theorems. We then use these results to prove Theorems \ref{thm:uniqueness-polystable-degeneration-intro}, \ref{thm:aut-reductive-adjoint-foliated-intro} and \ref{thm:stable-separatedness-adjoint-foliated-intro} in Section \ref{sec:reductivity-aut}.

\subsection*{Acknowledgments}
I would like to thank Caucher Birkar, Federico Bongiorno, Paolo Cascini, Ruadhai Dervan, Jihao Liu and Calum Spicer for many helpful conversations and valuable comments. I am supported by the Beijing Natural Science Foundation Project IS25037 and a Shuimu Scholar Programme Scholarship at Tsinghua University.

\section{Preliminaries}\label{sec:preliminaries}
\subsection{Foliations and adjoint foliated structures}

In this subsection we recall the basic definitions concerning foliations and adjoint foliated structures. We follow the standard conventions in the birational geometry of foliations; for general background on foliations we refer, for instance, to \cite{Brunella15,AD13, ACSS21, CS25}, and for adjoint foliated structures to \cite{CHLMSSX25,CHLMSSX24}. Throughout the paper, we fix $t\in \mathbb Q$ with $0<t<1$.

\begin{definition}\label{def: foliation}
Let $X$ be a normal variety. A \emph{foliation} on $X$ is a saturated coherent subsheaf $T_\F\subseteq T_X$ which is closed under the Lie bracket. Its rank is the generic rank of $T_\F$, and its codimension is $\codim(\F):=\dim X-\rk({\F})$. Equivalently, we can regard $\F$ as an integrable distribution on $X$.
\end{definition}

Since $T_{\F}$ is saturated in the reflexive sheaf $T_X$, it is reflexive. In particular, $\det(T_{\F}):=(\wedge^{\operatorname{rk}\F}T_{\F})^{**}$ is a rank-one reflexive sheaf, hence corresponds to a Weil divisor class on $X$.

\begin{definition}
Let $X$ be a normal variety and let $\F$ be a foliation on $X$. We say that $\F$ is \emph{algebraically integrable} if the leaf of $\F$ through a very general point of $X$ is Zariski open in an algebraic subvariety of $X$. Equivalently, the closure of the leaf through a very general point of $X_{\mathrm{reg}}$ is an algebraic subvariety.
\end{definition}

\begin{definition}\label{def: canonical class of foliation}
Let $X$ be a normal variety and let $\F$ be a foliation on $X$. The \emph{canonical class} of $\F$ is the Weil divisor class $K_{\F}$ determined by
\[
\OO_X(-K_{\F})\simeq \det(T_{\F}),
\]
or equivalently
\[
\OO_X(K_{\F})\simeq \det(T_{\F})^{\vee}.
\]
When $X$ is smooth, if $N_{\F}^{\vee}:=(T_X/T_{\F})^{\vee}$ denotes the conormal sheaf of $\F$, then
\[
K_X = K_{\F} + \det(N_{\F}^{\vee}).
\]
\end{definition}

We will usually also impose as an extra condition that $K_{\F}$ is $\Q$-Cartier.

\begin{definition}
A \emph{foliated pair} consists of a normal variety $X$ together with a foliation $\F$ such that $K_{\F}$ is $\QQ$-Cartier. A \emph{foliated polarised variety} is a triple $(X,\F,L)$ such that $(X,\F)$ is a foliated pair, and $L$ is a $\Q$-ample divisor on $X$.
\end{definition}

We now recall the main object of interest in this paper.

\begin{definition}\label{def: adjoint foliated structure}
An \emph{adjoint foliated structure} is a triple $(X,\F,t)$, where $X$ is a normal variety, $\F$ is a foliation on $X$, and $t\in[0,1]\cap \Q$. To such a triple we associate the $\QQ$-divisor
\[
K_{X,\F}^{[t]}:= tK_{\F}+(1-t)K_X.
\]
We say that $(X,\F,t)$ is $\QQ$-Gorenstein if $K_{X,\F}^{[t]}$ is $\QQ$-Cartier.
\end{definition}

For our purposes we will always implicitly assume that $\F$ is algebraically integrable. Unless explicitly stated otherwise, we work with adjoint foliated structures for which both $K_X$ and $K_{\F}$ are $\Q$-Cartier. Thus our standing assumptions are slightly stronger than $\mathbb Q$-Gorensteinness of the adjoint combination.

\begin{definition}
    We say that the adjoint foliated structure $(X,\F,t)$ is an \emph{adjoint Fano foliated structure} if it is $\Q$-Gorenstein and $-K_{X,\F}^{[t]}$ is ample. 
\end{definition}

\subsection{Foliated log discrepancy}

Let $E$ be a prime divisor over $X$, and let $\mu:Y\to X$ be a birational model on which $E$ appears. Denote by $\F_Y$ the induced foliation on $Y$. We write
\[
        K_{\F_Y}=\mu^*K_\F+\sum_i a(E_i,\F)E_i
\]
and define
\[
        \varepsilon_\F(E)=
        \begin{cases}
        0, & E \text{ is } \F_Y\text{-invariant},\\
        1, & E \text{ is } \F_Y\text{-transverse, i.e. not } \F_Y\text{-invariant}.
        \end{cases}
\]
The foliated log discrepancy is
\[
        A_{X,\F}(E):=a(E,\F)+\varepsilon_\F(E).
\]

For an adjoint foliated structure $(X,\F,t)$ we define the \emph{mixed log discrepancy},
\[
A^{[t]}_{X,\F}(E) = t\bigl(a(E,\F)+\varepsilon_\F(E)\bigr)+(1-t)\bigl(a(E,X)+1\bigr).
\]
Equivalently,
\[
        A^{[t]}_{X,\F}(E) :=  (1-t)A_X(E)+tA_{X,\F}(E).
\]

\begin{definition}
Let $(X,\F,t)$ be an adjoint foliated structure with $t\in[0,1]\cap \Q$. We say that $(X,\F,t)$ is
\begin{enumerate}
    \item \emph{log canonical} if $A^{[t]}_{X,\F}(E)\ge 0$ for every prime divisor $E$ over $X$;
\item \emph{klt} if $A^{[t]}_{X,\F}(E)> 0$ for every prime divisor $E$ over $X$;
\item \emph{$\epsilon$-log canonical} if $A^{[t]}_{X,\F}(E)\ge \epsilon$ for every prime divisor $E$ over $X$;
\item \emph{$\epsilon$-klt} if $A^{[t]}_{X,\F}(E)> \epsilon$ for every prime divisor $E$ over $X$.
\end{enumerate}
\end{definition}

When $u=0$, this definition recovers the usual singularity condition for $X$. When $t=1$, this definition recovers the foliated singularity condition. We will say that $X$ is \emph{potentially klt} if there exists an effective $\mathbb Q$-divisor $\Delta$ such that $(X,\Delta)$ is klt.

We will now define a mixed log canonical threshold for adjoint foliated structures. We extend $A^{[t]}_{X,\F}$ to real valuations by taking the usual lower-semicontinuous extension of the divisorial discrepancy function; equivalently, in the computations below it suffices to take the infimum over divisorial valuations.

\begin{definition}
\label{def:mixed-lct}
Let $D\ge 0$ be an effective $\QQ$-divisor on $X$. We define the \emph{mixed log canonical threshold} of $D$ by
\begin{equation}\label{eq: mixed lct}
    \lct^{[t]}(X,\F;D):= \inf_{v\in \Val_X^*}\frac{A^{[t]}_{X,\F}(v)}{v(D)},
\end{equation}
where $\Val_X^*$ denotes the set of nontrivial valuations of finite mixed discrepancy. Equivalently,
\[
\lct^{[t]}(X,\F;D) = \inf_{v \in \Val_X^{\mathrm{div}, *}} \frac{A^{[t]}_{X,\F}(v)}{v(D)}.
\]
\end{definition}

The same convention applies to ideals. If $\mathfrak a\subset \mathcal O_X$ is a nonzero ideal, then
\[
\lct^{[t]}(X,\F;\mathfrak a) = \inf_v \frac{A^{[t]}_{X,\F}(v)}{v(\mathfrak a)}.
\]
Valuations with $v(\mathfrak a)=0$ contribute $+\infty$.

\subsection{K-stability for adjoint foliated structures}\label{sec: prelim K-stab}

We recall the test configurations and stability notions used below.

Let $X$ be a normal projective variety of dimension $n$ over $\mathbb{C}$ and $L$ an ample $\mathbb{Q}$-line bundle on $X$. Let $\mathcal{F}\subset T_X$ be an algebraically integrable foliation such that the canonical divisor $K_{\mathcal{F}}$ is $\mathbb{Q}$-Cartier. 

\begin{definition}\label{def: foliated t.c.}
A \emph{foliated test configuration} for the polarised foliated variety $(X,\F,L)$ consists of a triple $(\pi:\X\to \A^1,\ \F_{\X},\ \L)$ such that:

\begin{enumerate}
\item $(\X,\L)$ is a normal test configuration for $(X,L)$, i.e.\ $\pi:\X\to\A^1$ is a flat projective morphism endowed with a $\G_m$-action lifting the standard action on $\A^1$, $\L$ is a relatively ample $\Q$-line bundle, and
\[
(\X,\L)|_{\pi^{-1}(\A^1\setminus\{0\})} \simeq (X,L)\times (\A^1\setminus\{0\})
\]
$\G_m$-equivariantly;

\item $\F_{\X}\subset T_{\X/\A^1}$ is a
$\G_m$-equivariant saturated integrable subsheaf;

\item over $\A^1\setminus\{0\}$ the foliation coincides with the product foliation: $$(\X,\F_{\X})|_{\pi^{-1}(\A^1\setminus\{0\})}
\simeq (X,\F)\times (\A^1\setminus\{0\});$$

\item $\F_{\X}$ is algebraically integrable.
\end{enumerate}

Such a test configuration will be called \emph{$\F$-compatible}.
\end{definition}

Given an arbitrary $\F$-compatible test configuration $\pi:(\X,\F_{\X},\L)\rightarrow \A^1$, let $\bar{\pi}:(\bar{\X},\F_{\bar{\X}},\bar{\L})\rightarrow \PP^1$ be its natural compactification and set $V:=L^n$ and 
\[\mu(X,\F,L) = \frac{-K^{[t]}_{X,\F}\cdot L^{n-1}}{L^n}\]
for the \emph{foliated slope}. Furthermore, we define the \emph{relative mixed canonical divisor} by
\[
K^{[t]}_{\bar{\X}/\PP^1,\F_{\bar{\X}}}:= (1-t)K_{\bar{\X}/\PP^1}+tK_{\F_{\bar{\X}}}.
\]

\begin{definition}\label{def:DFwt-corrected}
Let $(\X,\ \F_{\X},\ \L)$ be a normal $\F$-compatible test configuration for $(X,\F,L)$. We fix $t\in (0,1)$. The \emph{$t$-foliated Donaldson--Futaki invariant} is defined by
\[
\DF^{[t]}(\X,\F_{\X},\L) := \frac{1}{V}\left( \frac{n}{n+1}\,\mu(X,\F,L)\bar{\L}^{n+1} + K^{[t]}_{\bar{\X}/\PP^1,\F_{\bar{\X}}}\cdot \bar{\L}^n \right).
\]
\end{definition}

We will also call this invariant the ``mixed Donaldson--Futaki'' invariant for brevity. An $\F$-compatible test configuration is called product, respectively trivial, if the underlying test configuration is product, respectively trivial, and the foliation is the induced product foliation. It is called special if the underlying test configuration is special and the central fibre, equipped with the induced foliation and adjoint anticanonical polarisation, is an adjoint Fano foliated structure. The divisor $K^{[t]}_{\overline{\X}/\mathbb P^1, \F_{\overline{\X}}}$ is understood as a $\mathbb Q$-Weil divisor class. The intersection number $K^{[t]}_{\overline{\X}/\mathbb P^1, \F_{\overline{\X}}} \cdot \overline{\mathcal L}^{\,n}$ is interpreted in the usual Odaka--Wang sense, as the intersection of a Weil divisor/rank-one reflexive sheaf with $n$ $\mathbb Q$-Cartier polarisation classes \cite{Wang2012, Odaka12}. In particular, the definition of an $\F$-compatible test configuration does not require $K_{\overline{\X}/\mathbb P^1}$, $K_{\F_{\overline{\X}}}$, or their mixed combination to be $\mathbb Q$-Cartier.

\begin{definition}\label{def: k-stability definition}
Let $(X,\F,L)$ be a polarised foliated variety.

\begin{enumerate}
\item $(X,\F,L)$ is \emph{$t$-K-semistable} if $\DF^{[t]}(\X,\F_{\X},\L)\ge0$ for all normal $\F$-compatible test configurations.

\item $(X,\F,L)$ is \emph{$t$-K-stable} if $\DF^{[t]}(\X,\F_{\X},\L)>0$ for all non-trivial $\F$-compatible test configurations.

\item $(X,\F,L)$ is \emph{uniformly $t$-K-stable} if there exists $\zeta>0$ such that $\DF^{[t]}(\X,\F_{\X},\L) \ge \zeta\,J^{NA}(\X,\L)$ for all normal $\F$-compatible test configurations.
\item $(X,\F,L)$ is \emph{$t$-K-polystable} if it is $t$-K-semistable and every normal $\F$-compatible test configuration with $\DF^{[t]}(\X,\F_\X,\L)=0$ is of product type.
\end{enumerate}
\end{definition}
The non-Archimedean $J$-functional of the test configuration $J^{\NA}(\X,\L)$ is defined as follows: if
\[
\xymatrix{
& \mathcal Z \ar[dl]_{\Theta} \ar[dr]^{\Pi} & \\
\bar{\X} \ar@{-->}[rr] && X\times \PP^1
}
\]
is a normal common resolution, then
\[
J^{\NA}(\X,\L) := \frac{1}{V} \left(\Theta^*\bar{\L}\cdot \Pi^*p_1^*L^n- \frac{1}{n+1}\bar{\L}^{n+1} \right).
\]
Note that we will interchangeably say that either $(X,\F,L)$ is $t$-K-semistable, or that the adjoint foliated structure $(X,\F,t)$ is $t$-K-semistable. When $(X,\F,t)$ is an adjoint Fano foliated structure, the implicit polarisation is always $L\sim_{\mathbb Q}-K^{[t]}_{X,\F}$.

We also recall the following result from \cite{Pap26b}.

\begin{theorem}[{\cite[Theorem 1.2(3)]{Pap26b}}]\label{thm: k-stab implies klt}
    Let $(X,\F,t)$ be a normal projective $\Q$-Gorenstein adjoint Fano foliated structure for $0<t<1$. If $(X,\F,t)$ is $t$-K-semistable for $L = -K_{X,\F}^{[t]}$, then $(X,\F,t)$ is klt. In particular, $X$ is potentially klt and of Fano type.
\end{theorem}

This theorem allows us to restrict to klt adjoint Fano foliated structures in birational geometric arguments below.

\subsection{Filtrations of section rings}
\label{subsec:filtrations-section-rings}

We recall the filtration conventions used throughout the paper, following \cite[\S3.4]{Xu2025}. Let $X$ be a normal projective variety and let $L$ be an ample $\Q$-Cartier divisor. We fix an integer $r>0$ such that $rL$ is Cartier, and write
\[
        R=R(X,L):=\bigoplus_{m\in r\N}R_m, \qquad R_m:=H^0(X,mL).
\]

\begin{definition}
\label{def:multiplicative-filtration}
A decreasing $\R$-filtration $\cG^\lambda R_m\subset R_m$ for $\lambda\in\R$ is a collection of subspaces satisfying $\cG^{\lambda'}R_m\subset \cG^\lambda R_m$ when $\lambda\le \lambda'$. We require it to be left-continuous:
\[
        \cG^\lambda  R_m=\bigcap_{\lambda'<\lambda}\cG^{\lambda'}R_m.
\]
It is called \emph{multiplicative} if $\cG^\lambda R_m\cdot \cG^\mu R_\ell \subset \cG^{\lambda+\mu}R_{m+\ell}$ for all $m,\ell\in r\N$ and all $\lambda,\mu\in\R$.

The filtration is called \emph{linearly bounded} if there exist constants $C_-<C_+$ such that, for all $m\in r\N$,
\[
        \cG^{mC_-}R_m=R_m,
        \qquad
        \cG^{mC_+}R_m=0.
\]
\end{definition}

\begin{definition}
\label{def:jumping-numbers-filtration}
Let $\cG^\bullet R$ be a linearly bounded multiplicative filtration. For $m\in r\N$, we set $N_m:=h^0(X,mL)$. The jumping numbers of $\cG^\bullet R_m$ are the real numbers $a_{m,1}\le \cdots \le a_{m,N_m}$ defined by
\[
        a_{m,j} :=\inf\left\{\lambda\in\R\ \middle|\dim \cG^\lambda R_m\le N_m-j\right\}.
\]
\end{definition}

\begin{definition}
\label{def:DH-measure-filtration}
For $m\in r\N$, we define the probability measure
\[
        \nu_m(\cG) := \frac{1}{N_m} \sum_{j=1}^{N_m} \delta_{a_{m,j}/m}
\]
on $\R$. The measures $\nu_m(\cG)$ converge weakly to a compactly supported probability measure $\nu_{\operatorname{DH},\cG}$, called the \emph{Duistermaat--Heckman measure} of $\cG$ (cf. \cite[\S3.4]{Xu2025}).
\end{definition}

\begin{definition}\label{def:S-J-filtration}
Let $\cG^\bullet R$ be linearly bounded and multiplicative. We define
\[
        S_L(\cG) := \int_{\R}\lambda\,d\nu_{\cG}(\lambda).
\]
\end{definition}
Equivalently,
\[
        S_L(\cG) = \lim_{m\to\infty} \frac{1}{mN_m}\sum_{j=1}^{N_m}a_{m,j}.
\]

\subsection{The stacky test curves \texorpdfstring{$\Theta_R$}{ThetaR} and \texorpdfstring{$ST_R$}{STR}}\label{subsec:stacky-test-curves}

We recall the stack-theoretic valuative criteria used later in the paper. We follow \cite[\S 8.1.2]{Xu2025} and \cite{ABHLX20}. Let $R$ be a DVR with fraction field $K$, residue field $\kappa$, and uniformiser $\pi$. We define
\[
        \Theta_R:=[\mathbb A^1_R/\mathbb G_m] = [\operatorname{Spec}R[s]/\mathbb G_m],
\]
where $s$ has weight $-1$, and let
\[
        \Theta_R^\circ:=\Theta_R\setminus [0_\kappa/\mathbb G_m].
\]
An Artin stack $\X$ is called \emph{$\Theta$-reductive} if every
morphism $\Theta_R^\circ\to \X$ extends uniquely to a morphism $\Theta_R\to \X$.

We will use the following elementary description of coherent sheaves on $\Theta_R$. Let $j:\Theta_R^\circ\hookrightarrow \Theta_R$ be the open immersion. If $\E$ is a flat coherent sheaf on $\Theta_R^\circ$, then $\E$ corresponds to a finite $R$-module $E$ together with a $\mathbb Z$-filtration $\cG^\bullet E_K$ of $E_K:=E\otimes_RK$. Under this identification,
\[
        j_*\E \simeq \bigoplus_{p\in\mathbb Z} \bigl(E\cap \cG^pE_K\bigr)s^{-p}.
\]
Thus the extension over $\Theta_R$ is governed by the $R$-filtration $\cG^pE:=E\cap \cG^pE_K$.

We also recall the stacky test curve used in the valuative criterion for $S$-completeness. Let
\[
        \ST_R:= [\operatorname{Spec}R[s,t]/(st-\pi)/\mathbb G_m],
\]
where $\lambda\in \mathbb G_m$ acts by $\lambda\cdot(s,t)=(\lambda s,\lambda^{-1}t)$. Let $\ST_R^\circ:=\ST_R\setminus [(0,0)/\mathbb G_m]$. Then
\[
        \ST_R^\circ \simeq \operatorname{Spec}R\cup_{\operatorname{Spec}K}\operatorname{Spec}R.
\]
An Artin stack $\X$ is called $S$-complete if every morphism $\ST_R^\circ\to \X$ extends uniquely to a morphism $\ST_R\to \X$.

Let $j:\ST_R^\circ\hookrightarrow \ST_R$ be the open immersion. If $\E$ is a flat coherent sheaf on $\ST_R^\circ$, corresponding to two finite $R$-modules $E,E'$ together with an identification $E_K\simeq E'_K$, then, after viewing $E$ and $E'$ as $R$-lattices inside $E_K$, we have
\[
        j_*\E\simeq \bigoplus_{p\in\mathbb Z} \bigl(E\cap \pi^pE'\bigr)t^{-p}.
\]
Equivalently, the relevant comparison filtration is $\cG^pE:=E\cap \pi^pE'$.

\section{Admissible families}\label{sec: admissible families}

In this section, we define the notion of admissible families, which plays a central role in this paper. We follow the divisorial-sheaf and hull conventions of \cite[\S 7.1.1]{Xu2025} and \cite[\S 3]{Kol23}. We first recall Koll{\'a}r's condition for reflexive rank-one sheaves in families. For $m\in\mathbb Z$, we set
\[
\mathscr M^{[m]} :=
\begin{cases}
(\mathscr M^{\otimes m})^{**}, & m\geq 0,\\
((\mathscr M^\vee)^{\otimes (-m)})^{**}, & m<0.
\end{cases}
\]

\begin{definition} \label{def:kollar-condition-reflexive-sheaf}
Let $f:\X\to S$ be a flat morphism of locally Noetherian schemes, with $\X$ satisfying Serre's condition $S_2$, and let $\mathscr M$ be a reflexive rank-one sheaf on $\X$. We say that $\mathscr M$ satisfies the Koll\'ar condition over $S$ if, for every $m\in\mathbb Z$, the reflexive power $\mathscr M^{[m]}$ is flat over $S$ and commutes with arbitrary base change. That is, for
every morphism $T\to S$, with $\X_T:=\X\times_S T,$ the natural map
\[
        \left.\mathscr M^{[m]}\right|_{\X_T}  \longrightarrow \left(\left.\mathscr M\right|_{\X_T}\right)^{[m]}
\]
is an isomorphism.
\end{definition}

\begin{definition}\label{def:fibrewise-family-adjoint-foliated}
Let $S$ be a locally Noetherian scheme, let $0<t<1$ be rational, and choose a sufficiently divisible integer $r>0$ such that $r(1-t),\,rt\in\mathbb Z$. A \emph{fibrewise family of adjoint Fano foliated structures} over $S$ consists of:
\begin{enumerate}
\item a flat projective morphism $f:\X\to S$ where $\X$ is $S_2$;

\item a coherent subsheaf $T_{\F/S}\subset T_{\X/S}$;

\item a coherent rank-one reflexive sheaf $\mathscr K_{r,t}$ on $\X$ satisfying Koll\'ar's condition over $S$;
\end{enumerate}
such that for every geometric point $\bar s\to S$:
\begin{enumerate}
\item[(a)] the fibre $\X_{\bar s}$ is normal;

\item[(b)] the saturation of the image of $T_{\F/S}|_{\X_{\bar s}}\longrightarrow T_{\X_{\bar s}}$ defines an algebraically integrable foliation $\F_{\bar s}$ on $\X_{\bar s}$;

\item[(c)] there is an isomorphism
\[
        \mathscr K_{r,t}|_{\X_{\bar s}} \simeq \mathcal O_{\X_{\bar s}}  \left(  rK^{[t]}_{\X_{\bar s},\F_{\bar s}}\right);
\]

\item[(d)] the dual $\mathcal L_r:=\mathscr K_{r,t}^{\vee}$ is an $f$-ample line bundle.
\end{enumerate}
We write $\mathcal L:=\frac1r\mathcal L_r$ for the corresponding relatively ample $\mathbb Q$-line bundle. Thus, on every geometric fibre, $\mathcal L_{\bar s}\sim_{\mathbb Q}-K^{[t]}_{\X_{\bar s},\F_{\bar s}}$. If every geometric fibre $(\X_{\bar s},\F_{\bar s},t)$ is $t$-K-semistable, we call this a \emph{fibrewise $t$-K-semistable family}.
\end{definition}

The preceding notion is intended to provide a general framework for families of adjoint foliated structures and to formulate the moduli problem of $t$-K-semistable adjoint Fano foliated structures. For the valuative criteria of $\Theta$-reductivity and $S$-completeness, however, we work with the more restrictive notion below.

We will first fix some notation. Let $j\colon \X^{\mathrm{sm}}\hookrightarrow\X$ denote the relative smooth locus. For $m\in\mathbb Z$, we set $\omega_{\X/S}^{[m]} := j_*\omega_{\X^{\mathrm{sm}}/S}^{\otimes m}$. If $T_{\mathcal F/S}\subset T_{\X/S}$ is a saturated relative foliation, we set $\omega_{\mathcal F/S} :=\bigl(\det T_{\mathcal F/S}\bigr)^\vee$. For rank-one reflexive sheaves $\mathcal A$ and $\mathcal B$, we write $\mathcal A\widehat\otimes\mathcal B:=(\mathcal A\otimes\mathcal B)^{**}$.

\begin{definition}
\label{def:admissible-family-adjoint-foliated}
An \emph{admissible family of $t$-adjoint Fano foliated structures} over $S$ is a fibrewise family $f:(\X,\F,t)\to S$ such that:
\begin{enumerate}
\item the total space $\X$ is normal;

\item the subsheaf $ T_{\F/S}\subset T_{\X/S}$ is saturated and closed under the Lie bracket, and hence defines a relative foliation $\F$ on $\X$;
\item the relative foliation $\F$ is algebraically integrable;

\item the natural mixed reflexive sheaf on the total space satisfies
\[
\mathcal K_{r,t} \simeq \omega_{\mathcal X/S}^{[r(1-t)]} \widehat\otimes \omega_{\mathcal F/S}^{[rt]}.
\]
\end{enumerate}
\end{definition}

The last item along with Definition \ref{def:fibrewise-family-adjoint-foliated}(3) specifies that $K^{[t]}_{\X/S,\F} := (1-t)K_{\X/S}+tK_{\F/S}$ is $\mathbb Q$-Cartier, its reflexive multiples satisfy Koll\'ar's condition, and $\mathcal L \sim_{\mathbb Q,S} -K^{[t]}_{\X/S,\F}$ is $f$-ample. When the base is $\Theta_R$ or $\operatorname{ST}_R$, we say that a family is admissible if its pullback to the standard smooth atlas is an admissible family in the above sense and all the relevant data are $\mathbb G_m$-equivariant. We next refine this definition for families over the spectrum of a DVR.

\begin{definition}
\label{def:admissible-DVR-family-kollar}
Let $R$ be a DVR. An admissible DVR family $(\X,\F,t)\to \Spec R$ is an admissible family such that, for some sufficiently divisible $r>0$ and every $m\geq 0$, the $R$-module $f_*\mathcal L_r^{\otimes m}$ is finite free and the natural maps
\[
f_*\mathcal L_r^{\otimes m}\otimes_R k(p) \longrightarrow H^0\!\left(\mathcal X_p, \mathcal L_{r,p}^{\otimes m}\right)
\]
are isomorphisms for $p=\eta,\kappa$.
\end{definition}

We record some consequences of the above definitions.

\begin{lemma}[{\cite[Tag 0C22]{StacksProject}}]\label{lem:normal-total-space}
Let $R$ be a DVR and let $f\colon\mathcal X\to\Spec R$ be flat and of finite type. If its generic and special fibres are normal, then $\mathcal X$ is normal.
\end{lemma}

\begin{remark}\label{rem:relative-algebraic-integrability}
Let $\mathcal X\to\Spec R$ be integral, and let $T_{\mathcal F/R}\subseteq T_{\mathcal X/R}$ be a relative foliation. Since $K(\mathcal X)=K(\mathcal X_K)$, the rational distribution of $\mathcal F$ is determined by its generic fibre. Algebraic integrability is invariant under extension of the ground field; see \cite[\S 2.2.2]{Bos01} and \cite[Definition 3.2.3 and Lemma 3.2.4]{Bon21}. Thus algebraic integrability of the geometric generic fibre implies algebraic integrability of $\mathcal F$.
\end{remark}

\section{Filtrations and mixed Ding semistability}\label{sec: F-compatible filtrations}

Let $(X,\F,t)$ be an adjoint Fano foliated structure, and let $L:=-K^{[t]}_{X,\F}$. Thus $L$ is an ample $\Q$-Cartier divisor. We choose a sufficiently divisible integer $r_0>0$ so that $r_0L$ is very ample and the Veronese section ring below is generated by $R_{r_0}$, and we work with $R$ and $R_m$, as before.If we work with the Cartier divisor $L_{r_0}:=r_0L$ and write $D\sim_{\mathbb Q} mL_{r_0}=mr_0L$, then the coefficient $\frac{1}{m}D$ in the Cartier polarisation corresponds to $\frac{1}{mr_0}D$ in the polarisation $L=-K^{[t]}_{X,\F}$. We use the $\Q$-Cartier polarisation $L=-K^{[t]}_{X,\F}$ throughout, in order to avoid this extra factor.

For a real-valued filtration $\cG$, we denote by $\cG_{\mathbb Z}$ its associated integer-valued filtration
\[
        \cG_{\mathbb Z}^{\lambda}R_m :=\cG^{\lceil\lambda\rceil}R_m
\]
as in \cite[Definition 3.16]{Xu2025}.

\subsection{Base ideals and flag ideals}

\begin{definition}\label{def: ass trivial ass zero}
Let $\mathfrak a_\bullet=\{\mathfrak a_m\}_{m\in r_0\N}$ be a graded sequence of ideals on $X$. We say that $\mathfrak a_\bullet$ is \emph{asymptotically trivial} if $\mathfrak a_m=\cO_X$ for all sufficiently divisible $m$. We say that $\mathfrak a_\bullet$ is \emph{asymptotically zero} if $\mathfrak a_m=0$ for all sufficiently divisible $m$.
\end{definition}

For a graded sequence of ideals $\mathfrak a_\bullet=\{\mathfrak a_m\}_{m\in r_0\mathbb N}$, we set
\[
v(\mathfrak a_\bullet) := \lim_{m\to\infty}\frac{v(\mathfrak a_m)}{m}
\]
and
\[
\operatorname{lct}^{[t]} (X,\mathcal F;\mathfrak a_\bullet) := \inf_v \frac{A^{[t]}_{X,\mathcal F}(v)}{v(\mathfrak a_\bullet)}.
\]

\begin{definition}
Let $\cG^\bullet R$ be a linearly bounded multiplicative
filtration. For $m\in r_0\N$ and $\lambda\in \R$, we define the base ideal
\[
        I_{m,\lambda}(\cG) := \operatorname{Im} \left(\cG^\lambda R_m\otimes \cO_X(-mL) \longrightarrow \cO_X\right).
\]
Equivalently, $I_{m,\lambda}(\cG)$ is the ideal generated by sections in $\cG^\lambda R_m$.
\end{definition}

\begin{definition}
We choose integers $e_-<e_+$, outside the linear bounds of $\mathcal G$, such that $\cG^{me_-}R_m=R_m$, $\cG^{me_+}R_m=0$ for all sufficiently divisible $m\in r_0\N$. For such $m$, we define the \emph{flag ideal}
\[
\begin{aligned}
        \mathfrak I_m(\cG)&:={} I_{m,me_+}(\cG)+I_{m,me_+-1}(\cG)\tau +\cdots \\
        &\quad +I_{m,me_-+1}(\cG)\tau^{m(e_+-e_-)-1} +(\tau^{m(e_+-e_-)}) \subset \cO_{X\times \A^1},
\end{aligned}
\]
where $\tau$ is the coordinate on $\A^1$. Replacing $\mathfrak I_m(\cG)$ by its integral closure does not change its divisorial orders and hence does not change any mixed log canonical threshold below. It also does not change the normalised blow-up, since their Rees algebras have the same integral closure.
\end{definition}

Multiplicativity of $G$ gives $\mathfrak I_m(G)\cdot \mathfrak I_\ell(G) \subseteq \mathfrak I_{m+\ell}(G)$ so $\{\mathfrak I_m(G)\}_{m\in r_0\mathbb N}$ is a graded sequence of ideals on $X\times\mathbb A^1$.

\subsection{Finitely generated approximants}

Fix an integer $e_-$ below the lower linear bound of $\cG$. For $q\in r_0\N$, let $\cG^{(q)}$ be the $q$-th minimal approximating filtration of $\cG_{\mathbb Z}$ in the sense of \cite[Definition--Lemma 3.56]{Xu2025}. Explicitly, for $d\in r_0\N$:
\begin{enumerate}
\item if $d<q$, then $\cG^{(q),\lambda}R_d=R_d$ for $\lambda\le e_-d$ and is zero for $\lambda>e_-d$;
\item if $d=q$, then $\cG^{(q),\lambda}R_q=\cG_{\mathbb Z}^{\lambda}R_q$;
\item if $d>q$, then
\[
\cG^{(q),\lambda}R_d =\sum_{\substack{s\ge1,\;sq\le d\\
\mu_1,\ldots,\mu_s\in\mathbb Z\\
\mu_1+\cdots+\mu_s\ge \lambda-e_-(d-sq)}}\cG_{\mathbb Z}^{\mu_1}R_q\cdots \cG_{\mathbb Z}^{\mu_s}R_q\cdot R_{d-sq}.
\]
\end{enumerate}
Then $\{\cG^{(q)}\}_{q\in r_0\N}$ is an approximating sequence in the sense of \cite[Definition 3.55]{Xu2025}. Since $R$ is generated by $R_{r_0}$ and the filtration is linearly bounded, each $\cG^{(q)}$ is finitely generated.

\subsection{Flag ideals and foliated test configurations}

We first recall the following result from \cite{Pap26}.

\begin{lemma}[{\cite[Lemma 7.11]{Pap26}}]\label{lem:compatible-flag-blowup}
Let $\mathfrak I\subset \cO_{X\times\A^1}$ be any $\G_m$-invariant flag ideal supported on $X\times\{0\}$. Let
\[
        \Pi:\B:=\operatorname{Bl}_{\mathfrak I}(X\times\A^1)
        \to X\times\A^1
\]
be the blow-up, and let $\nu:\X:=\B^\nu\to \B$ be the normalisation. Then $\X$ carries a natural saturated integrable relative foliation $\F_{\X}\subset T_{\X/\A^1}$ which restricts to the product foliation on
\[
        \X|_{\A^1\setminus\{0\}}
        \simeq
        X\times(\A^1\setminus\{0\}).
\]
If $\F$ is algebraically integrable, then $\F_{\X}$ is algebraically integrable.
\end{lemma}

\begin{proposition}
\label{prop:approximants-TC}
Let $\cG^\bullet R$ be a linearly bounded multiplicative filtration. For every sufficiently divisible $q\in r_0\N$, the minimal approximant $\cG^{(q)}$ determines a normal semiample $\F$-compatible test configuration $(\X_q,\F_q,\L_q)\to \A^1$ of $(X,\F,L)$.
\end{proposition}

\begin{proof}
The approximant is integer-valued, and \cite[Lemma 3.59]{Xu2025} gives $\mathfrak I_{q\ell}(\cG^{(q)}) =\mathfrak  I_q(\cG^{(q)})^\ell$ for all $\ell\ge1$. Let $\Y_q\to X\times\A^1$ be the normalised blow-up of $\mathfrak I_q(\cG^{(q)})$. The pullback of $q p_1^*L$ twisted by the exceptional Cartier divisor is relatively base-point-free by the argument of \cite[Lemma 3.61]{Xu2025}; its semiample model is the normal test configuration $(\X_q,\L_q)$ of \cite[Definition 3.62]{Xu2025}. The $\Q$-polarisation is divided by $q$, so that its general fibre is $(X,L)$.

Lemma \ref{lem:compatible-flag-blowup} gives the saturated algebraically integrable transform on $\Y_q$; its pushforward and saturation on the semiample model gives $\F_q\subset T_{\X_q/\A^1}$. By \cite[Proposition 7.12]{Pap26}, $(\X_q,\F_q,\L_q)\to\A^1$ is a normal semiample $\F$-compatible test configuration.
\end{proof}

\begin{remark}\label{rem: ding in terms of semiample}
Although the definitions of $t$-K-stability and $t$-Ding stability in \cite{Pap26} are stated for ample test configurations, the preceding semiample test configurations may be replaced by their ample models. The mixed Ding functional and the non-Archimedean $J$-functional are unchanged under this passage. Thus testing semiample $\F$-compatible test configurations is equivalent to testing ample ones; see \cite[Proposition 7.5 and Remark 7.6]{Pap26}.
\end{remark}

\subsection{Mixed log canonical slopes of filtrations}

We keep the notation introduced at the beginning of the section. Thus $\cG^\bullet R$ is a linearly bounded multiplicative filtration of $R=R(X,L)$. For $a\in\R$ we use the notation $I^{(a)}_{m,\cG}:=I_{m,ma}(\cG)$ and $I^{(a)}_{\bullet,\cG}:=\{I^{(a)}_{m,\cG}\}_{m}$. Throughout this subsection, we assume that $(X,\F,t)$ is klt. This will be automatic in the applications below, since $t$-K-semistability implies klt by Theorem \ref{thm: k-stab implies klt}. Recall the following definitions.

\begin{definition}
\label{def:lambda-max-filtration}
Let $L$ be an ample $\Q$-Cartier divisor on a normal projective variety $X$, and let $\cG^\bullet R$ be a linearly bounded multiplicative decreasing filtration.

For $m\in r_0\N$, we define
\[
        \lambda_{\max,m}(\cG) := \sup\left\{ \lambda\in\R  \ \middle|\ \cG^\lambda R_m\neq 0 \right\},
\]
The asymptotic maximal weight is
\[
        \lambda_{\max}(\cG)  := \lim_{m\to\infty}\frac{\lambda_{\max,m}(\cG)}{m},
\]
where the limit is taken over sufficiently divisible $m\in r_0\N$.
\end{definition}

The limit exists because multiplicativity gives
\[
\lambda_{\max,m+\ell}(\cG)\ge \lambda_{\max,m}(\cG)+\lambda_{\max,\ell}(\cG),
\]
and hence Fekete's lemma applies.

Moreover,
\[
        \lambda_{\max}(\cG)  = \inf \left\{a\in\R\ \middle|\I^{(a)}_{\bullet,\cG}\text{ is asymptotically zero}\right\}.
\]

\begin{lemma}\label{lem:mixed-Xu-344}
For every linearly bounded multiplicative filtration $\cG^\bullet R$, the function
\[
        a\longmapsto \lct^{[t]} \left(X,\F; I^{(a)}_{\bullet,\cG}\right)
\]
is non-increasing and continuous, with values in $(0,+\infty]$, on $(-\infty,\lambda_{\max}(\cG))$. If
\[
\mu_{+\infty}^{[t]}(\cG):=\sup\left\{a<\lambda_{\max}(\cG)\ \middle|\ \lct^{[t]}(X,\F;I^{(a)}_{\bullet,\cG})=+\infty \right\},
\]
then the function is strictly decreasing on $[\mu_{+\infty}^{[t]}(\cG),\lambda_{\max}(\cG))$.
\end{lemma}

\begin{proof}
We follow the proof of \cite[Proposition 3.44]{Xu2025}, replacing the usual log discrepancy by the mixed log discrepancy.

First, if $a_1\ge a_0$, then $\cG^{ma_1}R_m\subset \cG^{ma_0}R_m$, hence $I^{(a_1)}_{m,\cG}\subset I^{(a_0)}_{m,\cG}$. Therefore
\[
\lct^{[t]}\left(X,\F;I^{(a_1)}_{\bullet,\cG}\right)\le\lct^{[t]}\left(X,\F;I^{(a_0)}_{\bullet,\cG}\right),
\]
proving monotonicity.

For continuity, we use the valuation formula
\[
        \lct^{[t]} \left(X,\F;I^{(a)}_{\bullet,\cG}\right)=\inf_v\frac{A^{[t]}_{X,\F}(v)}        {v\!\left(I^{(a)}_{\bullet,\cG}\right)}.
\]
Equivalently,
\[
        \frac{1}{\lct^{[t]}\left(X,\F;I^{(a)}_{\bullet,\cG}\right)} =  \sup_{A^{[t]}_{X,\F}(v)=1} v\!\left(I^{(a)}_{\bullet,\cG}\right).
\]

We will now show that for every valuation $v$, the function
\[
        a\longmapsto v\!\left(I^{(a)}_{\bullet,\cG}\right)
\]
is convex. To see this, take $a=\theta a_0+(1-\theta)a_1$ with $\theta\in\Q\cap[0,1]$. Then, after taking $m$ sufficiently divisible, multiplicativity gives
\[
I^{(a_0)}_{\theta m,\cG} \cdot I^{(a_1)}_{(1-\theta)m,\cG} \subset I^{(a)}_{m,\cG}.
\]
Applying $v$, dividing by $m$, and passing to the limit gives
\[
        v\!\left(I^{(a)}_{\bullet,\cG}\right)\le\theta v\!\left(I^{(a_0)}_{\bullet,\cG}\right) + (1-\theta) v\!\left(I^{(a_1)}_{\bullet,\cG}\right).
\]
The case of real $\theta$ follows from the one-sided approximation argument of \cite[Lemma 3.43]{Xu2025}. Hence the reciprocal of the mixed lct function is a supremum of convex functions, and is therefore convex. Since convex functions are continuous on the interior of their domain, this proves continuity.

The proof that the function is strictly decreasing follows directly from the above argument and from the relevant argument in \cite[Proof of Proposition 3.44]{Xu2025}.
\end{proof}

\begin{definition}
\label{def:mixed-delta-lc-slope}
Let $0<\delta\le 1$. The mixed $\delta$-log canonical slope of
$\cG$ is
\[
        \mu^{[t]}(\cG,\delta) := \sup \left\{a\in\R \ \middle|\ \lct^{[t]} \left(X,\F; I^{(a)}_{\bullet,\cG} \right)  \ge  \delta \right\}.
\]
When $\delta=1$, we write
\[
        \mu^{[t]}(\cG) := \mu^{[t]}(\cG,1)
\]
and call it the mixed log canonical slope of $\cG$.
\end{definition}

\begin{lemma}
\label{lem:mixed-Xu-346}
Let $\cG^\bullet R$ be a linearly bounded multiplicative
filtration. Then the function
\[
        \delta\longmapsto \mu^{[t]}(\cG,\delta)
\]
is continuous on $(0,1]$. 
\end{lemma}

\begin{proof}
By Lemma \ref{lem:mixed-Xu-344}, the function
\[
        a\longmapsto \lct^{[t]} \left( X,\F;  I^{(a)}_{\bullet,\cG} \right)
\]
is non-increasing and continuous on $(-\infty,\lambda_{\max}(\cG))$, and strictly decreasing after $\mu_{+\infty}^{[t]}(\cG)$. Hence for $c_{\cG}^{[t]}<\delta\le 1$ there is a unique value of $a$ with threshold equal to $\delta$, and this value is precisely $\mu^{[t]}(\cG,\delta)$. Let
\[c_{\cG}^{[t]}: = \lim_{a\to \lambda_{\max}(G)^-} \lct^{[t]}(X,\F,\mathfrak J_{\bullet,G}^{(a)}.\]
For $0<\delta\le c_{\cG}^{[t]}$, the inequality holds for all $a<\lambda_{\max}(\cG)$, so the supremum is $\lambda_{\max}(\cG)$. Continuity in $\delta$ follows from continuity and strict monotonicity.
\end{proof}

\begin{lemma}
\label{lem:finite-level-mixed-lct-approx}
Let $\mathfrak a_\bullet=\{\mathfrak a_m\}_{m\in r\N}$ be a graded
sequence of ideals. Then
\[
        \lct^{[t]} \left(X,\F;\mathfrak a_\bullet\right)  = \lim_{m\to\infty} m\lct^{[t]} \left(X,\F;\mathfrak a_m\right) = \sup_{m\in r\N} m\lct^{[t]} \left(X,\F;\mathfrak a_m\right).
\]
% Equivalently, for every $\eta>0$, all sufficiently large divisible $m$ satisfy
% \[
%         m\lct^{[t]} \left(X,\F;\mathfrak a_m\right) \ge \lct^{[t]} \left(X,\F;\mathfrak a_\bullet\right)-\eta.
% \]
\end{lemma}

\begin{proof}
By the valuation formula,
\[
        \lct^{[t]} \left(X,\F;\mathfrak a_\bullet\right) = \inf_v \frac{A^{[t]}_{X,\F}(v)} {v(\mathfrak a_\bullet)},
\]
where
\[
        v(\mathfrak a_\bullet) :=  \lim_{m\to\infty} \frac{1}{m}v(\mathfrak a_m),
\]
and the limit exists by Fekete's Lemma. For any $p\in \N$, we define $\mathfrak b_p = \cup_{m\geq p}\mathfrak a$. For any valuation $v$
\[v(\mathfrak b_p)+v(\mathfrak b_{p'})\geq v(\mathfrak b_{p+p'}),\]
so, if we set $a_p:= \lct^{[t]}(X,\F;\mathfrak b_p)$ we have $\frac{1}{a_p}+\frac{1}{a_{p'}}\geq \frac{1}{a_{p+p'}}$, i.e. $\{\frac{1}{a_p}\}_{p\in \N}$ is subadditive. Hence, Fekete's lemma applies and $\lim_{p\to \infty}\frac{1}{pa_p}$ exists, which implies that $\lim_{p\to \infty }p\cdot \lct^{[t]}(X,\F;\mathfrak b_p)$ exists and is equal to $\sup_pp\cdot \lct^{[t]}(X,\F;\mathfrak b_p)$. Since for any $m\in [p, p+1)$ we have $\mathfrak b_{p+1}\subseteq \mathfrak a_m \subseteq \mathfrak b_p $ we have 
\[\lct^{[t]}(X,\F;\mathfrak b_p)\leq \lct^{[t]}(X,\F;\mathfrak a_m)\leq \lct^{[t]}(X,\F;\mathfrak b_{p+1})\]
thus $\lim_{m\to \infty} m\cdot \lct^{[t]}(X,\F;\mathfrak a_m)$ exists and is equal to $\sup_mm\cdot \lct^{[t]}(X,\F;\mathfrak a_m)$, as required. 
\end{proof}

\subsection{Mixed Ding functionals of filtrations}

We first recall how test configurations enter the picture. By \cite[Definitions 7.1 and 7.2]{Pap26} The mixed Ding invariant of a normal $\F$-compatible test configuration $(\X,\F_{\X},\L)\to \A^1$ is defined birationally by using the vertical correction divisor
\[
        D^{[t]}_{\X,\F_{\X},\L}  \sim_{\Q} -\L - K^{[t]}_{\X/\A^1,\F_{\X}},
\]
supported on the central fibre, and the mixed log canonical threshold of $(\X,\F_{\X},D^{[t]}_{\X,\F_{\X},\L};\X_0,t)$. For a vertical $\mathbb Q$-divisor $D$ on a test configuration, we use the notation $\lct^{[t]}(\X,\F_\X;D;\X_0)$ for
\[
        \sup\left\{ c\in\mathbb R \ \middle|\ (\X,\F_\X,D+c\X_0,t)\text{ is sub-lc} \right\}.
\]
Here, all  lcts are computed on the natural compactification $(\overline{\X},\overline{\F},\overline{\L})\to \mathbb P^1$.

We have
\[
\Ding^{[t]}(\X,\F_{\X},\L) := -\frac{\overline{\L}^{\,n+1}}{(n+1)V} -(1-t) + \lct^{[t]} \!\left( \overline{\X},\overline{\F}; D^{[t]}_{\X,\F_{\X},\L}; \overline{\X}_0 \right).
\]
By \cite[Corollary 7.10]{Pap26}, $t$-K-semistability is equivalent to $t$-Ding-semistability for normal $\F$-compatible test configurations. In this section we compare this test-configuration Ding invariant with the filtration invariant below.

Let $S_L(\cG)$ denote the expected vanishing number of the Duistermaat--Heckman measure of $\cG$ as in Definition \ref{def:S-J-filtration}, and let
\[
        J(\cG) :=  \lambda_{\max}(\cG)-S_L(\cG)
\]
be the $J$-invariant of the filtration, which agrees with the usual non-Archimedean $J$-functional for finitely generated filtrations. The \emph{mixed Ding invariant} of $\cG$ is
\[
        \Ding^{[t]}(\cG)  :=  \mu^{[t]}(\cG)-S_L(\cG).
\]
Theorem \ref{thm:mixed-ding-compatible-filtrations} below shows that $t$-K-semistability implies the non-negativity of this invariant on all linearly bounded multiplicative filtrations.

\begin{definition}\label{def:mixed-cm-cinfty}
Let $\cG^\bullet R$ be a linearly bounded multiplicative filtration. Choose integers $e_-<e_+$ such that $ \cG^{me_-}R_m=R_m$ and $\cG^{me_+}R_m=0$ for all sufficiently divisible $m$. Let $X_0:=X\times\{0\}$. We define
\[
        c_m^{[t]}(\cG,e_+)  := \lct^{[t]}  \left(X\times\A^1,\F\times\A^1; \mathfrak I_m(\cG)^{1/m}; X_0\right).
\]
The limit
\[
        c_\infty^{[t]}(\cG,e_+)  := \lim_{m\to\infty}c_m^{[t]}(\cG,e_+)
\]
exists, and we set
\[
        \mathbf{L}^{[t]}(\cG) := c_\infty^{[t]}(\cG,e_+)+e_+-(1-t).
\]
\end{definition}

The shift by $(1-t)$ in the definition above reflects the fact that the product foliation is relative over $\A^1$, so the divisor $X_0$ contributes only to the ordinary canonical part of the mixed discrepancy. The existence of the limit above follows from an argument identical to \cite[Lemma 1.50]{Xu2025}, which we omit.

\begin{lemma}\label{lem:equivariant-mixed-lct}
Let $(Y,\mathcal G,t)$ be an algebraically integrable adjoint foliated structure endowed with a $\mathbb G_m$-action preserving $\mathcal G$. Let $\mathfrak a\subset\mathcal O_Y$ be a $\mathbb G_m$-invariant nonzero ideal, let $H$ be a $\mathbb G_m$-invariant effective Cartier divisor, and let $\lambda>0$. Assume that $c:= \operatorname{lct}^{[t]} \bigl(Y,\mathcal G;\mathfrak a^\lambda;H\bigr)$ is finite. Then $c$ is computed by a $\mathbb G_m$-invariant divisorial valuation.
\end{lemma}
\begin{proof}
We choose a $\mathbb G_m$-equivariant foliated log resolution
\[
\pi\colon (W,\mathcal G_W)\longrightarrow (Y,\mathcal G)
\]
which principalises $\mathfrak a$ and such that the union of the exceptional locus, the support of $\mathfrak a\mathcal O_W$, and the total transform of $H$ is contained in the toroidal boundary of the foliated log-smooth structure. We set $\mathfrak a\mathcal O_W=\mathcal O_W(-A)$ and let $E_1,\ldots,E_N$ be the irreducible components of this boundary. Let $a_i:=\ord_{E_i}(\mathfrak a)$ and $b_i:=\ord_{E_i}(H)$. 

Since $(W,\mathcal G_W)$ is foliated log smooth, the adjoint
foliated structure $(Y,\mathcal G,\mathfrak a^\lambda+sH,t)$ is sub-lc if and only if $A^{[t]}_{Y,\mathcal G}(E_i)-\lambda a_i-sb_i\geq 0$ for every $i$. Thus,
\[
c=\min_{\substack{1\leq i\leq N\\ b_i>0}}\frac{A^{[t]}_{Y,\mathcal G}(E_i)-\lambda a_i}{b_i}.
\]
Hence, $c$ is computed by one of the divisorial valuations
$\ord_{E_i}$.

Finally, $\mathbb G_m$ permutes the finite set $\{E_1,\ldots,E_N\}$. Since $\mathbb G_m$ is connected, its action on this finite set is trivial. Hence every $E_i$, and in particular a divisor computing $c$, is $\mathbb G_m$-invariant.
\end{proof}

\begin{lemma} \label{lem:slope-flag-ideal-comparison}
Let $(X,\F,t)$ be klt. With the above normalisation of the flag ideals,
\[
        \mu^{[t]}(\cG)=\mathbf{L}^{[t]}(\cG) = c_\infty^{[t]}(\cG,e_+)+e_+-(1-t).
\]
In particular, the right-hand side is independent of the choice of $e_+$.
\end{lemma}

\begin{proof}
We first record the product discrepancy computation used below. Let $v$ be a divisorial valuation over $X$, let $\xi>0$, and let $\widetilde v$ be the extension to $X\times\A^1$ defined by
\[
        \widetilde v\left(\sum_i f_i\tau^i\right) :=
        \min_i\{v(f_i)+i\xi\}.
\]
Note that it is enough to take $\xi\in\Q_{>0}$, since the left derivatives used below may be approximated by rational slopes; alternatively we may use the valuative formula for lcts over quasi-monomial valuations. We use the rational approximation formulation below. Then $\widetilde v(X_0)=\xi$. Recall that by the local calculation in \cite[Proposition 4.3 and Theorem 7.13, Claim 3]{Pap26} we have 
\begin{equation}\label{eq: mixed ld extension}
    A^{[t]}_{X\times\A^1,\F\times\A^1}(\widetilde v) =  A^{[t]}_{X,\F}(v)+(1-t)\xi
\end{equation}

We first show that $\mathbf{L}^{[t]}(\cG)\le \mu^{[t]}(\cG)$. Let $ \mu:=\mu^{[t]}(\cG)$. If $\mu=\lambda_{\max}(\cG)$, then the inequality follows from the same elementary estimate as in \cite[Equation (3.30)]{Xu2025}, with $1$ replaced by $1-t$. Thus we may assume that $\mu<\lambda_{\max}(\cG)$.

By Lemma \ref{lem:mixed-Xu-346}, $\lct^{[t]} \left(X,\F;I^{(\mu)}_{\bullet,\cG}\right)=1$. Hence, for every $0<\alpha<1$, there exists a divisorial valuation $v$ over $X$ such that
\[
        0<\alpha A^{[t]}_{X,\F}(v) \le  v\!\left(I^{(\mu)}_{\bullet,\cG}\right).
\]
We now define $f_v(a):=v\!\left(I^{(a)}_{\bullet,\cG}\right)$. By the proof of Lemma \ref{lem:mixed-Xu-344}, $f_v$ is convex and non-decreasing on $(-\infty,\lambda_{\max}(\cG))$. Let $\xi>0$ be a left derivative of $f_v$ at $a=\mu$. Then for all $a<\lambda_{\max}(\cG)$,
\[
        f_v(a) \ge f_v(\mu)+\xi(a-\mu) \ge \alpha A^{[t]}_{X,\F}(v)+\xi(a-\mu).
\]

For $I_{m,ma}\tau^{m(e_+-a)} \subset \mathfrak I_m(\cG)$, $\widetilde v$ we have
\[
\begin{aligned}
\frac{1}{m}\widetilde v \left(I_{m,ma}\tau^{m(e_+-a)}\right)&= \frac{1}{m}v(I_{m,ma}) +\xi(e_+-a)        \\
&\ge f_v(a)+\xi(e_+-a) \\
&\ge \alpha A^{[t]}_{X,\F}(v)+\xi(e_+-\mu).
\end{aligned}
\]
Taking the minimum over all summands gives
\[
        \frac{1}{m}\widetilde v(\mathfrak I_m(\cG))\ge\alpha A^{[t]}_{X,\F}(v)+\xi(e_+-\mu).
\]
Therefore, by the valuative formula for the threshold of $X_0$,
\[
\begin{aligned}
        c_m^{[t]}(\cG,e_+)&\le \frac{A^{[t]}_{X\times\A^1,\F\times\A^1}(\widetilde v) -\frac{1}{m}\widetilde v(\mathfrak I_m(\cG))}{\widetilde v(X_0)}  \\
        &\le\mu-e_++(1-t)+\frac{(1-\alpha)A^{[t]}_{X,\F}(v)}{\xi}.
\end{aligned}
\]
Since $f_v(e_-)=0$, while $f_v(\mu)\ge\alpha A^{[t]}_{X,\F}(v)$, convexity  gives
\[
        \xi(\mu-e_-)  \ge f_v(\mu)  \ge \alpha A^{[t]}_{X,\F}(v),
\]
and hence
\[
\frac{(1-\alpha)A^{[t]}_{X,\F}(v)}{\xi} \le \frac{1-\alpha}{\alpha}(\mu-e_-) \longrightarrow0.
\]
Passing first to the limit in $m$ and then letting $\alpha\to 1$, we get
\[
        c_\infty^{[t]}(\cG,e_+)  \le   \mu-e_++(1-t),
\]
which equivalently shows that $\mathbf{L}^{[t]}(\cG)\le \mu^{[t]}(\cG)$.

We now prove the reverse inequality $\mu^{[t]}(\cG)\le \mathbf{L}^{[t]}(\cG)$. If $\mathbf{L}^{[t]}(\cG)=\lambda_{\max}(\cG)$ there is nothing to prove. Thus, we assume that $\mathbf{L}^{[t]}(\cG)<\lambda_{\max}(\cG)$. 

Since both $\mathfrak I_m(\cG)$ and $X_0$ are $\G_m$-invariant, we compute the threshold on a $\G_m$-equivariant foliated log resolution which principalises $\mathfrak I_m(\cG)$. By Lemma \ref{lem:equivariant-mixed-lct}, the mixed log canonical threshold is the minimum of finitely many ratios attached to divisors on this resolution. Since $\G_m$ is connected, it fixes each divisor in this finite set. Thus, for each sufficiently divisible $m$, we may choose a $\G_m$-invariant divisorial valuation $\widetilde w_m$ computing $c_m^{[t]}(\cG,e_+)$, normalised by $\widetilde w_m(X_0)=1$.

By \cite[Lemma 1.33]{Xu2025}, $\widetilde w_m$ is the weight-one  extension of its restriction $w_m:=\widetilde w_m|_{K(X)}$. Hence Equation \eqref{eq: mixed ld extension} gives
\[
A^{[t]}_{X\times\mathbb A^1,\F\times\mathbb A^1}(\widetilde w_m)= A^{[t]}_{X,\F}(w_m)+(1-t).
\]
By the definition of $c_m^{[t]}$,
\[
\begin{aligned}
        c_m^{[t]}(\cG,e_+)+e_+-(1-t)&={}A^{[t]}_{X,\F}(w_m) - \frac{1}{m}\widetilde w_m(\mathfrak I_m(\cG)) +e_+ .
\end{aligned}
\]
On the other hand,
\[
        \frac{1}{m}\widetilde w_m(\mathfrak I_m(\cG))  =  \min_i \left\{ \frac{1}{m}w_m(I_{m,me_+-i})+\frac{i}{m} \right\}.
\]
Thus, for every real number $a$,
\[
        \frac{1}{m}w_m(I_{m,\lceil ma\rceil})  \ge  A^{[t]}_{X,\F}(w_m) -  \bigl(c_m^{[t]}(\cG,e_+)+e_+-(1-t)\bigr) +a  +O(1/m).
\]
We next prove the uniform discrepancy estimate needed to pass to the limit. Set
\[
a_0:=\frac12\left(\lambda_{\max}(\cG)+\mathbf L^{[t]}(\cG)\right), \qquad \gamma:=\lambda_{\max}(\cG)-\mathbf L^{[t]}(\cG)>0.
\]
Since
\[
c_m^{[t]}(\cG,e_+)+e_+-(1-t) \longrightarrow \mathbf L^{[t]}(\cG),
\]
the preceding inequality with $a=a_0$ implies, for all large $m$,
\[
\frac1m w_m(I_{m,\lceil ma_0\rceil}) -A^{[t]}_{X,\F}(w_m) \ge \frac{\gamma}{4}.
\]
We now put
\[
c_0:=\lct^{[t]} \left(X,\F;I^{(a_0)}_{\bullet,\cG}\right)>0.
\]
By Lemma \ref{lem:finite-level-mixed-lct-approx}, for all large $m$, $m\lct^{[t]} \left(X,\F;I_{m,\lceil ma_0\rceil}\right)\ge \frac{c_0}{2}$. Hence,
\[
\frac1m w_m(I_{m,\lceil ma_0\rceil}) \le \frac{2}{c_0}A^{[t]}_{X,\F}(w_m).
\]
Combining the last two inequalities gives the uniform bound
\[
        A^{[t]}_{X,\F}(w_m) \ge \delta_0:=\frac{c_0\gamma}{8}>0.
\]

We may now take $a=\mathbf L^{[t]}(\cG)$ in the preceding base-ideal estimate. The convergence of $c_m^{[t]}$ and the bound $A^{[t]}_{X,\F}(w_m)\ge\delta_0$ give
\[
        \liminf_m\frac{w_m(I_{m,\lceil m\mathbf{L}^{[t]}(\cG)\rceil})}{mA^{[t]}_{X,\F}(w_m)}\ge 1.
\]

Now, the finite-level valuative formula gives
\[
m\lct^{[t]} \left(X,\F;I_{m,\lceil m\mathbf L^{[t]}(\cG)\rceil}\right) \le \frac{mA^{[t]}_{X,\F}(w_m)} {w_m(I_{m,\lceil m\mathbf L^{[t]}(\cG)\rceil})},
\]
and Lemma \ref{lem:finite-level-mixed-lct-approx} identifies the limit of the left side with the asymptotic threshold. Consequently, $\lct^{[t]} \left( X,\F; I^{(\mathbf{L}^{[t]}(\cG))}_{\bullet,\cG} \right) \le 1$. Because the threshold at $\mathbf L^{[t]}(\cG)$ is finite, the strict-decrease statement in Lemma \ref{lem:mixed-Xu-344} therefore shows that the inequality above implies $\mu^{[t]}(\cG)\le \mathbf{L}^{[t]}(\cG)$.

Combining the two inequalities gives
\[
        \mu^{[t]}(\cG)=\mathbf{L}^{[t]}(\cG) = c_\infty^{[t]}(\cG,e_+)+e_+-(1-t).
\]
The independence of $e_+$ follows immediately from the formula, since $\mu^{[t]}(\cG)$ was defined without reference to $e_+$. Since the integer flag ideals of $\cG$ and $\cG_{\mathbb Z}$ agree, the same identity also gives
\[
\mu^{[t]}(\cG)=\mu^{[t]}(\cG_{\mathbb Z}), \qquad \Ding^{[t]}(\cG)=\Ding^{[t]}(\cG_{\mathbb Z}).
\]
\end{proof}

\begin{lemma}
\label{lem:mixed-filtration-approximation}
Let $\cG^\bullet R$ be a linearly bounded multiplicative filtration and let $\cG^{(q)}$ be its finitely generated approximants. Then 
\[
        \Ding^{[t]}(\cG^{(q)})\longrightarrow \Ding^{[t]}(\cG),
\]
and
\[
        J(\cG^{(q)})\longrightarrow J(\cG).
\]
\end{lemma}

\begin{proof}
By \cite[Theorem 3.58]{Xu2025}, $S_L(\cG^{(q)})\longrightarrow S_L(\cG)$. Also, by \cite[Theorem 3.60]{Xu2025} we have $\lambda_{\max}(\cG^{(q)}) \longrightarrow \lambda_{\max}(\cG)$ and therefore $J(\cG^{(q)})\to J(\cG)$.

For the mixed slope, since $\cG^{(q)}$ is integer-valued, \cite[Lemma 3.59]{Xu2025} gives $\mathfrak I_{q\ell}(\cG^{(q)}) =\mathfrak I_q(\cG^{(q)})^\ell$ for all $\ell\ge1$. Furthermore, $\cG^{(q)}$ agrees with $\cG_{\mathbb Z}$ in degree $q$, so $\mathfrak I_q(\cG^{(q)}) =\mathfrak I_q(\cG_{\mathbb Z}) =\mathfrak I_q(\cG)$.  It follows directly from the definitions that $c_\infty^{[t]}(\cG^{(q)},e_+) =c_q^{[t]}(\cG,e_+)$. Hence, the right-hand side converges to $c_\infty^{[t]}(\cG,e_+)$. Applying Lemma \ref{lem:slope-flag-ideal-comparison} to $\cG^{(q)}$ and $\cG$ yields
\[
        \mu^{[t]}(\cG^{(q)}) \longrightarrow \mu^{[t]}(\cG).
\]
Finally, $\Ding^{[t]}=\mu^{[t]}-S_L$, so the asserted Ding convergence follows.
\end{proof}

\begin{proposition}\label{lem:fg-filtration-tc-comparison}
Let $\cG$ be an integer-valued linearly bounded multiplicative filtration. Assume that, for some sufficiently divisible $q\in r_0\N$, $\mathfrak I_{q\ell}(\cG)=\mathfrak I_q(\cG)^\ell$ for all $\ell\ge1$. Let $(\X_{\cG},\F_{\cG},\L_{\cG})\to\A^1$ be the normal semiample test configuration obtained from the normalised blow-up of $\mathfrak I_q(\cG)$ and its semiample model, equipped with the saturated birational transform of $\F\times\A^1$. Then, for every $\eta\in[0,1]$,
\[
\Ding^{[t]}(\cG)-\eta J(\cG) \ge \Ding^{[t]}(\X_{\cG},\F_{\cG},\L_{\cG}) -\eta J^{\NA}(\X_{\cG},\L_{\cG}).
\]
\end{proposition}

\begin{proof}
We follow the converse direction of the proof of \cite[Theorem 3.63]{Xu2025}. Let $p:\Y\to\overline\X_{\cG}$ and $\rho:\Y\to X\times\mathbb P^1$ be the common normalised graph, and write $\mathfrak I_q(\cG)\cO_\Y=\cO_\Y(-E)$. We have
\[
        p^*\overline\L_{\cG} =\rho^*p_1^*L-\frac1qE.
\]
By the definition of the correction divisor,
\[
\begin{aligned}
p^*\left(K^{[t]}_{\overline\X_{\cG}/\mathbb P^1,
\overline\F_{\cG}}+D^{[t]}_{\X_{\cG},\F_{\cG},\L_{\cG}} \right)
&=-p^*\overline\L_{\cG}\\
&=\rho^*K^{[t]}_{X\times\mathbb P^1/\mathbb P^1, \F\times\mathbb P^1}+\frac1qE.
\end{aligned}
\]
Thus, taking mixed discrepancies of both sides shows that the two threshold inequalities agree for every divisorial valuation over the common graph; hence
\[
\lct^{[t]} \left( \overline\X_{\cG},\overline\F_{\cG}; D^{[t]}_{\X_{\cG},\F_{\cG},\L_{\cG}}; (\overline\X_{\cG})_0 \right) =c_q^{[t]}(\cG,e_+),
\]
and $c_q^{[t]}(\cG,e_+)=c_\infty^{[t]}(\cG,e_+)$. Combining the two, we obtain
\[
\mathbf L^{[t]}(\cG) = \lct^{[t]} \left( \overline\X_{\cG},\overline\F_{\cG}; D^{[t]}_{\X_{\cG},\F_{\cG},\L_{\cG}}; (\overline\X_{\cG})_0 \right) +e_+-(1-t).
\]

The section-ring argument in \cite[Equations (3.42) and (3.43)]{Xu2025} is independent of discrepancies and gives $\lambda_{\max}(\cG)-e_+ \le \lambda_{\max}(\cG_{\X})$ and $S_L(\cG)-e_+ \le S_L(\cG_{\X})$, where here $\cG_{\X}$ is the filtration induced by the constructed test configuration. Hence, we obtain
\[
\begin{aligned}
\Ding^{[t]}(\cG)-\eta J(\cG) &=\mathbf L^{[t]}(\cG) -(1-\eta)S_L(\cG)-\eta\lambda_{\max}(\cG)\\
&\ge \Ding^{[t]}(\X_{\cG},\F_{\cG},\L_{\cG}) -\eta J^{\NA}(\X_{\cG},\L_{\cG}),
\end{aligned}
\]
which proves the assertion.
\end{proof}

\begin{theorem}
\label{thm:mixed-ding-compatible-filtrations}
Let $(X,\F,t)$ be a $t$-K-semistable adjoint Fano foliated structure. Then for every linearly bounded multiplicative filtration $\cG^\bullet R(X,L)$, we have $\Ding^{[t]}(\cG)\ge 0$.

If $(X,\F,t)$ is uniformly $t$-K-stable with constant $\delta>0$, then, for every $0<\eta<\min\{\delta,1\}$,
\[
        \Ding^{[t]}(\cG)\ge \eta J(\cG)
\]
for every linearly bounded multiplicative filtration $\cG^\bullet R(X,L)$.
\end{theorem}

\begin{proof}
Replacing $\cG$ by $\cG_{\mathbb Z}$ does not change $\Ding^{[t]}$ or $J$, by Lemma \ref{lem:slope-flag-ideal-comparison} and \cite[Lemma 3.28]{Xu2025}. We may therefore assume that $\cG$ is integer-valued. Let $\cG^{(q)}$ be its minimal approximants. By Proposition \ref{prop:approximants-TC}, each $\cG^{(q)}$ defines a normal semiample $\F$-compatible test configuration $(\X_q,\F_q,\L_q)\to \A^1$.

Since $(X,\F,t)$ is $t$-K-semistable, it is $t$-Ding-semistable by \cite[Corollary 7.10]{Pap26}. By Remark \ref{rem: ding in terms of semiample}, we may test on $(\X_q,\F_q,\L_q)$. Hence, $\Ding^{[t]}(\X_q,\F_q,\L_q)\ge 0$, while Proposition \ref{lem:fg-filtration-tc-comparison} gives
\[
\Ding^{[t]}(\cG^{(q)}) \ge \Ding^{[t]}(\X_q,\F_q,\L_q) \ge0
\]
for every $q$.

We now pass to the limit. By Lemma \ref{lem:mixed-filtration-approximation}, $\Ding^{[t]}(\cG^{(q)})\to \Ding^{[t]}(\cG)$. Therefore
\[
        \Ding^{[t]}(\cG)\ge 0.
\]

Assume now that $(X,\F,t)$ is uniformly $t$-K-stable with constant $\delta>0$, and fix $0<\eta<\min\{\delta,1\}$. By \cite[Corollary 7.10]{Pap26}, $\Ding^{[t]}(\X_q,\F_q,\L_q) \ge \eta J^{\NA}(\X_q,\L_q)$. Using Proposition \ref{lem:fg-filtration-tc-comparison} with this value of $\eta$ gives
\[
\Ding^{[t]}(\cG^{(q)})-\eta J(\cG^{(q)}) \ge \Ding^{[t]}(\X_q,\F_q,\L_q) -\eta J^{\NA}(\X_q,\L_q) \ge0.
\]
Lemma \ref{lem:mixed-filtration-approximation} gives convergence of both terms. Letting $q\to\infty$, we obtain
\[
        \Ding^{[t]}(\cG)\ge \eta J(\cG).
\]
This proves both assertions.
\end{proof}

\section{Inversion of adjunction with ideals} \label{sec:F-compatible-IOA}

In this section we prove the inversion-of-adjunction statements needed in the valuative arguments. Throughout, we work over $\C$. The same arguments apply over any algebraically closed field of characteristic zero for which the cited MMP, qdlt, adjunction, and Bertini results are available. Let $f:(X,\F)\to C$ be a flat projective morphism to a smooth curve, and let $S:=X_0$ be a reduced normal Cartier fibre. We assume that $X$ is normal, that $T_{\F}\subset T_{X/C}$, that $\F$ is algebraically integrable, and that $0<t<1$. In particular, $S$ is $\F$-invariant. We assume that the divisor $K^{[t]}_{X/C,\F}+(1-t)S$ is $\Q$-Cartier. Since $S$ is normal, the saturation of $T_{\F}|_S$ in $T_S$ defines the induced foliation on $S$, which we denote by $\F_S$.

We use the  convention for ideals from the previous sections; if $\mathfrak a\subset \cO_X$ is a nonzero ideal and $c\geq 0$, then $(X,\F,B+c\mathfrak a,t)$ is lc if
\[
        A^{[t]}_{X,\F}(E)-\operatorname{ord}_E(B)
        -c\,\operatorname{ord}_E(\mathfrak a)\geq 0
\]
for every prime divisor $E$ over $X$. Here $B$ denotes an effective $\mathbb Q$-divisor or, more generally, a fixed boundary contribution for which the displayed discrepancies are defined. Since all applications concern the fibre $S$, we use the threshold near $S$:
\[
\lct^{[t]}_{X,\F;B;S}(\mathfrak a) := \sup\left\{
c\geq 0 \;\middle|\; (X,\F,B+c\mathfrak a,t)\text{ is lc in a neighbourhood of }S \right\}
\]
and omit the final $S$ from the notation when no confusion is possible. Equivalently, the infimum in the divisorial formula of Equation \eqref{eq: mixed lct} is taken over prime divisors whose centres meet $S$. If $X\to\Spec R$ is proper over a DVR of finite type over $\C$, every open neighbourhood of the entire special fibre is all of $X$, so this local notation agrees with the total-space threshold used later.

We also fix the following terminology. We say that $\mathcal A=(X,\F,(1-t)S+B,t)$ is \emph{plt near $S$} if $A_{\mathcal A}(S)=0$ and every prime divisor $E\ne S$ whose centre meets $S$ satisfies $A_{\mathcal A}(E)>0$. Thus failure of plt near $S$, once $S$ is an lc place, means that there is a prime divisor $E\ne S$, with $c_X(E)\cap S\ne\varnothing$, such that $A_{\mathcal A}(E)\leq0$.

\subsection{Invariant vertical divisors}

We first recall the elementary invariance property of vertical divisors.

\begin{lemma}[{\cite[Lemma 4.6]{Pap26}}] \label{lem:vertical-divisors-invariant}
Let $f:(X,\F)\to C$ be as above, with $T_{\F}\subset T_{X/C}$. Let $E$ be a prime divisor over $X$ which is vertical over $C$. Then $E$ is invariant for the induced foliation on any birational model on which $E$ appears.
\end{lemma}

\subsection{Divisorial representatives and single-place extraction}

\begin{lemma} \label{lem:general-divisorial-representative}
Let $\cA=(X,\F,B+\lambda\mathfrak a,t)$ be lc in a neighbourhood of $S$, where $\lambda\in\Q_{\geq0}$, and let $E$ be an lc place of $\cA$. After shrinking around $S$, there is an effective $\Q$-Cartier $\Q$-divisor $D$ such that
\begin{enumerate}
\item $(X,\F,B+D,t)$ is lc;
\item $E$ is an lc place of $(X,\F,B+D,t)$; and
\item $\ord_E(D)=\lambda\ord_E(\mathfrak a)$.
\end{enumerate}
\end{lemma}

\begin{proof}
There is nothing to prove if $\lambda=0$, so we assume that $\lambda>0$. Choose a foliated log resolution $\nu:W\to X$ on which $E$ appears and which principalises $\mathfrak a$, i.e. $\mathfrak a\cO_W=\cO_W(-F)$. After twisting by a sufficiently ample Cartier divisor $A$ on $X$, we may assume that $M:=\nu^*A-F$ is base-point-free in a neighbourhood of $\nu^{-1}(S)$. We now apply the Bertini-type theorem \cite[Theorem 3.28(1)]{CHLMSSX25} to the semiample $\Q$-divisor $\lambda M$ and to the crepant transform of $\cA$. It gives a sufficiently general effective $\Q$-divisor $L\sim_{\Q}\lambda M$ such that adding $L$ preserves log canonicity. We choose $L$ so that it does not contain $E$.

Since $\lambda F+L\sim_{\Q}\nu^*(\lambda A)$, this divisor descends to an effective $\Q$-Cartier divisor $D\sim_{\Q}\lambda A$ on $X$, and $\nu^*D=\lambda F+L$. As $L$ does not contain $E$, $\ord_E(D)=\lambda\ord_E(\mathfrak a)$. The crepant transform is lc by the preceding Bertini argument, and its coefficient along $E$ is unchanged from the ideal transform. Thus $E$ remains an lc place. Pushing down proves all three assertions.
\end{proof}

\begin{lemma}\label{lem:single-place-extraction}
Let $\cA=(X,\F,B+\lambda\mathfrak a,t)$ be an lc adjoint foliated structure in a neighbourhood of $S$, where $\lambda\in\Q_{\geq0}$, and suppose that $X$ is potentially klt there. Let $E$ be an exceptional lc place of $\cA$. Then, after shrinking around $S$, there exists a projective birational morphism $\rho:Y\to X$ such that $Y$ is $\Q$-factorial and $E$ is the only $\rho$-exceptional prime divisor.
\end{lemma}

\begin{proof}
By Lemma \ref{lem:general-divisorial-representative}, there is an effective boundary $D$ for which $(X,\F,B+D,t)$ is lc and $E$ is an lc place. An lc place has discrepancy $-\bigl(t\varepsilon_{\F}(E)+(1-t)\bigr)<0$, so applying \cite[Theorem 2.2.3]{CHLMSSX25} to the one-element set $\{E\}$ gives a projective $\Q$-factorial model over $X$ whose only exceptional prime divisor is $E$, as claimed.
\end{proof}

\begin{lemma}
\label{lem:extracted-place-meets-special-fibre}
Let $S\subset X$ be a normal Cartier divisor and let $\rho:Y\to X$ be a projective birational morphism from a normal variety whose only exceptional prime divisor is $E$. If $c_X(E)\cap S\ne\varnothing$, and $S_Y$ is the strict transform of $S$, then $E\cap S_Y\ne\varnothing$.
\end{lemma}

\begin{proof}
If $c_X(E)\subset S$, then, since $S$ is Cartier, $\rho^{-1}(S)_{\red}=S_Y\cup E$. The fibres of $\rho$ are connected because $X$ is normal and $\rho_*\cO_Y=\cO_X$. Hence the two components meet.

If $c_X(E)\not\subset S$, then $E$ is not a component of $\rho^{-1}(S)$, so $\rho^{-1}(S)_{\red}=S_Y$. We then choose a point of $c_X(E)\cap S$. Its inverse image meets $E$ and is contained in $\rho^{-1}(S)_{\red}=S_Y$. Thus again $E\cap S_Y\ne\varnothing$.
\end{proof}

\begin{lemma}\label{lem:restriction-over-lc-invariant-component}
Let $f\colon (Y,\mathcal F_Y)\to C$ be a flat morphism to a smooth curve, with $Y$ $\mathbb Q$-factorial, and let $S_Y\subset Y$ be a reduced $\mathcal F_Y$-invariant divisor mapping to a closed point of $C$. Let $E\subset Y$ be a distinct prime divisor such that $E\cap S_Y\neq\varnothing$. Suppose that the adjoint boundary on $Y$ has the form $(1-t)S_Y+\delta E+B'$, where $E\not\subset\operatorname{Supp}B'$ and $S_Y\not\subset\operatorname{Supp}(\delta E+B')$. Let $\nu\colon S_Y^\nu\to S_Y$ be the normalisation.

Then the induced adjoint foliated structure on $S_Y^\nu$ is:
\begin{enumerate}
\item not klt if $\delta\geq t\epsilon_{\mathcal F_Y}(E)+(1-t)$;
\item not lc if $\delta> t\epsilon_{\mathcal F_Y}(E)+(1-t)$.
\end{enumerate}
In both cases, the failure occurs along a prime divisor whose
centre lies over $E\cap S_Y$.
\end{lemma}

\begin{proof}
We let $\delta_0:=t\epsilon_{\mathcal F_Y}(E)+(1-t)$. Since $Y$ is $\mathbb Q$-factorial, the restriction $E|_{S_Y^\nu}$ is a well-defined nonzero effective $\mathbb Q$-divisor supported over $E\cap S_Y$. We now choose a prime divisor $P\subset\operatorname{Supp}(E|_{S_Y^\nu})$ and set $r_P:=\operatorname{ord}_P(E|_{S_Y^\nu})>0$.

Let $B_{S_Y^\nu}(\delta)$ denote the boundary induced by adjunction. By the construction of adjunction in
\cite[Theorem 1.8]{CHLMSSX24}, the induced boundary depends
linearly on the ambient boundary, and hence
\[
B_{S_Y^\nu}(\delta) = B_{S_Y^\nu}(\delta_0) + (\delta-\delta_0)E|_{S_Y^\nu}.
\]
The local coefficient calculation in the proof of \cite[Theorem 1.8]{CHLMSSX24} shows that, at the critical coefficient $\delta_0$, the induced structure is not klt along a divisor lying over $E\cap S_Y$. Thus, after choosing such a component $P$, we have
\[
A^{[t]}_{S_Y^\nu,\mathcal F_{S_Y^\nu}}(P) - \operatorname{ord}_P B_{S_Y^\nu}(\delta_0) \leq0.
\]
It follows that
\[
\begin{aligned}
&A^{[t]}_{S_Y^\nu,\mathcal F_{S_Y^\nu}}(P)- \operatorname{ord}_P B_{S_Y^\nu}(\delta)
\\
&\qquad\leq -(\delta-\delta_0)r_P.
\end{aligned}
\]
If $\delta=\delta_0$, the right-hand side is zero, so the induced
structure is not klt. If $\delta>\delta_0$, then the right-hand side is strictly negative because $r_P>0$, so the induced structure is not lc.
\end{proof}

\begin{lemma}
\label{lem:principalised-ideal-adjunction}
Let $\cA_0=(X,\F,(1-t)S+B,t)$ be an adjoint foliated structure for which the mixed canonical divisor is $\Q$-Cartier, and let $\mathfrak a\cO_S\ne0$. Choose a foliated log resolution $\mu:W\to X$ which principalises $\mathfrak a$, i.e. $\mathfrak a\cO_W=\cO_W(-F)$ and let $S_W$ be the strict transform of $S$. If
\[
 K^{[t]}_{W,\F_W}+(1-t)S_W+B_W = \mu^*\bigl(K^{[t]}_{X,\F}+(1-t)S+B\bigr),
\]
then, on the normalisation $S_W^\nu$,
\[
 \overline{\mathfrak a_S\cO_{S_W^\nu}} = \cO_{S_W^\nu}\bigl(-F|_{S_W^\nu}\bigr).
\]
Furthermore, for every $c\geq0$, adjunction of the crepant transform of $(\mathcal X,\mathcal F,(1-t)S+B+c\mathfrak a,t)$
 to $S_W^\nu$ agrees with the crepant transform of $(S,\mathcal F_S,B_S+c\mathfrak a_S,t)$.
\end{lemma}

\begin{proof}
Because $\mathfrak a\cO_S\ne0$, the divisor $S_W$ is not a component of $F$. We write locally $\mathfrak a=(f_1,\ldots,f_r)$. On $W$, the functions $\mu^*f_i$ have the common divisorial factor defining $F$, and their quotients generate the unit ideal. Restricting to $S_W$ and then normalising shows that the divisorial integral closure of the restricted ideal is exactly $\cO_{S_W^\nu}(-F|_{S_W^\nu})$.

The construction of the different in \cite[Theorem 1.8]{CHLMSSX24} is performed on a common foliated log resolution by combining ordinary and foliated adjunction. Applying that construction to the displayed crepant equality, and then adding $cF$, adds precisely $cF|_{S_W^\nu}$ to the restricted boundary. The same proof identifies the latter divisor with the pullback of $c\mathfrak a_S$. This proves the discrepancy identity. 
\end{proof}

\begin{lemma} \label{lem:mixed-ideal-threshold-attained}
Assume that $\cA_0=(X,\F,(1-t)S+B,t)$ is lc on a neighbourhood of $S$, and that $\mathfrak a\cO_S\ne0$. If
\[
 c=\lct^{[t]}_{X,\F;(1-t)S+B;S}(\mathfrak a)<+\infty,
\]
then $c\in\Q$ and there is a prime divisor $E$ over $X$, with $c_X(E)\cap S\ne\varnothing$, $\ord_E(\mathfrak a)>0$, and
\[
 A_{\cA_0}(E)-c\ord_E(\mathfrak a)=0.
\]
\end{lemma}

\begin{proof}
We take the simultaneous foliated log resolution of $(X,\F,S+B,\mathfrak a)$ used in Lemma \ref{lem:principalised-ideal-adjunction}. On this model, the coefficient of every component is affine-linear in $c$. Let $c_0$ be the minimum of the finitely many ratios
\[
 \frac{A_{\cA_0}(D)}{\ord_D(\mathfrak a)}
\]
over components $D$ with positive denominator and with centre meeting $S$. At every $0\leq c\leq c_0$, the coefficient of each non-invariant component is at most $1$, and that of each invariant component is at most $1-t$. The adjoint foliated structure at $c_0$ is lc by \cite[Lemma 2.14]{CHLMSSX24}; hence the adjoint foliated structure with $c<c_0$ is also lc. If $c>c_0$, a component attaining the minimum has negative discrepancy. Thus $c=c_0$, proving both rationality and attainment.
\end{proof}

\subsection{Singularity criteria on admissible DVR families}
\label{subsec:singularity-criteria-admissible-DVR}

We next establish the singularity properties of admissible families over a DVR that will be needed below. We first expand Lemma \ref{lem:single-place-extraction}.

\begin{lemma}\label{lem:qdlt-prescribed-nklt}
Let $\cA=(X,\F,B,t)/U$ be an adjoint foliated structure with $t<1$, and let $E_1,\dots,E_r$ be finitely many nklt places of $\cA$. Then there exists a projective birational morphism
\[
        h:(Y,\F_Y,B_Y,t)\longrightarrow(X,\F,B,t)
\]
such that:
\begin{enumerate}
\item $Y$ is $\Q$-factorial and klt;
\item $(Y,\F_Y,B_Y,t)$ is qdlt;
\item every $E_i$ appears as a prime divisor on $Y$;
\item every $h$-exceptional prime divisor is an nklt place of $\cA$.
\end{enumerate}
If $\cA$ is lc, then $h$ may be chosen crepant.
\end{lemma}

\begin{proof}
We first choose the initial foliated log resolution in the proof of \cite[Theorem 3.6]{CHLMSSX24} so that every $E_i$ appears as a divisor. In that proof, the relative MMP contracts precisely the exceptional divisors $D$ satisfying $a(D,\cA)> -t\varepsilon_{\F}(D)-(1-t)$, namely the exceptional divisors which are not nklt places of $\cA$. Hence, none of the prescribed divisors $E_i$ is contracted. The output is therefore a $\Q$-factorial qdlt modification on which all the $E_i$ appear. The remaining assertions, including the klt property of the underlying variety and crepantness in the lc case, all follow from \cite[Theorem 3.6]{CHLMSSX24}. For the case of an lc place, this prescribed extraction is also recorded explicitly in \cite[Proposition 3.8]{CHLMSSX24}.
\end{proof}

\begin{theorem}[Inversion of adjunction over a DVR] \label{thm:dvr-fibre-ioa}
Let $R$ be a DVR essentially of finite type over $\C$ and let $f:(X,\F,t)\longrightarrow\Spec R$ be a flat projective family. Let $S:=X_\kappa$ be its special fibre. Assume that:
\begin{enumerate}
\item $X$ is normal and $S$ is a reduced normal Cartier divisor;
\item $T_\F\subset T_{X/R}$ and $\F$ is algebraically integrable;
\item the divisor $K^{[t]}_{X/R,\F}+(1-t)S$ is $\Q$-Cartier;
\item $(S,\F_S,t)$ is klt.
\end{enumerate}
Then $(X,\F,(1-t)S,t)$ is plt in a neighbourhood of $S$.
\end{theorem}

\begin{proof}
We fix the adjoint foliated structure $\cA:=(X,\F,(1-t)S,t)$. The divisor $S$ is invariant by Lemma \ref{lem:vertical-divisors-invariant}, and
\[
        A_\cA(S)=A^{[t]}_{X,\F}(S)-(1-t)\ord_S(S)=(1-t)-(1-t)=0.
\]
Thus $S$ is an lc place of $\cA$.

Suppose that $\cA$ is not plt near $S$. Thus, there exists a prime divisor $E\ne S$ over $X$, with $c_X(E)\cap S\ne\varnothing$, such that
\[
        A_\cA(E) = A^{[t]}_{X,\F}(E) -(1-t)\ord_E(S) \leq0.
\]
Thus $E$ is an nklt place of $\cA$.

We now apply Lemma \ref{lem:qdlt-prescribed-nklt} to $\cA$ and the prescribed nklt place $E$. We obtain a projective birational morphism
\[
        h:(Y,\F_Y,B_Y,t)\longrightarrow (X,\F,(1-t)S,t),
\]
where $Y$ is $\Q$-factorial and klt, the induced adjoint foliated structure is qdlt, and $E$ appears as an $h$-exceptional divisor.

Let $S_Y$ be the strict transform of $S$. Every $h$-exceptional prime divisor $D$ is an nklt place of $\cA$, and its coefficient in the qdlt boundary $B_Y$ is equal to $t\varepsilon_{\F_Y}(D)+(1-t)$. In particular, a vertical exceptional divisor is invariant and has coefficient $1-t$.

We write
\begin{equation}
\label{eq:qdlt-discrepancy-error}
        h^*\left(K^{[t]}_{X/R,\F}+(1-t)S\right)= K^{[t]}_{Y/R,\F_Y}+B_Y+F,
\end{equation}
where $F$ is $h$-exceptional. Notice that, for every $h$-exceptional prime divisor $D$, $\coeff_D(F)=-A_\cA(D)$. All such divisors are nklt places, so $F$ is effective.

If some $h$-exceptional divisor is vertical, then
\[
 h^{-1}(S)_{\red} = S_Y\cup\{\text{vertical $h$-exceptional divisors}\}.
\]
This set is connected because $h_*\cO_Y=\cO_X$; hence one of the vertical components meets $S_Y$. If there is no vertical exceptional divisor, the prescribed divisor $E$ is horizontal and Lemma \ref{lem:extracted-place-meets-special-fibre} shows that $E\cap S_Y\ne\varnothing$. Thus, in both cases there exists a component $D\ne S_Y$ of $B_Y$ which meets $S_Y$, and its coefficient in $B_Y$ is $t\varepsilon_{\F_Y}(D)+(1-t)$. Lemma \ref{lem:restriction-over-lc-invariant-component} therefore shows that the adjoint foliated structure induced by $(Y,\F_Y,B_Y,t)$ on $S_Y^\nu$ is not klt.

Let $g:S_Y^\nu\longrightarrow S$ be the induced birational morphism. Restricting \eqref{eq:qdlt-discrepancy-error} to $S_Y^\nu$ and applying \cite[Theorem 1.8]{CHLMSSX24}, we obtain
\[
        g^*K^{[t]}_{S,\F_S} = K_{\cA_{S_Y^\nu}}+F|_{S_Y^\nu},
\]
where $\cA_{S_Y^\nu}$ denotes the adjoint foliated structure induced by $(Y,\F_Y,B_Y,t)$ on $S_Y^\nu$. Since $Y$ is $\Q$-factorial, $F|_{S_Y^\nu}$ is a well-defined effective $\Q$-divisor. Hence for every prime divisor $P$ over $S_Y^\nu$,
\[
        A^{[t]}_{S,\F_S}(P) \leq A_{\cA_{S_Y^\nu}}(P).
\]
Since $\cA_{S_Y^\nu}$ is not klt, this implies that $(S,\F_S,t)$ is not klt, contradicting assumption (4). Therefore $(X,\F,(1-t)S,t)$ is plt near $S$.
\end{proof}

\begin{corollary} \label{cor:dvr-total-potentially-klt}
Under the assumptions of Theorem \ref{thm:dvr-fibre-ioa}, the adjoint foliated structure $(X,\F,t)$ is klt in a neighbourhood of $S$. In particular, the underlying variety $X$ is potentially klt near $S$.
\end{corollary}

\begin{proof}
By Theorem \ref{thm:dvr-fibre-ioa}, $ (X,\F,(1-t)S,t)$ is plt near $S$. Every divisor other than $S$ whose centre meets $S$ already has positive discrepancy, and removing the boundary $(1-t)S$ strictly increases the discrepancy of every divisor with positive order along $S$, including $S$ itself. Hence $(X,\F,t)$ is klt near $S$.

Since $t<1$, the underlying variety of a klt adjoint foliated structure is potentially klt by \cite[Theorem 1.10(1)]{CHLMSSX24}. Therefore $X$ is potentially klt in a neighbourhood of $S$.
\end{proof}

\begin{proposition}
\label{prop: families satisfy IOA}
Let $f:(\X,\F,t)\longrightarrow\Spec R$ be an admissible $t$-K-semistable family over a DVR. Then:
\begin{enumerate}
\item $(\X,\F,(1-t)\X_\kappa,t)$ is plt near $\X_\kappa$;
\item $(\X,\F,t)$ is klt near $\X_\kappa$;
\item the underlying variety $\X$ is potentially klt near
$\X_\kappa$.
\end{enumerate}
\end{proposition}

\begin{proof}
By admissibility, $\X$ is normal near the special fibre, $\X_\kappa$ is a reduced normal Cartier divisor, $T_\F\subset T_{\X/R}$, the foliation is algebraically integrable, and $K^{[t]}_{\X/R,\F}+(1-t)\X_\kappa$ is $\Q$-Cartier. By Theorem \ref{thm: k-stab implies klt}, the $t$-K-semistable special fibre $(\X_\kappa,\F_\kappa,t)$ is klt. Theorem \ref{thm:dvr-fibre-ioa} gives (1), and Corollary \ref{cor:dvr-total-potentially-klt} gives (2) and (3).
\end{proof}

\begin{remark}
\label{rem:comparison-Cascini-Liu-IOA}
P. Cascini and J. Liu have informed the author of a more general inversion-of-adjunction statement for adjoint foliated structures: if $S$ is a normal $\Q$-Cartier component of the round-down of the boundary, the ambient variety is potentially klt, and the adjunction to $S$ is klt, then the ambient adjoint foliated structure is plt near $S$. Theorem \ref{thm:dvr-fibre-ioa} is tailored to a Cartier fibre over a DVR. In this setting, potential klt singularities of the ambient variety are a consequence, rather than an assumption, of inversion of adjunction and \cite[Theorem 1.10(1)]{CHLMSSX24}.
\end{remark}

\subsection{Ideal adjunction}

We now prove the finite-level inversion-of-adjunction statement.

\begin{theorem}[Inversion of adjunction with ideals]
\label{thm:F-compatible-IOA-finite}
Let $f:(X,\F)\to C$ be as above, and let $S=X_0$ be a reduced normal Cartier fibre. Assume:
\begin{enumerate}
\item $T_{\F}\subset T_{X/C}$;
\item $X$ is potentially klt near $S$;
\item the divisor $K^{[t]}_{X/C,\F}+(1-t)S$ is $\Q$-Cartier;
\item $(X,\F,(1-t)S,t)$ is lc near $S$.
\end{enumerate}
Let $\mathfrak a\subset \cO_X$ be any nonzero coherent ideal not vanishing identically along $S$. Then
\[
        \lct^{[t]}_{X,\F;(1-t)S}(\mathfrak a)  = \lct^{[t]}_{S,\F_S}(\mathfrak a_S).
\]
\end{theorem}

\begin{proof}
We set $c_X := \lct^{[t]}_{X,\F;(1-t)S}(\mathfrak a)$ and $c_S := \lct^{[t]}_{S,\F_S}(\mathfrak a_S)$. We first prove $c_X\leq c_S$. Suppose that $(X,\F,(1-t)S+c\mathfrak a,t)$ is lc near $S$. We take the simultaneous foliated log resolution of Lemma \ref{lem:principalised-ideal-adjunction}, which we denote by $\mu:W\to X$, and write, as before, $\mathfrak a\cO_W=\cO_W(-F)$. By crepancy, adjunction along the invariant component $S_W$, followed by Lemma \ref{lem:principalised-ideal-adjunction}, we see that $(S,\F_S,c\mathfrak a_S,t)$ is lc. This proves $c_X\leq c_S$.

For the reverse inequality, we fix $c<c_S$, and suppose for contradiction that $\cA_c=(X,\F,(1-t)S+c\mathfrak a,t)$ is not lc near $S$. Then $\lambda:=c_X<c$. By Lemma \ref{lem:mixed-ideal-threshold-attained}, $\lambda\in\Q$ and there is a prime divisor $E\ne S$, with $c_X(E)\cap S\ne\varnothing$, such that, after setting $q:=\ord_E(\mathfrak a)>0$, we have
\begin{equation}\label{eq:ideal-threshold-place}
    A^{[t]}_{X,\F}(E) -(1-t)\ord_E(S)-\lambda q=0.
\end{equation}

We first assume that $E$ is exceptional over $X$. By Lemma \ref{lem:single-place-extraction}, there is a projective birational morphism $\rho:Y\to X$ from a $\Q$-factorial variety whose only exceptional prime divisor is $E$. Let $S_Y$ be the strict transform of $S$. By Lemma \ref{lem:extracted-place-meets-special-fibre}, $E\cap S_Y\ne\varnothing$.

We write
\[
 K^{[t]}_{Y,\F_Y}+(1-t)S_Y+\Gamma_Y = \rho^*\bigl(K^{[t]}_{X,\F}+(1-t)S\bigr).
\]
Since $Y$ is $\Q$-factorial and $q=\ord_E(\mathfrak a)$, there is an inclusion of divisorial ideals
\begin{equation}\label{eq:ideal-contained-qE}
    \mathfrak a\cO_Y\subseteq\cO_Y(-qE).
\end{equation}
For every prime divisor $P$ over $Y$, the above gives $\ord_P(\mathfrak a)\geq q\ord_P(E)$. Hence log canonicity of $\cA_\lambda$ implies log canonicity of the adjoint foliated structure
\begin{equation}\label{eq:divisorial-comparison-lambda}
    (Y,\F_Y,(1-t)S_Y+\Gamma_Y+\lambda qE,t).
\end{equation}
Equation \eqref{eq:ideal-threshold-place} shows that the coefficient of $E$ in \eqref{eq:divisorial-comparison-lambda} is exactly $t\varepsilon_{\F_Y}(E)+(1-t)$. After replacing $\lambda qE$ by $cqE$, that coefficient becomes strictly larger. By Lemma \ref{lem:restriction-over-lc-invariant-component}, and adjunction to $S_Y^\nu$ is therefore not lc.

On the other hand, restricting \eqref{eq:ideal-contained-qE} to
$S_Y^\nu$ gives $\ord_P(\mathfrak a_S)\geq q\ord_P(E|_{S_Y^\nu})$ for every prime divisor $P$ over $S_Y^\nu$. Consequently the crepant transform of $(S,\F_S,c\mathfrak a_S,t)$ is at least as singular as the restriction of the structure with $cqE$. It is therefore not lc, contradicting $c<c_S$.

If $E$ is not exceptional, then $\ord_E(S)=0$ and
\[
        A^{[t]}_{X,\F}(E) =t\varepsilon_{\F}(E)+(1-t).
\]
Thus Equation \eqref{eq:ideal-threshold-place} gives $\lambda q=t\varepsilon_{\F}(E)+(1-t)$. Let $P$ be a prime divisor on $S$ contained in $E\cap S$. Restriction of $\mathfrak a\subseteq\cO_X(-qE)$ gives $\ord_P(\mathfrak a_S)\geq q$. Furthermore,
\[
 t\varepsilon_{\F_S}(P)+(1-t)  \leq t\varepsilon_{\F}(E)+(1-t);
\]
if $E$ is invariant, then $P$ is invariant, while for a
non-invariant $E$ the right-hand side is $1$. Hence, in both cases,
\[
\begin{aligned}
 A^{[t]}_{S,\F_S}(P)-c\ord_P(\mathfrak a_S) &\leq t\varepsilon_{\F}(E)+(1-t)-cq\\
 &=(\lambda-c)q<0.
\end{aligned}
\]
Thus $(S,\F_S,c\mathfrak a_S,t)$ is not lc in this case as well.

The contradiction proves that $\cA_c$ is lc near $S$ for every $c<c_S$, and hence $c_X\geq c_S$. Combining the two inequalities gives $c_X=c_S$.
\end{proof}

\subsection{Graded sequences}

We now pass from a single ideal to a graded sequence.

\begin{corollary}
\label{cor:F-compatible-IOA-graded}
We keep the assumptions of Theorem \ref{thm:F-compatible-IOA-finite}, and let $\mathfrak a_\bullet=\{\mathfrak a_m\}_{m\geq 1}$ be any graded sequence of ideals such that no $\mathfrak a_m$ vanishes identically along $S$. Then
\[
        \lct^{[t]}_{X,\F;(1-t)S} (\mathfrak a_\bullet)  = \lct^{[t]}_{S,\F_S} (\mathfrak a_{\bullet,S}).
\]
\end{corollary}

\begin{proof}
For every $m\geq1$, Theorem \ref{thm:F-compatible-IOA-finite} gives
\[
        \lct^{[t]}_{X,\F;(1-t)S}(\mathfrak a_m) =  \lct^{[t]}_{S,\F_S}((\mathfrak a_m)_S).
\]
Using Lemma \ref{lem:finite-level-mixed-lct-approx} on both sides, we obtain
\[
\begin{aligned}
        \lct^{[t]}_{X,\F;(1-t)S} (\mathfrak a_\bullet) &= \sup_m m\, \lct^{[t]}_{X,\F;(1-t)S}(\mathfrak a_m) \\
        &= \sup_m m\, \lct^{[t]}_{S,\F_S}((\mathfrak a_m)_S) \\
        &= \lct^{[t]}_{S,\F_S} (\mathfrak a_{\bullet,S}).
\end{aligned}
\]
\end{proof}

For completeness, we explain how the base ideals used in Sections \ref{sec: theta red} and \ref{sec: s-completeness} fit in the above setting. We let $f:(\X,\F,t)\to\Spec R$ be an admissible DVR family, choose a Cartier multiple $L_r=-rK^{[t]}_{\X/R,\F}$, and put $\mathcal R_m=f_*\cO_{\X}(mL_r)$. We recall that a multiplicative filtration of the relative section ring is a decreasing left-continuous filtration of every $\mathcal R_m$ by $R$-submodules, compatible with multiplication, with base ideal (for $a\in\R$)
\begin{equation}\label{eq:relative-raw-base-ideal}
    I^{(a)}_{m,\cG} := \operatorname{Im}\!\left( \cG^{\lceil ma\rceil}\mathcal R_m\otimes \cO_{\X}(-mL_r)\longrightarrow\cO_{\X} \right).
\end{equation}

Multiplicativity shows that $\{I^{(a)}_{m,\cG}\}_{m\geq1}$ is a graded sequence. We define its fibre ideal by actual restriction,
\[
 I^{(a)}_{m,\cG,\kappa} := I^{(a)}_{m,\cG}\cO_{\X_\kappa}.
\]

\begin{theorem}
\label{thm:lct-special-fibre-filtrations}
Let $f:(\X,\F,t)\to \Spec R$ be an admissible $t$-K-semistable DVR family, and let $\X_\kappa$ be the special fibre. Let $\cG^\bullet$ be any linearly bounded multiplicative filtration on the relative anti-adjoint section ring, and assume that the base ideals $I^{(a)}_{\bullet,\cG}=\{I^{(a)}_{m,\cG}\}_{m\geq1}$ do not vanish identically along $\X_\kappa$. Then
\[
        \lct^{[t]}_{\X,\F; (1-t)\X_\kappa}  \left(I^{(a)}_{\bullet,\cG}\right) =        \lct^{[t]}_{\X_\kappa,\F_\kappa} \left(I^{(a)}_{\bullet,\cG,\kappa}\right).
\]
\end{theorem}

\begin{proof}
Proposition \ref{prop: families satisfy IOA} gives that $\X$ is potentially klt near $\X_\kappa$ and that $(\X,\F,(1-t)\X_\kappa,t)$ is plt, hence lc, there. Therefore all hypotheses of Corollary \ref{cor:F-compatible-IOA-graded} are satisfied with $S=\X_\kappa$, and we obtain the desired equality.
\end{proof}

\section{Relative extraction and finite generation}\label{sec:relative-extraction-fg}

In this section, we prove the extraction and finite generation statement used in the valuative arguments. The result is the adjoint foliated analogue of \cite[Corollaries 1.68 and 1.70]{Xu2025}.

Throughout this section, $R$ denotes a DVR essentially of finite type over an algebraically closed field $k$ of characteristic zero, with residue field $\kappa$, and $f\colon (\X,\F,t)\longrightarrow \Spec R$ is a projective admissible DVR family. We retain this convention in Sections \ref{sec: theta red} and \ref{sec: s-completeness}.

We further assume that $\X$ is potentially klt. This is automatic for an admissible $t$-K-semistable family by Proposition \ref{prop: families satisfy IOA}. We set $L:=-K^{[t]}_{\X/R,\F}$ and choose a sufficiently divisible integer $r>0$ such that $L_r:=rL$ is Cartier. For $q\geq 1$ and an effective $\Q$-Cartier divisor $D\sim_{\Q,R}qL_r$ we consider the adjoint foliated structures $\cA_0:= \bigl(\X,\F,(1-t)\X_\kappa,t\bigr)$ and $\cA_D := \left(\X,\F, (1-t)\X_\kappa+\frac{1}{rq}D, t\right)$. Since $\X_\kappa$ is principal over $\Spec R$, we have
\[
K_{\cA_D}\sim_{\Q,R} -L+\frac{1}{rq}D \sim_{\Q,R}0.
\]

We first record a consequence of the definition of a
$t$-K-semistable family.

\begin{lemma}
\label{lem:A0-unique-lc-place}
If the family is $t$-K-semistable, then $\cA_0$ is lc and
$\X_\kappa$ is its unique lc place.
\end{lemma}

\begin{proof}
By Proposition \ref{prop: families satisfy IOA}, the adjoint foliated structure $\cA_0$ is plt on an open neighbourhood of $\X_\kappa$.

Let $V\subseteq \X$ be a nonempty closed subvariety. Since $\X\to\Spec R$ is proper, its image is a nonempty closed subset of $\Spec R$, and therefore contains the closed point. Hence $V$ meets $\X_\kappa$. In particular, the centre on $\X$ of every prime divisor over $\X$ meets the neighbourhood on which $\cA_0$ is plt. It follows that $\X_\kappa$ is the unique lc place of $\cA_0$.
\end{proof}

\subsection{Non-terminal extraction}

We first prove an ample one-place extraction statement.

\begin{proposition}
\label{prop:non-terminal-extraction}
Let $\cA=(\X,\F,B,t)$ be an lc algebraically integrable adjoint foliated structure on a potentially klt variety $\X$. Assume that $\cA$ has a unique lc place $S$, where $S$ is a prime divisor on $\X$. Let $G$ be an exceptional prime divisor over $\X$ such that
\[
0<A_{\cA}(G)< t\varepsilon_\F(G)+(1-t).
\]
Then there exists a projective birational morphism $\rho\colon \cZ\longrightarrow \X$ such that:
\begin{enumerate}
\item $\cZ$ is normal and $\Q$-factorial;
\item $G$ appears as a $\rho$-exceptional prime divisor on $\cZ$;
\item $G$ is the unique $\rho$-exceptional prime divisor;
\item $-G$ is $\rho$-ample;
\item the induced foliation $\F_{\cZ}$ is algebraically integrable
and relative over $\Spec R$.
\end{enumerate}
\end{proposition}

\begin{proof}
We choose a sufficiently small rational number $\epsilon>0$ and set $\cA':=\cA-\epsilon S$. Since $S$ is the unique lc place of $\cA$, the adjoint foliated structure $\cA'$ is klt. Moreover,
\[
A_{\cA'}(G) = A_{\cA}(G)+\epsilon\ord_G(S).
\]
After decreasing $\epsilon$, we may therefore assume that $A_{\cA'}(G) < t\varepsilon_\F(G)+(1-t)$ or equivalently, $a(G,\cA')<0$.

By \cite[Theorem 2.2.3]{CHLMSSX25}, there exists a projective birational morphism $\rho_0\colon \cZ_0\longrightarrow \X$ such that $\cZ_0$ is $\Q$-factorial and $G$ is its unique exceptional prime divisor.

Let $\cA'_{\cZ_0}$ be the crepant transform of $\cA'$. Since $\cA'$ is klt, so is $\cA'_{\cZ_0}$. The coefficient of $G$ in the boundary of $\cA'_{\cZ_0}$ is
\[
-a(G,\cA')= t\varepsilon_{\F_{\cZ_0}}(G)+(1-t)-A_{\cA'}(G)>0.
\]
We now choose a rational number $0<\eta<-a(G,\cA')$ and decrease the coefficient of $G$ by $\eta$. The resulting adjoint foliated structure $\cC_\eta$ is klt and satisfies $K_{\cC_\eta} \sim_{\Q,\X} -\eta G$.

As in the proof of \cite[Corollary 1.68(ii)]{Xu2025}, the divisor $-G$ is big over $\X$. Hence $K_{\cC_\eta}$ is big over $\X$, and \cite[Theorem 2.1.1]{CHLMSSX25} gives a relative good minimal model and a relative canonical model for $\cC_\eta/\X$.

Applying the argument of \cite[Corollary 1.68(ii)]{Xu2025} to this canonical model, and using that $G$ is the unique $\rho_0$-exceptional prime divisor, we obtain a projective birational morphism $\rho\colon \cZ\longrightarrow \X$ whose unique exceptional prime divisor is $G$ and for which $-G$ is $\rho$-ample. The resulting variety $\cZ$ is normal and $\Q$-factorial.

The induced foliation on $\cZ$ is the saturated birational transform of $\F$. It remains algebraically integrable because algebraic integrability is invariant under birational transform of the induced foliation. It is relative over $\Spec R$, since the entire construction is performed over $\Spec R$.
\end{proof}

\subsection{Finite generation after extraction}

We now apply Proposition \ref{prop:non-terminal-extraction} to a
suitable convex perturbation of $\cA_0$ and $\cA_D$.

\begin{proposition}
\label{prop:extracted-Rees-fg}
Assume that:
\begin{enumerate}
\item $\cA_0$ is lc and $\X_\kappa$ is its unique lc place;
\item $\cA_D$ is lc on a neighbourhood of $\X_\kappa$;
\item $G$ is an exceptional prime divisor over $\X$ satisfying $A_{\cA_0}(G)>0$ and $A_{\cA_D}(G) < t\varepsilon_\F(G)+(1-t)$.
\end{enumerate}
Then there exists a projective birational morphism $\rho\colon \cZ\longrightarrow \X$ satisfying conclusions {\rm(1)--(5)} of Proposition \ref{prop:non-terminal-extraction}, and the bigraded algebra
\[
R_G(\X,L_r) := \bigoplus_{a\geq0}\bigoplus_{p\geq0} H^0\!\left( \cZ, \cO_{\cZ}(a\rho^*L_r-pG) \right)
\]
is finitely generated.
\end{proposition}

\begin{proof}
Since $\X\to\Spec R$ is proper, an open neighbourhood of the whole closed fibre is all of $\X$. Indeed, the proper image of a nonempty closed complement would be a closed subset of $\Spec R$ containing the generic point but not the closed point. It follows that $\cA_D$ is lc on all of $\X$.

For $0<s<1$, we set
\[
\cA_s := s\cA_0+(1-s)\cA_D = \left( \X,\F, (1-t)\X_\kappa+\frac{1-s}{rq}D, t\right).
\]
By \cite[Lemma 3.27]{CHLMSSX25}, the adjoint foliated structure $\cA_s$ is lc. Moreover,
\[
K_{\cA_s} \sim_{\Q,R} sK_{\cA_0}+(1-s)K_{\cA_D} \sim_{\Q,R} -sL.
\]
For every prime divisor $E$ over $\X$, the mixed discrepancy is affine in $s$, so 
\[
A_{\cA_s}(E) = sA_{\cA_0}(E)+(1-s)A_{\cA_D}(E).
\]
In particular,
\[
A_{\cA_s}(G) = sA_{\cA_0}(G)+(1-s)A_{\cA_D}(G).
\]
Since $A_{\cA_0}(G)>0$ and $0\leq A_{\cA_D}(G)< t\varepsilon_\F(G)+(1-t)$ we may choose $s>0$ sufficiently small so that
\begin{equation}
\label{eq:As-G-interior}
0<A_{\cA_s}(G)<t\varepsilon_\F(G)+(1-t).
\end{equation}

We next show that $\X_\kappa$ is the unique lc place of $\cA_s$.
Since $\X_\kappa$ is invariant, log canonicity of $\cA_D$ gives
\[
(1-t)+\frac{1}{rq}\ord_{\X_\kappa}(D) \leq 1-t.
\]
Thus $\ord_{\X_\kappa}(D)=0$, and thus $\X_\kappa$ is an lc place of $\cA_D$. It is therefore an lc place of $\cA_s$.

Conversely, suppose that $E$ is an lc place of $\cA_s$. Since
$\cA_0$ and $\cA_D$ are lc, both terms in
\[
0 = A_{\cA_s}(E) = sA_{\cA_0}(E)+(1-s)A_{\cA_D}(E)
\]
are nonnegative. Hence $A_{\cA_0}(E)=A_{\cA_D}(E)=0$. Since $\X_\kappa$ is the unique lc place of $\cA_0$, we obtain $E=\X_\kappa$.

We may now apply Proposition \ref{prop:non-terminal-extraction} to $\cA_s$, with unique lc place $\X_\kappa$, and to the divisor $G$. This gives a projective birational morphism $\rho\colon \cZ\longrightarrow \X$ such that $\cZ$ is normal and $\Q$-factorial, $G$ is its unique $\rho$-exceptional prime divisor, and $-G$ is $\rho$-ample. Let $(\cA_s)_{\cZ}$ denote the crepant transform of $\cA_s$, and let $S_{\cZ}$ be the strict transform of $\X_\kappa$. Then $S_{\cZ}$ is the unique lc place of $(\cA_s)_{\cZ}$.

The prime divisor $S_{\cZ}$ is distinct from $G$, and hence is not contained in $G$. By \cite[Lemma 3.25]{CHLMSSX25}, there exists a rational number $\delta_0>0$ such that $\cB_{s,\delta_0} := (\cA_s)_{\cZ}+\delta_0G$ is lc. We now choose
\[
0<\delta< \min\{\delta_0,A_{\cA_s}(G)\}.
\]
Then
\[
\cB_{s,\delta} = \left(1-\frac{\delta}{\delta_0}\right)(\cA_s)_{\cZ} + \frac{\delta}{\delta_0}\cB_{s,\delta_0}
\]
is lc by \cite[Lemma 3.27]{CHLMSSX25}.

If $E$ is an lc place of $\cB_{s,\delta}$, then $E$ is an lc place of $(\cA_s)_{\cZ}$. Hence $E=S_{\cZ}$. Thus $S_{\cZ}$ is the unique lc place of $\cB_{s,\delta}$.

Since $K_{(\cA_s)_{\cZ}} \sim_{\Q,R} -s\rho^*L$, we have
\[
-K_{\cB_{s,\delta}} \sim_{\Q,R} s\rho^*L-\delta G.
\]
Because $L$ is ample over $\Spec R$ and $-G$ is $\rho$-ample, the divisor on the right is ample over $\Spec R$ after decreasing $\delta$, if necessary.

Lowering the coefficient of the unique lc place $S_{\cZ}$ by a sufficiently small positive rational number produces a klt adjoint foliated structure. Therefore, by \cite[Theorem 3.9]{CHLMSSX25}, the variety $\cZ$ is potentially klt.

We may now apply \cite[Theorem 2.4.2]{CHLMSSX25} to $\cB_{s,\delta}/\Spec R$. It follows that $\cZ$ is of Fano type over $\Spec R$. Equivalently, there exists an effective $\Q$-divisor $\Delta_{\cZ}$ such that $(\cZ,\Delta_{\cZ})$ is klt and $-\bigl(K_{\cZ}+\Delta_{\cZ}\bigr)$ is ample over $\Spec R$.

Applying \cite[Corollary 1.70]{Xu2025} to the divisor classes $\rho^*L_r$ and $-G$ gives finite generation of
\[
\bigoplus_{a,p\geq0} H^0\!\left( \cZ, \cO_{\cZ}(a\rho^*L_r-pG) \right),
\]
completing the proof.
\end{proof}

\subsection{Finite generation and extraction theorem}

Combining the preceding results gives the form of extraction and
finite generation that will be used later.

\begin{theorem}[Relative non-terminal extraction and finite generation]
\label{thm:relative-non-terminal-extraction-fg}
Let $f\colon(\X,\F,t)\longrightarrow\Spec R$ be an admissible DVR family. Assume:
\begin{enumerate}
\item $\X$ is potentially klt;
\item $T_\F\subseteq T_{\X/R}$;
\item $\F$ is algebraically integrable;
\item $L:=-K^{[t]}_{\X/R,\F}$ is $f$-ample and $\Q$-Cartier, and $L_r=rL$ is Cartier;
\item
$\cA_0 = \bigl(\X,\F,(1-t)\X_\kappa,t\bigr)$ is lc with $\X_\kappa$ as its unique lc place;
\item $D\geq0$ is an effective $\Q$-Cartier divisor satisfying $D\sim_{\Q,R}qL_r$;
\item
\[
\cA_D = \left( \X,\F, (1-t)\X_\kappa+\frac{1}{rq}D, t \right)
\]
is lc on a neighbourhood of $\X_\kappa$;
\item $G$ is an exceptional prime divisor over $\X$ such that $A_{\cA_D}(G) < t\varepsilon_\F(G)+(1-t)$ and $A_{\cA_0}(G)>0$.
\end{enumerate}
Then there exists a projective birational morphism $\rho\colon \cZ\longrightarrow \X$ such that:
\begin{enumerate}
\item $\cZ$ is normal and $\Q$-factorial;
\item $G$ appears as a $\rho$-exceptional prime divisor on $\cZ$;
\item $G$ is the unique $\rho$-exceptional prime divisor;
\item $-G$ is $\rho$-ample;
\item the induced foliation $\F_{\cZ}$ is algebraically integrable
and relative over $\Spec R$;
\item the bigraded algebra
\[
R_G(\X,L_r):= \bigoplus_{a\geq0}\bigoplus_{p\geq0} H^0\!\left(\cZ, \cO_{\cZ}(a\rho^*L_r-pG) \right)
\]
is finitely generated.
\end{enumerate}
\end{theorem}

\begin{proof}
This is Proposition \ref{prop:extracted-Rees-fg}.
\end{proof}

\begin{remark}
Suppose that $G$ satisfies the hypotheses of Theorem \ref{thm:relative-non-terminal-extraction-fg} and that $c_\X(G)\subseteq\X_\kappa$. Then $G$ is invariant by Lemma \ref{lem:vertical-divisors-invariant}. Hence the theorem applies to every vertical exceptional divisor satisfying the two discrepancy inequalities in item {\rm(8)}.
\end{remark}

\section{\texorpdfstring{Valuative criterion of $\Theta$-reductivity}{Valuative criterion of Theta-reductivity}}\label{sec: theta red}

We now prove the valuative criterion for $\Theta$-reductivity. The argument follows the strategy presented in \cite[\S 8.2.1]{Xu2025}, itself based on \cite{ABHLX20}.

Let $R$ be a DVR essentially of finite type over $k$ with fraction field $K$, and let $\widetilde f^\circ:(\X^\circ,\L^\circ) \to \Theta_R\setminus 0$ be a polarised family. Equivalently, this consists of a polarised family $(\X,\L)\to \Spec R$, a polarised family $(\X_K,\L_K)\to \Theta_K$, and an identification of the generic fibre of $(\X,\L)$ with the fibre over $1\in \Theta_K$. For each sufficiently divisible $m$, we define $V_m:=H^0(\X,\mathcal O_{\X}(m\L))$ and $V_{K,m}:= H^0(\X_K,\mathcal O_{\X_K}(m\L_K))$. The family over $\Theta_K$ induces a decreasing $\mathbb Z$-filtration $\cG^\bullet V_{K,m}$. We define $\cG^pV_m:=V_m\cap \cG^pV_{K,m}\subset V_{K,m}$. 

\begin{lemma}[{\cite[Corollary 8.17]{Xu2025}}]
\label{lem:theta-extension-fg}
The extension of $\widetilde f^\circ:(\X^\circ,\L^\circ)\to \Theta_R\setminus 0$ to a family $\widetilde f:(\X,\L)\to \Theta_R$ of flat polarised projective varieties is unique. Moreover, such an extension exists as a family of flat polarised projective schemes if and only if $\bigoplus_{m,p}\cG^pV_m$ is finitely generated.
\end{lemma}

\begin{definition}
Let $(X,\F,t)$ be an adjoint Fano foliated structure and let $L$ be an ample $\Q$-line bundle with $L\sim_\Q -K^{[t]}_{X,\F}$. A \emph{special $t$-K-semistable degeneration} of $(X,\F,t)$ is a normal $\F$-compatible special test configuration $(\Y,\F_\Y,\L_\Y)\to \A^1$ of $(X,\F,L)$ such that the central fibre is $t$-K-semistable and $\DF^{[t]}(\Y,\F_\Y,\L_\Y)=0$.
\end{definition}

\begin{lemma}
\label{lem:mixed-two-step-degeneration}
Let $(X,\F,L)$ be an adjoint Fano foliated structure and let $(\X_1,\F_1,\mathcal L_1)\to\A^1$ be a normal $\F$-compatible special test configuration with integral central fibre $(Y,\F_Y,L_Y)$. Let $(\X_2,\F_2,\mathcal L_2)\to\A^1$ be a normal $\F_Y$-compatible test configuration with (normal) integral central fibre $(Z,\F_Z,L_Z)$, equivariant for the $\G_m$-action on $Y$ induced by $\X_1$. Write $\xi$ and $\xi_0$ for the two commuting one-parameter subgroups acting on $Z$.

For every sufficiently small positive rational number $\varepsilon$, after a finite base change there is a normal $\F$-compatible test configuration $\X_\varepsilon$ of $(X,\F,L)$ with central fibre $(Z,\F_Z,L_Z)$ and $\G_m$-action on $Z$ given by $N(\xi+\varepsilon\xi_0)$ for some $N\in\mathbb Z_{>0}$. Moreover,
\[
 \DF^{[t]}(\X_\varepsilon) =N\left( \DF^{[t]}(\X_1) +\varepsilon\DF^{[t]}(\X_2)  \right).
\]
\end{lemma}

\begin{proof}
We use the lexicographic construction of \cite[Lemma 5.38]{Xu2025}, along with the linearity statements of \cite[Proposition 5.37]{Xu2025}. Let $v$ be the valuation over $X$ induced by $\X_1$. The first test configuration $\X_1$ identifies 
\[\operatorname{gr}_vR = \bigoplus_{m\in r\cdot \N, a\in \N}R_{m,a}\]
where $R_{m,a}:= \cG_v^aR_m/\cG_V^{>a}R_m$, $\cG_v$ is the filtration associated to the valuation $v$, and $R = \bigoplus_{m\in r\cdot \N}H^0(-mK_{X,\F}^{[t]})$. Furthermore, $\X_{1,0} = \operatorname{Proj}\operatorname{gr}_v$.

The second test configuration $\X_2$ gives a finitely generated filtration $v_0$ of this associated graded ring. We define the lexicographically ordered two-step filtration
\[
 s\longmapsto \bigl(v(s),v_0(\operatorname{in}_v(s))\bigr)
\]
on $R(X,L)$. This filtration gives the identification
\[\operatorname{gr}_sR = \bigoplus_{m\in r\cdot \N, (a,b)\N\times \N}R_{m,a,b} = \operatorname{gr}_{v_0}(\operatorname{gr}_vR).\]
Here
\[R_{m,a,b}:= (R_{m,a})_{\geq b}/(R_{m,a})_>b.\]
If we let $\xi$ be the $\G_m$-action induced by the action of the test configutration $\X_1$, and by $\xi'$ the $\G_m$-action induced by the test configuration $\X_2$, then, using the above identifications, an identical argument as in \cite[Lemma 5.38]{Xu2025}, shows that the test configuration $\X_\epsilon$ is induced by the filtration associated to $s$, and that the co-weight of the action on $Z$ induced by the test configuration $\X_\epsilon$ is $N(\xi+\epsilon\xi')$.

On the common birational locus the product foliation has a rational transform on this model. Saturating that transform inside the relative tangent sheaf, as in Lemma \ref{lem:compatible-flag-blowup}, gives the unique $\F$-compatible relative foliation. Finally, arguing as in \cite[proof of Proposition 5.37]{Xu2025}, we obtain 
\[\DF^{[t]}(\Y,\F_Y) = \DF(Z,\F_Z;N(\xi+\epsilon\xi'))=N\left(\DF^{[t]}(Z,\F_Y;\xi)+\epsilon\DF^{[t]}(Z,\F_Y;\xi')\right)\]
i.e.  the mixed Donaldson--Futaki invariant is the weight of the mixed CM intersection class. On the fixed central fibre $Z$ this weight is linear in the cocharacter, so evaluation on $N(\xi+\varepsilon\xi_0)$ proves the necessary formula 
\end{proof}

\begin{corollary}\label{cor:zero-DF-special-limit-semistable}
Let $(X,\F,t)$ be $t$-K-semistable, and let $\X_1$ be an $\F$-compatible special test configuration with $\DF^{[t]}(\X_1)=0$. Then its central fibre is $t$-K-semistable, i.e. it is a special $t$-K-semistable degeneration.
\end{corollary}

\begin{proof}
Suppose that the central fibre $Y$ is not $t$-K-semistable. By \cite[Theorem 6.9 and Corollary 6.10]{Pap26}, applied equivariantly for the commuting $\G_m$-action induced by $\X_1$, there is an $\F_Y$-compatible special test configuration $\X_2$ of $Y$ with negative mixed Donaldson--Futaki invariant. Then, Lemma \ref{lem:mixed-two-step-degeneration} produces an $\F$-compatible test configuration of $X$ with negative mixed Donaldson--Futaki invariant, contradicting $t$-K-semistability.
\end{proof}

\begin{theorem}
\label{thm:adjoint-foliated-theta-extension}
Let $R$ be a DVR with fraction field $K$ and residue field $\kappa$. Let $(\X,\F,t)\to \Spec R$ be an admissible family of $t$-K-semistable adjoint Fano foliated structures, and set $\L:=-K^{[t]}_{\X/R,\F}$. Let $(\Y_K,\F_{\Y_K},\M_K) \to \A^1_K$ be a special $t$-K-semistable degeneration of $(X_K,\F_K,L_K)$. Then, after replacing $R$ by a finite extension, there exists a unique
admissible family $(\Y,\F_{\Y},\M) \to \A^1_R$ such that
\[
        (\Y,\F_{\Y},\M)\times_R K \simeq (\Y_K,\F_{\Y_K},\M_K),
\]
and
\[
        (\Y,\F_{\Y},\M)|_{\{1\}\times \Spec R} \simeq (\X,\F,\L).
\]
Moreover, every geometric fibre of $(\Y,\F_{\Y},\M) \to \A^1_R$ is $t$-K-semistable, and $\M \sim_{\Q,\A^1_R} -K^{[t]}_{\Y/\A^1_R,\F_{\Y}}$.
\end{theorem}

\begin{proof}
By admissibility, $T_\F\subset T_{\X/R}$, the induced fibre foliations are algebraically integrable, and $\L=-K^{[t]}_{\X/R,\F}$ is $f$-ample and $\Q$-Cartier. Since the fibres are $t$-K-semistable, Theorem \ref{thm: k-stab implies klt} implies that $(\X_\kappa,\F_\kappa,t)$ is klt. Hence the underlying fibre $\X_\kappa$ is potentially klt. By Corollary \ref{cor:dvr-total-potentially-klt}, the total space $\X$ is potentially klt near $\X_\kappa$. Finally, by Theorem \ref{thm:dvr-fibre-ioa}, the adjoint structure $\cA_0=(\X,\F,(1-t)\X_\kappa,t)$ is plt near $\X_\kappa$.

We choose a sufficiently divisible integer $r>0$ such that $L_r:=-rK^{[t]}_{\X/R,\F}$ is Cartier and the section algebra $\bigoplus_{m\ge 0} f_*\mathcal O_{\X}(mL_r)$ is compatible with base change. To avoid extra factors, we continue to write $L=-K^{[t]}_{\X/R,\F}$, and all degrees $m$ below are taken sufficiently divisible so that $mL$ is Cartier. Let $E_m:=f_*\mathcal O_{\X}(mL)$. This is a finite free $R$-module with $E_m\otimes_R K\simeq H^0(\X_K,mL_K)$, $E_m\otimes_R\kappa \simeq H^0(\X_\kappa,mL_\kappa)$.

Since the degeneration of the generic fibre is nontrivial and special, the special-test-configuration/divisorial-valuation correspondence \cite[Proposition 4.3, Proposition 4.7 and Corollary 4.8]{Pap26} gives a special divisor $G_K$ over $\X_K$. Its induced filtration is the valuation filtration is
\[
        \cG_K^pH^0(\X_K,mL_K) = \{s\in H^0(\X_K,mL_K)\mid \ord_{G_K}(s)\ge p\}.
\]
As in \cite[Theorem 8.18]{Xu2025}, we define an $R$-filtration by
\[
        \cG^pV_m := \{s\in V_m\mid \ord_{G_K}(s|_{\X_K})\ge p\}.
\]
We let $\cG_\kappa^\bullet$ be the induced filtration on the special fibre
\[
        \cG_\kappa^p(V_m\otimes_R\kappa) := \operatorname{Im}  \left( \cG^pV_m\to V_m\to V_m\otimes_R\kappa  \right).
\]
If the generic special degeneration is trivial, then we take the product extension over $\A^1_R$. Thus we may assume from now on that the generic degeneration is nontrivial.

Since the filtered pieces are $R$-lattices, the filtrations $\mathcal G_K^\bullet$ on $\bigoplus_m H^0(\X_K,mL_K)$ and $\mathcal G_\kappa^\bullet$ on $\bigoplus_m H^0(\X_\kappa,mL_\kappa)$ have the same weight measures in each degree, hence the same Duistermaat--Heckman measure (compare with \cite[Equation (8.8) and Theorem 8.18]{Xu2025}). Hence
\[
        S(\cG_K)=S(\cG_\kappa).
\]

The degeneration of the generic fibre is special and satisfies $\DF^{[t]} (\Y_K,\F_{\Y_K},\M_K)=0$. By \cite[Corollary 7.10]{Pap26}, we also have $\Ding^{[t]}(\cG_K)=0$, hence $\mu^{[t]}(\cG_K)=S(\cG_K)$.

For every real $a$, Theorem \ref{thm:lct-special-fibre-filtrations} applies to the base ideals $I^{(a)}_{\bullet,\cG}$ and their actual restrictions to the special fibre, and gives
\[
\lct^{[t]}_{\X,\F;(1-t)\X_\kappa} (I^{(a)}_{\bullet,\cG}) = \lct^{[t]}_{\X_\kappa,\F_\kappa} (I^{(a)}_{\bullet,\cG,\kappa}).
\]
Thus, $\lct^{[t]}_{\X,\F;(1-t)\X_\kappa} (I^{(a)}_{\bullet,\cG}) $ is at least $1$, and restriction to the generic fibre gives $\lct^{[t]}_{\X_K,\F_K}(I^{(a)}_{\bullet,\mathcal G_K})\ge 1$. Taking suprema over such $a$ gives
\[
        \mu^{[t]}(\mathcal G_\kappa)\leq \mu^{[t]}(\mathcal G_K),
\]
and thus
\[
        \mu^{[t]}(\cG_\kappa) \le \mu^{[t]}(\cG_K) = S(\cG_K) = S(\cG_\kappa).
\]

On the other hand, the special fibre $(\X_\kappa,\F_\kappa,t)$ is $t$-K-semistable and hence $t$-Ding semistable. Theorem \ref{thm:mixed-ding-compatible-filtrations} gives $\Ding^{[t]}(\cG_\kappa)\ge 0$, which implies that $\mu^{[t]}(\cG_\kappa)\ge S(\cG_\kappa)$. Therefore
\[
        \mu^{[t]}(\cG_\kappa) = S(\cG_\kappa) = S(\cG_K) = \mu^{[t]}(\cG_K).
\]
We denote this common value by $\mu$.

Since $G_K$ is special, $\lambda_{\max}(\cG_K)>\mu^{[t]}(\cG_K)$ (compare with \cite[Equation (3.30)]{Xu2025}). By equality of Duistermaat--Heckman measures, $\lambda_{\max}(\cG_\kappa) =\lambda_{\max}(\cG_K)$, and hence $\lambda_{\max}(\cG_\kappa)>\mu^{[t]}(\cG_\kappa)$. By Lemma \ref{lem:mixed-Xu-346}, the function
\[
        a\longmapsto \lct^{[t]} \left(\X_\kappa,\F_\kappa; I^{(a)}_\bullet(\cG_\kappa) \right)
\]
is continuous in a neighbourhood of $a=\mu^{[t]}(\cG_\kappa)$. Therefore, after choosing $0<\varepsilon<1-t$ sufficiently small, we have
\[
        \lct^{[t]} \left(\X_\kappa,\F_\kappa; I^{(\mu-\varepsilon)}_\bullet(\cG_\kappa) \right) >1.
\]
By Theorem \ref{thm:lct-special-fibre-filtrations}, and Lemma \ref{lem:finite-level-mixed-lct-approx}, after replacing $m$ by a sufficiently large, divisible integer, we have
\[
m\cdot \lct^{[t]}_{\X,\F;(1-t)\X_\kappa} \bigl(I_{m,(\mu-\varepsilon)m}(G)\bigr)>1.
\]
Equivalently,
\[
(\X,\F,(1-t)\X_\kappa+ \frac{1}{m}I_{m,(\mu-\varepsilon)m}(G),t)
\]
is lc near $\X_\kappa$. We now choose a general section $s\in \cG^{(\mu-\varepsilon)m}V_m$ and let $D=(s=0)$. By taking a common log resolution of the base ideal and applying Bertini (cf.  \cite[Theorem 3.28(1)]{CHLMSSX25})  to the moving part, as in the proof of Lemma \ref{lem:general-divisorial-representative}, the mixed log canonicity of the ideal pair implies that the adjoint foliated structure $\left(\X,\F,(1-t)\X_\kappa+\frac1mD,t\right)$ is lc near $\X_\kappa$.

Let $G$ be the divisor over $\X$ obtained by extending $G_K$. After replacing the model if necessary, let $G\subset \Y$ be the closure of the divisor $G_K\subset \Y_K$. Then $G$ is a horizontal prime divisor over $\operatorname{Spec}R$. Since $D\in \cG^{(\mu-\varepsilon)m}V_m$ we have $\ord_G(D)\ge (\mu-\varepsilon)m$.

\begin{claim} \label{claim:generic-discrepancy-extension}
\[
A^{[t]}_{\X,\F;(1-t)\X_\kappa}(G) = A^{[t]}_{\X_K,\F_K}(G_K)=\mu^{[t]}(\cG_K).
\]
\end{claim}

\begin{proof}[Proof of Claim]
We choose a normal model $\rho:\Y\to \X$ on which $G$ appears as a prime divisor, and such that the restriction of $G$ to the generic fibre is the divisor $G_K$ over $\X_K$. Since $G$ dominates $\operatorname{Spec}R$, it is horizontal over $R$. Hence $\operatorname{ord}_G(\X_\kappa)=0$. Therefore
\[
A^{[t]}_{\X,\F;(1-t)\X_\kappa}(G) = A^{[t]}_{\X,\F}(G).
\]

We make the comparison directly for the mixed canonical divisor, since admissibility does not require $K_{\X/R}$ and $K_{\F/R}$ to be separately $\Q$-Cartier. Choose $r>0$ as in Definition \ref{def:admissible-family-adjoint-foliated}. Pulling back the rank-one reflexive sheaf
\[
 \mathscr K_{r,t} =\left(\omega_{\X/R}^{[r(1-t)]} \widehat\otimes \omega_{\F/R}^{[rt]} \right)^{**}
\]
to $\Y$ and taking its reflexive hull gives the divisorial pullback of $rK^{[t]}_{\X/R,\F}$. The coefficient of the horizontal divisor $G$ in the resulting mixed discrepancy divisor is unchanged by restriction to the generic fibre, as restriction is defined at the generic point of $G$, where $R\setminus\{0\}$ is inverted, and $G|_{\Y_K}=G_K$. Moreover, $G$ is invariant for $\F_\Y$ if and only if $G_K$ is invariant for the restricted foliation. Thus both the mixed canonical coefficient and the $\varepsilon_\F$-term agree after restriction, and consequently
\[
A^{[t]}_{\X,\F}(G) = A^{[t]}_{\X_K,\F_K}(G_K).
\]
Since $\operatorname{ord}_G(\X_\kappa)=0$, we obtain the desired equality
\[
A^{[t]}_{\X,\F;(1-t)\X_\kappa}(G) =A^{[t]}_{\X_K,\F_K}(G_K).
\]

Finally, by the special test configuration/divisorial valuation correspondence, the special divisor $G_K$ is normalised so that
\[
A^{[t]}_{\X_K,\F_K}(G_K)=\mu^{[t]}(\cG_K).
\]
Therefore
\[
A_{\cA_0}(G)=A^{[t]}_{\X,\F;(1-t)\X_\kappa}(G)=\mu^{[t]}(\cG_K).
\]
\end{proof}

Therefore we obtain
\[
\begin{aligned}
        A^{[t]}_{\X,\F;(1-t)\X_\kappa+\frac{1}{m}D}(G) &=A^{[t]}_{\X,\F;(1-t)\X_\kappa}(G) - \frac{1}{m}\ord_G(D)  \\
        &\le\mu-(\mu-\varepsilon)\\
        &=\varepsilon<1-t \leq t\varepsilon_{\F}(G)+(1-t).
\end{aligned}
\]
Moreover, $A_{\cA_0}(G)=\mu>0$.

Since $c_{\X}(G)$ is horizontal and its closure meets $\X_\kappa$, any open neighbourhood of $\X_\kappa$ contains the generic point of $c_{\X}(G)$. Thus lc near $\X_\kappa$ gives lc in a neighbourhood of $c_{\X}(G)$. We write $m=rq$ for the fixed Cartier index $r$, and use Lemma \ref{lem:A0-unique-lc-place} for $\cA_0$. Suppose first that $G$ is exceptional over $\X$. Then all the assumptions of Theorem \ref{thm:relative-non-terminal-extraction-fg} are satisfied, and there exists a projective birational morphism $\rho:\cZ\to\X$ extracting $G$, with $-G$ $\rho$-ample, such that
\[
        R_G(\X,L) := \bigoplus_{m,p\ge 0} H^0\!\left(\cZ,m\rho^*\L-pG\right)
\]
is finitely generated. In this case, for $p\geq0$,
\[
\cG^pE_m=H^0\!\left(\cZ,\cO_{\cZ}(m\rho^*L-pG)\right).
\]
If $G$ is not exceptional, its centre is a prime divisor on $\X$. The Fano-type conclusion of \cite[Theorem 2.4.2]{CHLMSSX25}, applied after a small $\Q$-factorialisation, gives finite generation of the multisection algebra generated by $L$ and that prime divisor. 

Thus in either case the section algebra
\[
\mathcal R(\cG)= \bigoplus_{m\ge0}\bigoplus_{p\in\mathbb Z} \cG^pV_m\,u^{-p} \subset \bigoplus_{m\ge0}V_m\otimes_R K[u,u^{-1}]
\]
is finitely generated. 

By Lemma \ref{lem:theta-extension-fg}, this finite generation gives the unique flat polarised extension. It is realised by the normalised relative Proj
\[
\Y:=\operatorname{Proj}_{\A^1_R}\overline{\mathcal R(\mathcal G)},
\]
with relatively semiample line bundle $\M=\cO_\Y(1)$. The grading by $p$ gives the $\G_m$-action, so this is a test-configuration family over $R$. Since $G$ restricts to $G_K$ over the generic point, we have
\[
        (\Y,\M)\times_R K \simeq (\Y_K,\M_K).
\]
Over $1\in \A^1_R$, the Rees construction recovers $(\X,\L)$.

The normalised Rees construction is birational to $\X\times_R\A^1_R$ and is an isomorphism over the locus where the filtration is trivial. On their common function field, pull back the rational distribution $\F\times_R\A^1_R$ and let
\[
T_{\F_{\Y}} := \left(T_{\Y/\A^1_R}\cap  T_{\F\times_R\A^1_R}\otimes K(\Y)\right)^{\mathrm{sat}}.
\]
This is the unique saturated birational transform of the foliation. Its Lie bracket vanishes in the torsion-free quotient because it does so on the common birational open set; algebraic integrability is birationally invariant. The construction restricts to the prescribed foliations over $K$ and over $\{1\}\times\operatorname{Spec}R$. Thus $(\Y,\F_{\Y},\M) \to \A^1_R$ is $\F$-compatible. 

We now study the polarisation. By Claim \ref{claim:generic-discrepancy-extension}, the discrepancy shift in the divisorial filtration is $A^{[t]}_{\X,\F;(1-t)\X_\kappa}(G)=\mu$. For the sufficiently divisible index $r$, the reflexive anti-adjoint section algebra
\[
 \bigoplus_{q\geq0} h_*\mathcal O_{\mathcal Y} \left(-qrK^{[t]}_{\mathcal Y/\mathbb A^1_R,\F_{\mathcal Y}}\right)
\]
is therefore identified with the normalised divisorial algebra above, after the linear regrading $p\mapsto p+qr\mu$. This regrading is exactly a character twist and does not change the relative Proj or its underlying polarisation. The identification follows from  \cite[Proposition 4.3]{Pap26}. Thus,
\[
\mathcal M\sim_{\mathbb Q,\mathbb A^1_R} -K^{[t]}_{\mathcal Y/\mathbb A^1_R,\F_{\mathcal Y}}.
\]

It remains to check $t$-K-semistability of the fibres. Away from the closed origin $0_\kappa\in \A^1_R$, the fibres are either fibres of the original admissible $t$-K-semistable family or fibres of the given special $t$-K-semistable degeneration.

At $0_\kappa$, the equality $ \mu^{[t]} \cG_\kappa)=S(\cG_\kappa)$ shows that the induced $\F_\kappa$-compatible test configuration has $\Ding^{[t]}=0$ and is thus special by \cite[Theorem 7.7]{Pap26}. Furthermore, since the test configuration is $\F_\kappa$-compatible special, we have $\Ding^{[t]}=\DF^{[t]}=0$ by \cite[Corollary 7.10]{Pap26}. Its central fibre is $t$-K-semistable by Corollary \ref{cor:zero-DF-special-limit-semistable}. 

Thus every geometric fibre of $(\Y,\F_{\Y},\M)\to \A^1_R$ is $t$-K-semistable and hence normal. Uniqueness of the extension follows from Lemma \ref{lem:theta-extension-fg}, while uniqueness of the foliation follows since the relative foliation is its unique saturated extension. The mixed reflexive sheaf satisfies Koll\'ar's condition, hence the family is admissible in the sense of Section \ref{sec: admissible families}, proving the theorem.
\end{proof}

\section{\texorpdfstring{Valuative criterion of $S$-completeness/separatedness}{Valuative Criterion of S-completeness/separatedness}}\label{sec: s-completeness}

We will now verify $S$-completeness. Let $R$ be a DVR essentially of finite type over $k$ with uniformiser $\pi$, fraction field $K$, and residue field $\kappa$. We let
\[
        \ST_R:=\left[\Spec R[s,u]/(su-\pi)\, /\, \G_m\right],
\]
where $\G_m$ acts by $\lambda\cdot(s,u)=(\lambda s,\lambda^{-1}u)$. Let $0\in \ST_R$ be the closed stacky point and let
\[
        j:\ST_R^\circ:=\ST_R\setminus\{0\}\hookrightarrow \ST_R.
\]
The open stack $\ST_R^\circ$ is covered by the two open substacks $U_s:=\{s\neq 0\}/\G_m$ and $U_u:=\{u\neq 0\}/\G_m$, both canonically isomorphic to $\Spec R$. On $U_s$ we normalise $s=1$,
while on $U_u$ we normalise $u=1$. Their intersection is
\[
        U_s\cap U_u = \{s\neq 0,\ u\neq 0\}/\G_m \simeq \Spec K.
\]
Thus $\ST_R^\circ$ is obtained by gluing two copies of $\Spec R$ along their common generic point $\Spec K$.

Let
\[
        (\X,\F,\L)\to \Spec R,  \qquad (\X',\F',\L')\to \Spec R
\]
be two admissible DVR families, and suppose that their generic fibres are identified by an isomorphism
\[
        \varphi_K: (\X_K,\F_K,\L_K) \xrightarrow{\sim} (\X'_K,\F'_K,\L'_K).
\]
We place $(\X,\F,\L)$ over $U_s$ and $(\X',\F',\L')$ over $U_u$. The isomorphism $\varphi_K$ identifies the two families over $U_s\cap U_u\simeq \Spec K$, hence descent gives a family
\[
        (\cZ^\circ,\F^\circ,\M^\circ) \to \ST_R^\circ.
\]
Its restriction to $U_s$ is $(\X,\F,\L)$ and its restriction to $U_u$ is $(\X',\F',\L')$. The polarisation $\M^\circ$ is obtained by gluing $\L$ and $\L'$ via $\varphi_K^*\L'_K\simeq \L_K$, and the relative foliations glue because $\varphi_K$ identifies $T_{\F_K}\subset T_{\X_K}$ with $T_{\F'_K}\subset T_{\X'_K}$. Thus $T_{\F^\circ}\subset T_{\cZ^\circ/\ST_R^\circ}$ is a saturated relative foliation whose restrictions to the two charts are $T_\F$ and $T_{\F'}$, respectively.

After replacing $\L$ and $\L'$ by sufficiently divisible Cartier multiples, we assume that the anti-adjoint section rings are compatible with base change. For such $m$, we set $E_m:=H^0(\X,m\L)$ and $E'_m:=H^0(\X',m\L')$. Using the generic-fibre identification, we view both as $R$-lattices in the common $K$-vector space $H^0(\X_K,m\L_K)$. We define comparison filtrations
\[
        \cG^pE_m:=E_m\cap \pi^pE'_m, \qquad \cG'^{\,p}E'_m:=E'_m\cap \pi^pE_m.
\]
We let $R_\kappa:=\bigoplus_m H^0(\X_\kappa,m\L_\kappa)$, $R'_\kappa:= \bigoplus_m H^0(\X'_\kappa,m\L'_\kappa)$  and define the induced filtrations on the special fibres by
\[
        \cH^pH^0(\X_\kappa,m\L_\kappa)  := \operatorname{Im}
        \left(\cG^pE_m\to E_m\to E_m\otimes_R\kappa\right),
\]
and
\[
        \cH'^{\,p}H^0(\X'_\kappa,m\L'_\kappa) := \operatorname{Im}
        \left(\cG'^{\,p}E'_m\to E'_m\to E'_m\otimes_R\kappa\right).
\]
Equivalently, $\cH^p(E_m/\pi E_m)=\frac{\cG^pE_m+\pi E_m}{\pi E_m}$, and similarly for $\cH'^{\,p}$. When real shifts occur below, filtration indices are understood with the usual ceiling convention; we suppress this from the notation.

We first recall a key result regarding the existence and uniqueness of the extension.

\begin{lemma}[{\cite[Corollary 8.23]{Xu2025}}]\label{lem:underlying-STR-extension}
The glued polarised family $g^\circ:(\cZ^\circ,\M^\circ)\to \ST_R^\circ$ admits a unique extension as a flat polarised projective family $g:(\cZ,\M)\to \ST_R$ if and only if the graded $\kappa$-algebra
\[
\operatorname{Gr}_{\cH}R_\kappa:=\bigoplus_{m,p} \frac{\cG^pE_m+\pi E_m}     {\cG^{p+1}E_m+\pi E_m}
\]
is finitely generated. In this case
\[
        \cZ =
        \Proj_{\ST_R} \bigoplus_{m\ge 0} j_*g^\circ_*\OO_{\cZ^\circ}(m\M^\circ),
\]
and $\M=\OO_{\cZ}(1)$. Moreover, the fibre over the closed stacky point is $\cZ_0 = \Proj_\kappa \operatorname{Gr}_{\cH}R_\kappa$.
\end{lemma}

The same argument as in \cite[Proof of Theorem 8.31]{Xu2025}, shows that the extension $\cZ$ is normal. As such, we will assume $\Z$ is normal throughout, and justify this later fully in the proof of Lemma \ref{lem:admissibility-of-STR-extension}.

\begin{lemma}\label{lem:foliation-extends-over-STR}
The extension $(\cZ,\M)\to \ST_R$ carries a unique saturated relative foliation $T_{\F_\cZ}\subset T_{\cZ/\ST_R}$ extending $T_{\F^\circ}$.
\end{lemma}

\begin{proof}
Work on the standard normal two-dimensional atlas $B$ of $\ST_R$, and let $j_{\mathcal Z}:\mathcal Z^\circ\hookrightarrow\mathcal Z$ denote the induced open immersion. Inside $j_{\mathcal Z*}(T_{\mathcal Z^\circ/B^\circ})$, define
\[
T_{\F_{\mathcal Z}} := \left( T_{\mathcal Z/B}\cap j_{\mathcal Z*}T_{\F^\circ}\right)^{\mathrm{sat}}.
\]
Equivalently, this is the reflexive extension of $T_{\F^\circ}$ intersected with the relative tangent sheaf. It is a saturated relative subsheaf and its restriction to $\mathcal Z^\circ$ is $T_{\F^\circ}$.

On the standard atlas, the complement of the punctured base is the closed origin, of codimension two. Since $\mathcal Z\to B$ is flat, its inverse image has codimension two in $\mathcal Z$. Hence bracket closure may be checked on $\mathcal Z^\circ$ and extended reflexively. Thus $T_{\F_\cZ}$ is a relative foliation.  Since $\mathcal Z\setminus\mathcal Z^\circ$ has codimension two, a very general point of $\mathcal Z$ lies in $\mathcal Z^\circ$. The leaf through such a point agrees with a leaf of the original glued foliation, hence is Zariski open in an algebraic subvariety. Thus $\F_{\mathcal Z}$ is algebraically integrable. Uniqueness follows from saturation, since any saturated relative foliation extending $T_{\F^\circ}$ must agree with the above image on the codimension-one locus. Both saturated subsheaves are reflexive and agree on the complement of a codimension-two subset; hence they agree globally.
\end{proof}

\begin{lemma}
\label{lem:opposite-associated-graded}
The filtrations $\cH^\bullet$ and $\cH'^{\,\bullet}$ are linearly bounded and multiplicative. Moreover, there is an isomorphism of graded $\kappa$-algebras
\[
        \operatorname{Gr}_{\cH} R_\kappa \simeq \operatorname{Gr}_{\cH'} R'_\kappa
\]
which sends weight $p$ on the left to weight $-p$ on the right. In particular,
\[
        S(\cH')=-S(\cH).
\]
\end{lemma}

\begin{proof}
The lattice computation of \cite[Lemma 8.27]{Xu2025} gives canonical isomorphisms
\[
        \cG^pE_m = E_m\cap \pi^pE'_m \simeq \pi^{-p}E_m\cap E'_m = \cG'^{-p}E'_m.
\]
These isomorphisms are compatible with multiplication, and hence give an isomorphism of bigraded algebras with the weight grading reversed. The equality $S(\cH')=-S(\cH)$ follows from the weight reversal.
\end{proof}

\begin{lemma}
\label{lem:comparison-filtration-valuation-description}
Let $v:=\ord_{\X'_\kappa}$, viewed as a divisorial valuation over
$\X$, and set
\[
 a := A^{[t]}_{\X,\F;(1-t)\X_\kappa}(v).
\]
Then, for every sufficiently divisible $m$ and every $p\in\mathbb Z$, $s\in \cG^pE_m$ if and only if $v(s)-ma\ge p$. Equivalently,
\[
        \cG^pE_m  = \{s\in E_m \mid v(s)\ge ma+p\}.
\]
\end{lemma}

\begin{proof}
This is the same calculation as \cite[Lemma 8.25]{Xu2025}, with $K_X+\Delta$ replaced by the mixed adjoint canonical divisor $K^{[t]}_{\X/R,\F}$. Let us explain. 

Let $E'\subset W$ be the strict transform of $\X'_\kappa$. Since $E'$ is a divisor on $\X'$, its coefficient in
\[
K^{[t]}_{W/R,\F_W} - \rho'^*K^{[t]}_{\X'/R,\F'}
\]
is zero. On the other hand, its coefficient in
\[
K^{[t]}_{W/R,\F_W} - \rho^*K^{[t]}_{\X/R,\F}
\]
is $A^{[t]}_{\X,\F;(1-t)\X_\kappa}(v)=a$. Consequently,
\[
\coeff_{E'}(\rho'^*L'-\rho^*L)=-a.
\]
Thus a section $s\in E_m$ lies in $\pi^pE'_m$ if and only if $v(s)-ma\ge p$, proving the claim.
\end{proof}

Now, let $v':=\ord_{\X_\kappa}$, viewed as a divisorial valuation over $\X'$, and set $a' := A^{[t]}_{\X',\F;(1-t)\X'_\kappa}(v')$.

\begin{lemma}
\label{lem:support-comparison-filtrations}
The supports of the Duistermaat-Heckman measures $\nu_{\operatorname{DH}, \cH, R}$ and $\nu_{\operatorname{DH}, \cH', R}$ are $[-a, a']$ and $[-a',a]$, respectively.
\end{lemma}

\begin{proof}
This is the same argument as in \cite[Lemma 8.28]{Xu2025}. By Lemma \ref{lem:comparison-filtration-valuation-description}, for every $s\in E_m$, $s\in \cG^pE_m$ if and only if $v(s)-ma\geq p$. In particular, $\cG{-ma}E_m=E_m$  and therefore $\lambda_{\min}(\cH)\geq -a$. Similarly, $\lambda_{\min}(\cH')\geq -a'$. By Lemma \ref{lem:opposite-associated-graded}, there is a graded isomorphism $\operatorname{Gr}_{H}R_{\kappa} \simeq \operatorname{Gr}_{H'}R'_{\kappa}$ which sends the weight-$p$ part on the left to the weight-$(-p)$ part on the right. It follows that $\lambda_{\max}(\cH)\leq a'$.

We now prove the reverse inequality. Let $m\gg 0$ be sufficiently divisible, and choose a general section $s'\in H^0\!\left( \X'_{\kappa}, mL'_{\kappa} \right)$ which does not vanish along the centre of $v'$ on $\X'_{\kappa}$. By the symmetric form of Lemma \ref{lem:comparison-filtration-valuation-description}, the $\cH'$-weight of $s'$ is exactly $-ma'$. Hence, by Lemma \ref{lem:opposite-associated-graded} we obtain a nonzero element of $\cH$-weight $ma'$. Consequently, $\lambda_{\max}(\cH)\geq a'$  and hence $\lambda_{\max}(\cH)=a'$. Interchanging the two families gives $\lambda_{\max}(\cH')=a$. Hence, $\lambda_{\min}(\cH)=-a$, proving the required result.
\end{proof}

\begin{lemma}
\label{lem:mixed-slope-inequalities-ST}
We have $\mu^{[t]}(\cH)\le 0$ and $\mu^{[t]}(\cH')\le 0$.
\end{lemma}

\begin{proof}
We prove the inequality for $\cH$. The proof for $\cH'$ is identical. Let $v=\ord_{\X'_\kappa}$ and $a=A^{[t]}_{\X,\F;(1-t)\X_\kappa}(v)$. If $v$ is not exceptional over $\X$, then its centre is the Cartier special fibre and Lemma \ref{lem:comparison-filtration-valuation-description} identifies $\cH$ with a shift of the trivial filtration. The asserted inequality is then immediate. We may therefore assume that $v$ is exceptional. The plt statement in Theorem \ref{thm:dvr-fibre-ioa} gives $a>0$. By Lemma \ref{lem:comparison-filtration-valuation-description}, the base ideals of $\cH^\bullet$ are, up to integral closure, the restrictions to $\X_\kappa$ of the valuation ideals $\mathfrak a_m(v):=\{f\in \OO_\X \mid v(f)\ge ma\}$. Let $\mathfrak a_\bullet(v)=\{\mathfrak a_m(v)\}_{m\ge 1}$. Since $v(\mathfrak a_\bullet(v))=a$, the valuation $v$ gives
\[
        \lct^{[t]}_{\X,\F;(1-t)\X_\kappa} \bigl(\mathfrak a_\bullet(v)\bigr)  \le \frac{A^{[t]}_{\X,\F;(1-t)\X_\kappa}(v)}{v(\mathfrak a_\bullet(v))}= 1.
\]
The sequence $\mathfrak a_\bullet(v)$ is a graded sequence of ideals on the total space, so Corollary \ref{cor:F-compatible-IOA-graded} applies directly and gives
\[
        \lct^{[t]}_{\X,\F;(1-t)\X_\kappa} \bigl(\mathfrak a_\bullet(v)\bigr) = \lct^{[t]}_{X_\kappa,\F_\kappa} \bigl(\mathfrak a_{\bullet,\kappa}(v)\bigr).
\]
Thus the base ideals of $\cH^\bullet$ have mixed threshold at most $1$.  By the strict monotonicity statement in Lemma \ref{lem:mixed-Xu-344}, we obtain $\mu^{[t]}(\cH)\le0$.
\end{proof}

\begin{lemma} \label{lem:ding-zero-comparison-filtrations}
Assume that the two special fibres $(\X_\kappa,\F_\kappa,t)$ and $(\X'_\kappa,\F'_\kappa,t)$ are $t$-K-semistable. Then
\[
        \Ding^{[t]}(\cH)=0, \qquad \Ding^{[t]}(\cH')=0,
\]
and
\[
        \mu^{[t]}(\cH)=0=\mu^{[t]}(\cH').
\]
\end{lemma}

\begin{proof}
The filtrations $\cH^\bullet$ and $\cH'^{\,\bullet}$ are linearly bounded and multiplicative by Lemma \ref{lem:opposite-associated-graded}. Therefore Theorem \ref{thm:mixed-ding-compatible-filtrations}, which applies to arbitrary such filtrations, gives
\[
        \Ding^{[t]}(\cH)\ge 0, \qquad \Ding^{[t]}(\cH')\ge 0.
\]
On the other hand, using Lemma \ref{lem:opposite-associated-graded} and Lemma \ref{lem:mixed-slope-inequalities-ST}, we get
\[
\begin{aligned}
        \Ding^{[t]}(\cH)+\Ding^{[t]}(\cH') &= \bigl(\mu^{[t]}(\cH)-S(\cH)\bigr) +  \bigl(\mu^{[t]}(\cH')-S(\cH')\bigr)\\
        &= \mu^{[t]}(\cH)+\mu^{[t]}(\cH')\\
        &\le 0,
\end{aligned}
\]
because $S(\cH')=-S(\cH)$. Hence both non-negative terms must vanish:
\[
        \Ding^{[t]}(\cH)=0, \qquad \Ding^{[t]}(\cH')=0.
\]
Since
\[
        \Ding^{[t]}(\cH)=\mu^{[t]}(\cH)-S(\cH), \qquad \Ding^{[t]}(\cH')=\mu^{[t]}(\cH')-S(\cH'),
\]
and $S(\cH')=-S(\cH)$, we get
\[
        \mu^{[t]}(\cH)+\mu^{[t]}(\cH')=0.
\]
Together with
\[
        \mu^{[t]}(\cH)\le 0, \qquad \mu^{[t]}(\cH')\le 0,
\]
this implies
\[
        \mu^{[t]}(\cH)=0, \qquad  \mu^{[t]}(\cH')=0.
\]
\end{proof}

\begin{lemma}
\label{lem:finite-generation-comparison-graded}
Assume that the two special fibres are $t$-K-semistable. Then $\cH$  is finitely generated. Equivalently, $\cH'$ is finitely generated.
\end{lemma}

\begin{proof}
By Lemma \ref{lem:ding-zero-comparison-filtrations}, we have $\Ding^{[t]}(\cH)=0$. As before, we let $v:=\ord_{\X'_\kappa}$ and $a:=A^{[t]}_{\X,\F;(1-t)\X_\kappa}(v)$. If $v$ is not exceptional over $\X$, then the comparison filtration is a shift of the trivial filtration, so its section algebra and associated graded algebra are finitely generated. We henceforth assume that $v$ is exceptional; then $a>0$ by the plt form of Theorem \ref{thm:dvr-fibre-ioa}. Since $\mu^{[t]}(\cH)=0$, the $a$-shift $\cH^a$ satisfies $\mu^{[t]}(\cH^a)=a$. By Lemma \ref{lem:mixed-Xu-346}, applied to $\cH^a$, we get
\[
        \lct^{[t]}_{X_\kappa,\F_\kappa}
        \left(I^{(a)}_{\bullet}(\cH^a)\right)=1.
\]
Unwinding the $a$-shift exactly as in \cite[Corollary 8.30]{Xu2025}, this implies that for $m\gg 0$ sufficiently divisible and for $0<\varepsilon<1-t$ sufficiently small,
\[
        \lct^{[t]}_{\X,\F;(1-t)\X_\kappa} \left(\frac{1}{m}I_{m,(a-\varepsilon)m}(\cG)\right)\ge 1.
\]
Equivalently, for a general divisor $D\in \cG^{(a-\varepsilon)m}E_m$, the adjoint foliated structure
\[
 \left(\X,\F,(1-t)\X_\kappa+\frac{1}{m}D,t\right)
\]
is log canonical near $X_\kappa$.

Let $v=\ord_{X'_\kappa}$. By Lemma \ref{lem:comparison-filtration-valuation-description} $s\in \cG^pE_m$ if and only if $v(s)-ma\ge p$. The condition $D\in \cG^{(a-\varepsilon)m}E_m$ implies in particular $v(D)\ge m(a-\varepsilon)$. Hence
\[
\begin{aligned}
        A^{[t]}_{\X,\F;(1-t)\X_\kappa+\frac{1}{m}D}(v) &= A^{[t]}_{\X,\F;(1-t)\X_\kappa}(v) - \frac{1}{m}v(D)\\
        &\le a-(a-\varepsilon)  = \varepsilon < t\varepsilon_{\F}(v)+(1-t).
\end{aligned}
\]
If $r$ is a Cartier index of $L$, we write the sufficiently divisible degree as $m=rq$. Using Lemma \ref{lem:A0-unique-lc-place}, all the hypotheses of Theorem \ref{thm:relative-non-terminal-extraction-fg} are satisfied for the divisor $X'_\kappa$.

Theorem \ref{thm:relative-non-terminal-extraction-fg} gives finite
generation of
\[
        \mathcal R(\cG) = \bigoplus_{m,p}  \cG^pE_m\, u^{-p}.
\]
Its tensor over $\kappa$ yields $\oplus_p \cH^pR$, which is finitely generated.
\end{proof}

\begin{lemma} \label{lem:comparison-branches-special}
Assume that the two special fibres are $t$-K-semistable. The restrictions of the comparison extension to the coordinate divisors $(s=0)$ and $(u=0)$ are $\F$-compatible and $\F'$-compatible special test configurations. Both have mixed Donaldson--Futaki invariant zero and have the common central fibre
\[
 \cZ_0=\Proj\operatorname{Gr}_{\cH}R_\kappa \simeq \Proj\operatorname{Gr}_{\cH'}R'_\kappa.
\]
\end{lemma}

\begin{proof}
On $\cZ$, $(s=0)$ and $(u=0)$ correspond to two divisors $\cZ_s$ and $\cZ_u$. These are Cohen-Macaulay by an identical argument as in \cite[Proof of Theorem 8.31]{Xu2025}. Here, $(\cZ_s,\F_{\cZ_s})$ is an $\F_{\kappa}$-compatible test configuration for $(\X_\kappa,\F_\kappa)$ and $(\cZ_u,\F_{\cZ_u})$ is an $\F_{\kappa'}$-compatible test configuration for $(\X_\kappa,\F_{\kappa'})$. These test configurations have identical central fibres and opposite $\G_m$-actions, by Lemma \ref{lem:opposite-associated-graded}. Thus, by the definition of the mixed Donaldson--Futaki invariant, we obtain 
\[\DF^{[t]}(\cZ_s,\F_{\cZ_s})+\DF^{[t]}(\cZ_u,\F_{\cZ_u}) = 0\]
which in turn implies $\DF^{[t]}(\cZ_s,\F_{\cZ_s})=\DF^{[t]}(\cZ_u,\F_{\cZ_u}) = 0$. Thus, by \cite[Theorem 1.2]{Pap26} both test configurations are special. 
\end{proof}

\begin{lemma}\label{lem:admissibility-of-STR-extension}
Assume that the two special fibres are $t$-K-semistable. Let $(\cZ,\F_\cZ,\M)\to \ST_R$ be the extension constructed above. Then this is an admissible family of adjoint Fano foliated structures. Moreover, the fibre over the closed stacky point is $\cZ_0 = \Proj\operatorname{Gr}_{\cH}R_\kappa \simeq \Proj\operatorname{Gr}_{\cH'}R'_\kappa$.
\end{lemma}

\begin{proof}
Admissibility is smooth-local on the base, so we pull back to the standard atlas
\[
        B:=\Spec R[s,u]/(su-\pi)\longrightarrow \ST_R.
\]
By abuse of notation, we continue to write $ h:(\cZ,\F_\cZ,\M)\to B$ after this base change.

By construction,
\[
        \cZ = \Proj_B \bigoplus_{m\geq 0}  j_*h^\circ_*\OO_{\cZ^\circ}(m\M^\circ),
\]
where $j:B^\circ\hookrightarrow B$ is the complement of the closed point. Flatness of $h$ and relative ampleness of $\M=\OO_\cZ(1)$ are supplied by \cite[Corollary 8.23]{Xu2025}, as recorded in Lemma \ref{lem:underlying-STR-extension}.

The total space $\cZ^\circ$ is normal, and flatness implies that $\cZ\setminus\cZ^\circ$ has codimension two. The relative Proj algebra is reflexively extended from $B^\circ$, so $\OO_\cZ=j_*\OO_{\cZ^\circ}$. Thus $\cZ$ is $S_2$, while all of its codimension-one points lie in the normal open set $\cZ^\circ$; Serre's criterion gives normality. The closed fibre is normal as well, because it is the central fibre of either special branch from Lemma \ref{lem:comparison-branches-special}.

By Lemma \ref{lem:foliation-extends-over-STR}, the reflexive saturated extension
\[
T_{\F_\cZ/B} = \left(T_{\cZ/B}\cap j_*T_{\F^\circ/B^\circ}\right)^{\mathrm{sat}}
\]
is a relative algebraically integrable foliation. 

We verify precisely the mixed Koll\'ar condition required by Definition \ref{def:admissible-family-adjoint-foliated}. Choose $r>0$ sufficiently divisible so that $r(1-t)$ and $rt$ are integers, and set
\[
        \mathscr K_{r,t}  := \left( \omega_{\cZ/B}^{[r(1-t)]}  \widehat\otimes  \omega_{\F_\cZ/B}^{[rt]} \right)^{**},
\]
viewed as the divisorial hull from the relative regular locus. This is a mostly flat rank-one reflexive sheaf. By \cite[Proposition 7.8]{Xu2025} we obtain a locally closed hull-flat stratum of $B$ characterised by flatness and compatibility of $\mathscr K_{r,t}$ and all its reflexive powers with base change. The punctured open $B^\circ$ factors through this stratum because it is glued from the two original admissible families. By Lemma \ref{lem:comparison-branches-special}, both complete coordinate divisors factor through the same stratum, including their common closed point. Since $B=B^\circ\cup(s=0)\cup(u=0)$, the stratum is all of $B$. Hence the mixed sheaf $\mathscr K_{r,t}$ satisfies Koll\'ar's condition over $B$.

Over $B^\circ$ its dual is the anti-adjoint polarisation
\[
 \mathscr K_{r,t}^{\vee}|_{\cZ^\circ} \simeq \OO_{\cZ^\circ}(r\M^\circ).
\]
Both sides extend as reflexive rank-one sheaves on the normal scheme $\cZ$. No codimension-one point of $\cZ$ maps to the codimension-two origin of $B$, by flatness and the dimension formula. They therefore agree at every codimension-one point and hence globally:
\[
        \OO_\cZ(r\M) \simeq \mathscr K_{r,t}^{\vee}   =\OO_\cZ\!\left(-rK^{[t]}_{\cZ/B,\F_\cZ}\right).
\] 
Equivalently, $\M\sim_{\Q,B}-K^{[t]}_{\cZ/B,\F_\cZ}$. Since the construction is $\G_m$-equivariant, this descends from the atlas $B$ to $\ST_R$. Therefore $K^{[t]}_{\cZ/\ST_R,\F_\cZ}$ is $\Q$-Cartier and $\M$ is the relative anti-adjoint polarisation.

Finally, the fibre description follows from Lemma \ref{lem:underlying-STR-extension} and Lemma \ref{lem:opposite-associated-graded}.

We have verified the normality and flatness of $\mathcal Z\to B$, the existence of a saturated algebraically integrable relative foliation, Koll\'ar's condition for the required mixed reflexive powers, and the anti-adjoint identity, and the anti-adjoint identity $\mathcal M\sim_{\mathbb Q,B} -K^{[t]}_{\cZ/B,\F_\cZ}$. Thus the pulled-back family over $B$ is admissible, and hence so is the descended family over $\mathrm{ST}_R$.
\end{proof}

\begin{lemma}
\label{lem:closed-fibre-semistable-ST}
Assume that the two special fibres $(\X_\kappa,\F_\kappa,t)$ and $(\X'_\kappa,\F'_\kappa,t)$ are $t$-K-semistable. Then the closed stacky fibre $(\cZ_0,\F_{\cZ_0},\M|_{\cZ_0})$ is $t$-K-semistable.
\end{lemma}

\begin{proof}
By Lemma \ref{lem:comparison-branches-special}, either coordinate branch is a special test configuration of a $t$-K-semistable special fibre, has mixed Donaldson--Futaki invariant zero, and has central fibre $\cZ_0$. Corollary \ref{cor:zero-DF-special-limit-semistable} therefore applies and shows that $\cZ_0$ is $t$-K-semistable.
\end{proof}

\begin{theorem} \label{thm:adjoint-foliated-S-completeness}
Let $(\X,\F,t)\to \Spec R$ and $(\X',\F',t)\to \Spec R$ be two admissible families of $t$-K-semistable adjoint Fano foliated structures with isomorphic generic fibres $(X_K,\F_K,L_K)\simeq (X'_K,\F'_K,L'_K)$. Then the induced family over $\ST_R^\circ$ extends uniquely to an admissible family $(\cZ,\F_\cZ,\M)\to \ST_R$ whose geometric fibres are $t$-K-semistable.
\end{theorem}

\begin{proof}
Since the two families are admissible and $t$-K-semistable, \cite[Theorem 1.2.(3)]{Pap26b} implies that $(\X_\kappa,\F_\kappa,t)$ and $(\X'_\kappa,\F'_\kappa,t)$ are klt and that the underlying special fibres are potentially klt. By Corollary \ref{cor:dvr-total-potentially-klt}, the total spaces $\X$ and $\X'$ are potentially klt near their special fibres. By Theorem \ref{thm:dvr-fibre-ioa}, the adjoint foliated structures $(\X,\F,(1-t)\X_\kappa,t)$ and $(\X',\F',(1-t)\X'_\kappa,t)$ are plt near their special fibres.

The generic-fibre isomorphism glues the two families over $\ST_R^\circ$, giving $(\cZ^\circ,\F^\circ,\M^\circ)\to \ST_R^\circ$. Consider the comparison filtrations $\cH$ and $\cH'$ defined above.

By Lemma \ref{lem:finite-generation-comparison-graded}, $\operatorname{Gr}_{\cH}R_\kappa$ is finitely generated. Hence Lemma \ref{lem:underlying-STR-extension} constructs a unique flat polarised extension $(\cZ,\M)\to \ST_R$, which is normal by Lemma \ref{lem:admissibility-of-STR-extension}. By Lemma \ref{lem:foliation-extends-over-STR}, the glued relative foliation extends uniquely to a saturated relative algebraically integrable foliation $T_{\F_\cZ}\subset T_{\cZ/\ST_R}$ while by Lemma \ref{lem:admissibility-of-STR-extension}, the resulting family $(\cZ,\F_\cZ,\M)\to \ST_R$ is admissible.

Away from the closed stacky point, the geometric fibres are $t$-K-semistable because the family restricts to the original two admissible $t$-K-semistable DVR families. The fibre over the closed stacky point is $t$-K-semistable by Lemma \ref{lem:closed-fibre-semistable-ST}. Thus all geometric fibres of $(\cZ,\F_\cZ,\M)\to \ST_R$ are $t$-K-semistable.

This proves existence. Uniqueness of the extension follows from Lemma \ref{lem:underlying-STR-extension} and from the uniqueness of the saturated foliation in Lemma \ref{lem:foliation-extends-over-STR}.
\end{proof}

We also obtain the following uniqueness theorem for $t$-K-polystable degenerations.

\begin{theorem}
\label{thm:uniqueness-polystable-degeneration}
Let $(X,\F,t)$ be a $t$-K-semistable adjoint Fano foliated structure, with $L\sim_{\Q}-K^{[t]}_{X,\F}$. Let $(\X_1,\F_1,\L_1)\to \A^1$ and $(\X_2,\F_2,\L_2)\to \A^1$ be two $\F$-compatible special test configurations of $(X,\F,L)$ with $\DF^{[t]}(\X_i,\F_i,\L_i)=0$ and with $t$-K-polystable central fibres $(Y_i,\F_{Y_i},M_i)$. Then
\[
        (Y_1,\F_{Y_1},M_1) \simeq (Y_2,\F_{Y_2},M_2)
\]
as polarised foliated varieties.
\end{theorem}

\begin{proof}
Let $R=k[\pi]_{(\pi)}$. We base change the two test configurations by $\Spec R\to \A^1$, $\tau\mapsto \pi$. This gives two admissible $\F$-compatible DVR families $(\X_{1,R},\F_{1,R},\L_{1,R})\to \Spec R$ and $(\X_{2,R},\F_{2,R},\L_{2,R})\to \Spec R$. Their generic fibres are both canonically isomorphic to $(X,\F,L)\times_k K$, because a test configuration is trivial over $\A^1\setminus\{0\}$. Their special fibres are respectively $(Y_1,\F_{Y_1},M_1)$ and $(Y_2,\F_{Y_2},M_2)$.

Applying Theorem \ref{thm:adjoint-foliated-S-completeness} to the two DVR families, we obtain an $\F$-compatible admissible family $(\cZ,\F_\cZ,\M)\to \ST_R$ extending the glued family over $\ST_R^\circ$. The restriction of this family to the two coordinate branches of the special fibre of $\ST_R$ gives $\F$-compatible special test configurations $(\cZ_s,\F_{\cZ_s},\M_s)\to \A^1_\kappa$ of $(Y_1,\F_{Y_1},M_1)$, and $(\cZ_u,\F_{\cZ_u},\M_u)\to \A^1_\kappa$ of $(Y_2,\F_{Y_2},M_2)$, with common central fibre over the closed
stacky point $0\in \ST_R$.

By Lemma \ref{lem:comparison-branches-special}, these two special test configurations have zero mixed Donaldson--Futaki invariant:
\[
        \DF^{[t]}(\cZ_s,\F_{\cZ_s},\M_s)=0=\DF^{[t]}(\cZ_u,\F_{\cZ_u},\M_u).
\]
Since $(Y_i,\F_{Y_i},M_i)$ are assumed to be $t$-K-polystable, every $\F$-compatible special test configuration with zero mixed Donaldson--Futaki invariant is an $\F$-compatible product test configuration. Therefore $(\cZ_s,\F_{\cZ_s},\M_s)$ is the product test configuration of $(Y_1,\F_{Y_1},M_1)$, and $(\cZ_u,\F_{\cZ_u},\M_u)$ is the product test configuration of $(Y_2,\F_{Y_2},M_2)$.

The two product test configurations have the same central fibre, namely the fibre of $\cZ\to \ST_R$ over the closed stacky point. Since every fibre of a product test configuration is isomorphic, as a polarised foliated variety, to its general fibre, the common central fibre is isomorphic to both $(Y_1,\F_{Y_1},M_1)$ and $(Y_2,\F_{Y_2},M_2)$. Therefore
\[
(Y_1,\F_{Y_1},M_1) \simeq (Y_2,\F_{Y_2},M_2).
\]
\end{proof}

\section{\texorpdfstring{Reductivity of the automorphism group of $t$-K-polystable adjoint Fano foliated structures}{Reductivity of the automorphism group of t-K-polystable adjoint Fano foliated structures}}\label{sec:reductivity-aut}

Throughout, $k$ is an algebraically closed field of characteristic zero. We let $(X,\F,t)$ be a $t$-K-polystable adjoint Fano foliated structure and choose $r>0$ sufficiently divisible such that $L:=\mathscr L_{r,t}=\left(\omega_X^{[r(1-t)]}\widehat\otimes\omega_{\F}^{[rt]}\right)^{[-1]}$ is very ample. We denote by
\[
\Aut_t(X,\F):= \Aut(X,\F,\mathscr L_{r,t})
\]
the group of automorphisms $g\in\Aut(X)$ such that $g_*T_{\F}=T_{\F}$. Every such automorphism preserves $\omega_X$, $\omega_{\F}$, and hence $\mathscr L_{r,t}$. In particular,
\[
        \Aut_t(X,\F)= \left\{g\in \Aut(X)\ \middle|\ g_*T_{\F}=T_{\F}, \text{ and } g^*\mathscr L_{r,t}\simeq \mathscr L_{r,t} \right\},
\]
hence
\[
        \Aut_t(X,\F)=\Aut(X,\F)\subset \Aut(X,L)\subset \PGL(H^0(X,L)).
\]
The inclusion $T_{\F}\subset T_X$ determines a quotient $T_X\twoheadrightarrow Q:=T_X/T_{\F}$. The group $\Aut(X,L)$ acts on the relevant Quot scheme parametrising quotients of $T_X$ with the Hilbert polynomial of $Q$. The subgroup preserving $T_{\F}$,  equivalently the quotient $T_X\twoheadrightarrow Q$, is precisely the stabiliser of this Quot point $[T_X\twoheadrightarrow Q]$. Hence it is closed in $\Aut(X,L)$. Since $\Aut(X,L)$ is a linear algebraic group, so is $\Aut(X,\F)$.

For a linear algebraic group $G$, we define
\[
        \Lambda_G := \{\lambda(\pi)\in G(k((\pi)))\mid \lambda:\G_m\to G \text{ is a one-parameter subgroup}\}.
\]
We will use the following form of the converse to the Iwahori
decomposition.

\begin{proposition}[{\cite[Proposition 4.2]{ABHLX20}}]\label{prop:ABHLX20-4.2}
    If
\[
        G(k((\pi)))=G(k\llbracket\pi\rrbracket)\Lambda_G G(k\llbracket\pi\rrbracket),
\]
then $G$ is reductive.
\end{proposition}
We will also use the following lemma, which reduces the double-coset equality to algebraic points.

\begin{lemma}[{\cite[Lemma 4.3]{ABHLX20}}]\label{lem:ABHLX20-4.3}
    If $G$ is a linear algebraic group, then for any $g \in G(K)$, there is an algebraic point $g_0$ such that $g \cdot g^{-1}_0 \in G(R)$, where algebraic means that $g_0 \in G(k(C))$ for the function field $k(C)$ of a smooth curve over $k$ embedded in $K$ via a dominant morphism $\Spec(R) \to C$.
\end{lemma}

\subsection{Isotrivial extensions}

Let $R$ be a DVR essentially of finite type over $k$, with uniformiser $\pi$, fraction field $K$, and residue field $\kappa$. We define $S_R:=\Spec R[s,u]/(su-\pi)$ and note that $S_R^\circ:=S_R\setminus\{(s=u=0)\}$.

The scheme $S_R^\circ$ is the union of $\Spec R[s]_s$ and $\Spec R[u]_u$ glued along $\Spec K[s]_s\simeq \G_{m,K}$. The group $\G_m$ acts on $S_R$ by $\lambda\cdot(s,u)=(\lambda s,\lambda^{-1}u)$. Note that  $\ST_R^\circ = \left[S_R^\circ/\G_m\right]$ and $\ST_R = \left[S_R/\G_m\right]$.

Let
\[
\alpha_K\colon (X_K,\F_K,L_K) \xrightarrow{\sim} (X_K,\F_K,L_K)
\]
be an element of $\operatorname{Aut}(X_K,\F_K)$. Using this, we glue the two trivial families $(X,\F,L)\times \Spec R[s]_s$ and $(X,\F,L)\times \Spec R[u]_u$ over $\Spec K[s]_s$ using the $\G_m$-equivariant isomorphism induced by $\alpha_K$. This gives a $\G_m$-equivariant admissible family
\[
(\cZ^\circ,\F_{\cZ^\circ},\M^\circ)\to S_R^\circ.
\]
Equivalently, after quotienting by $\G_m$, it defines a family over $\ST_R^\circ$.

\begin{lemma} \label{lem:isotrivial-STR-extension-foliated}
Let $(X,\F,t)$ be $t$-K-polystable. Then the resulting $\G_m$-equivariant admissible family over $S_R^\circ$  $(\cZ^\circ,\F_{\cZ}^\circ)\rightarrow S_R^\circ$ extends to a $\G_m$-equivariant admissible family $(\cZ,\F_{\cZ},\M)\to S_R$, where the fibre over $s=u=0$ is isomorphic to the central fibre, i.e. 
\[
    (\cZ_{\overline{0}}, \F_{\cZ, \overline{0}}, \M_{\overline{0}})\simeq (X_{\overline{\kappa}},\F_{\overline{\kappa}},L_{\overline{\kappa}}).
\]
\end{lemma}

\begin{proof}
By Theorem \ref{thm:adjoint-foliated-S-completeness}, the family over $\operatorname{ST}_R^\circ$ extends uniquely to an admissible family over $\operatorname{ST}_R$. Pulling back to the standard atlas gives a $\mathbb G_m$-equivariant admissible family over $S_R$.

The restrictions to the two coordinate axes are $\mathcal F$-compatible special test configurations of $(X_\kappa,\mathcal F_\kappa,t)$ with zero mixed Donaldson--Futaki invariant. Since $(X,\mathcal F,t)$ is $t$-K-polystable, both are product test configurations. Their common central fibre is therefore isomorphic to $(X_\kappa,\mathcal F_\kappa,L_\kappa)$.
\end{proof}

\subsection{Reductivity of the automorphism group}

\begin{theorem}[Reductivity of the automorphism group]
\label{thm:aut-reductive-adjoint-foliated}
Let $(X,\F,t)$ be a $t$-K-polystable adjoint Fano foliated structure. Then $\Aut(X,\F)$ is reductive.
\end{theorem}

\begin{proof}
We follow the proof of \cite[Proposition 4.4]{ABHLX20}. We set $G:=\Aut(X,\F)$, $R = k\llbracket\pi\rrbracket$ and $K = k((\pi))$. As observed in the preceding discussion, $G$ is a linear algebraic group. By Proposition \ref{prop:ABHLX20-4.2}, it suffices to prove $G(K)=G(R)\Lambda_G G(R)$, while Lemma \ref{lem:ABHLX20-4.3}, reduces this to algebraic points of $G(K)$. 

Thus, let $C$ be a smooth pointed curve with closed point $x$. Set $R_0:=\mathcal O_{C,x}$ and choose a uniformiser $\pi\in R_0$. Since $C$ is smooth at $x$, $\widehat{R_0}\simeq k\llbracket\pi\rrbracket=R$. Let $K_0=\operatorname{Frac}(R_0)$; then the completion map induces an embedding $K_0\hookrightarrow K=k((\pi))$. Let $g\in G(K_0)$ be an algebraic point. We shall prove that the image of $g$ in $G(K)$ lies in $G(R)\Lambda_G G(R)$.

The element $g$ gives an isomorphism of polarised foliated generic fibres
\[
        g: (X_{K_0},\F_{K_0},L_{K_0}) \xrightarrow{\sim} (X_{K_0},\F_{K_0},L_{K_0}).
\]
Using $g$, we glue the two trivial families over
\[
        S_{R_0}^\circ = \Spec R_0[s,u]/(su-\pi)\setminus\{s=u=0\}.
\]

Applying Lemma \ref{lem:isotrivial-STR-extension-foliated} over $R_0$, we obtain an extension over $S_{R_0}$. Base-changing along $R_0\rightarrow \widehat{R}_0\simeq k[[\pi]]$ yields a $\GG_m$-equivariant admissible family over $S_{\widehat{R}_0}$
\[
        (\cZ,\F_{\cZ},\M) \to S:=\Spec R[s,u]/(su-\pi),
\]
whose fibre over $s=u=0$ is isomorphic to $(X,\F,L)$. Under the identification of the closed fibre with $(X,\F,L)$, the $\G_m$-action on the closed fibre preserves $L$ and the saturated foliation $T_{\F}$. Hence it defines a one-parameter subgroup $\lambda:\G_m\to G$.

We now compare $(\cZ,\F_{\cZ},\M)$ with the product family $(X,\F,L)\times S\to S$, where $\G_m$ acts on $S$ by $z\cdot(s,u)=(z s,z^{-1}u)$ and acts on the factor $(X,\F,L)$ through the one-parameter subgroup $\lambda$.

Let $I_{\mathrm{pol}}$ be the Isom scheme between the polarised family $(\mathcal Z,\mathcal M)\to S$  and the product polarised family $(X,L)\times S\to S$. Inside $I_{\mathrm{pol}}$, let $I$ be the closed subscheme of isomorphisms carrying $T_{\F_{\mathcal Z}}$ to the pullback of $T_{\F}$. This is a closed condition: equivalently, the induced quotient of the relative tangent sheaf must equal the corresponding point of the relative Quot scheme. By Lemma \ref{lem:isotrivial-STR-extension-foliated} the fibre over $s=u=0$ is isomorphic to $(X,\F,L)$. Away from the closed point, the family is obtained by gluing trivial families. Moreover, the restrictions to the coordinate divisors are product test configurations. Hence every geometric fibre of $(\mathcal Z,\F_{\mathcal Z},\mathcal M)\to S$ is isomorphic to $(X,\F,L)$. In particular, the morphism $I\to S$ is a $G$-torsor over $S$ by \cite[Lemma 2.3.2]{SdJ10}. The foliation-preserving condition is closed, so the polarised Isom torsor restricts to the $G=\Aut(X,\F)$-torsor of foliation-preserving isomorphisms.

The $\G_m$-actions on both families induce a $\G_m$-action on $I$, and the projection $I\to S$ is $\G_m$-equivariant. The chosen identification of the fibre over $s=u=0$ with $(X,\F,L)$ gives a $\G_m$-equivariant section $S_0:=\Spec k\to I$. Since $I\to S$ is a $G$-torsor and $G$ is smooth in characteristic zero, $I\to S$ is smooth. Because $\G_m$ is linearly reductive, the usual infinitesimal lifting property may be applied on invariant affine neighbourhoods and the lifts may be chosen equivariantly. Thus the section over $S_0$ extends to compatible $\G_m$-equivariant sections over all infinitesimal neighbourhoods
\[
        S_n:=\Spec \cO_S/\mathfrak m_0^{n+1}.
\]

The equivariant effectivity argument in \cite[proof of Proposition 4.4]{ABHLX20} now algebraises this compatible formal section. More explicitly, $I$ is affine over $S$ because it is a torsor under the affine group $G$, and the completed graded algebra maps defined by the sections on the $S_n$ are effective over the complete base. Their weight decompositions therefore give a $\G_m$-equivariant section $S\to I$. Consequently there is a $\G_m$-equivariant isomorphism
\[
        (\cZ,\F_{\cZ},\M) \simeq (X,\F,L)\times S,
\]
where the product has diagonal $\G_m$-action via $\lambda$ on $(X,\F,L)$.

We now restrict this isomorphism to $S^\circ=S\setminus\{s=u=0\}$ and quotient by the $\G_m$-action. The original family over $S^\circ$ was obtained by gluing two trivial families using $g\in G(K)$, whereas the diagonal product family is obtained by gluing using $\lambda(\pi)\in G(K)$. The equivariant isomorphism over $S^\circ$ therefore gives elements $a,b\in G(R)$ such that $a\cdot g=\lambda(\pi)\cdot b$. Equivalently,
\[
        g=a^{-1}\lambda(\pi)b \in G(R)\Lambda_GG(R).
\]
Thus every algebraic point of $G(K)$ lies in $G(R)\Lambda_GG(R)$. Proposition \ref{prop:ABHLX20-4.2} and Lemma \ref{lem:ABHLX20-4.3} therefore imply that $G$ is reductive.
\end{proof}

\subsection{\texorpdfstring{Finiteness of the automorphism group for $t$-K-stable adjoint Fano foliated structures}{Finiteness of the automorphism group for t-K-stable adjoint Fano foliated structures}}

We now tackle the $t$-K-stable case. The proof is similar to that of Theorem \ref{thm:uniqueness-polystable-degeneration}.

\begin{theorem} \label{thm:stable-separatedness-adjoint-foliated}
Let $C$ be a smooth pointed curve with closed point $0$, and let $(\X,\F,t)\to C$ and $(\X',\F',t)\to C$ be admissible families of $t$-K-semistable adjoint Fano foliated structures. Suppose there is an isomorphism over the punctured curve $C^\circ=C\setminus\{0\}$:
\[
        (\X,\F)|_{C^\circ}\simeq (\X',\F')|_{C^\circ}.
\]
If the central fibre $(X_0,\F_0,t)$ is $t$-K-stable, then the isomorphism over $C^\circ$ extends uniquely
to an isomorphism
\[
        (\X,\F,t)\simeq (\X',\F',t)
\]
over $C$.
\end{theorem}

\begin{proof}
Since the given isomorphism already exists away from $0$, it suffices to extend it over a neighbourhood of $0$. We may therefore replace $C$ by the spectrum of the DVR $R:=\cO_{C,0}$. 

By Theorem \ref{thm:adjoint-foliated-S-completeness}, the two DVR families glue and extend to an admissible family $(\cZ,\F_{\cZ})\rightarrow \ST_R$ over $\ST_R$ and the two special fibres $(X_0,\F_0,t)$ and $(X'_0,\F'_0,t)$ are $S$-equivalent, in the sense that there are $\F$ and $\F'$-compatible special test configurations $(\Y,\F_{\Y})\to\A^1$ of $(X_0,\F_0,t)$ and $(\Y',\F_{\Y'})\to\A^1$ of $(X'_0,\F'_0,t)$ whose central fibres are isomorphic as adjoint Fano foliated structures. Moreover, these special test configurations have zero mixed Donaldson--Futaki invariant. 

Since $(X_0,\F_0,t)$ is $t$-K-stable, every
$\F$-compatible special test configuration of
$(X_0,\F_0,t)$ with zero mixed Donaldson--Futaki invariant is
trivial. Hence $(\Y,\F_{\Y})\to\A^1$ is the trivial test configuration, and thus its central fibre is $(X_0,\F_0,t)$. Since the two special test configurations have isomorphic central fibres, we obtain that $(X'_0,\F'_0,t)$ degenerates by a zero-$\DF^{[t]}$ special test configuration to $(X_0,\F_0,t)$.

Recall that by the proof of Theorem \ref{thm:adjoint-foliated-S-completeness} the two $\F$- and $\F'$-compatible special test configurations above are induced by $\F$ and $\F'$-compatible filtrations $\cH^\bullet$ and $\cH'^{\,\bullet}$ on the two special fibres. In particular, since $(\Y,\F_{\Y})\to\A^1$ is trivial, the filtration $\cH^\bullet$ is also trivial up to a shift. After normalising the generic identification by a character, we may assume this shift is zero, i.e., $\cH^0=H^0(X_\kappa,mL_\kappa)$, which implies that $E_m=(E_m\cap E'_m)+\pi E_m$. 

Set $M_m:=E_m/(E_m\cap E'_m)$. The equality above gives $M_m=\pi M_m$. Since $M_m$ is a finite $R$-module, Nakayama's lemma implies $M_m=0$. Therefore $E_m=E_m\cap E'_m$, and hence $E_m\subset E'_m$. Recall that we have an isomorphism $\operatorname{Gr}_\cH R_\kappa\simeq\operatorname{Gr}_{\cH'}R'_\kappa$. Since $\cH^\bullet$ is trivial, this isomorphism implies that $\cH'^{\,\bullet}$ is also trivial (up to a shift). In particular, the test configuration $(\Y',\F_{\Y'})\to\A^1$ of $(X'_0,\F'_0,t)$ is also trivial, and hence we obtain the isomorphism $(X'_0,\F'_0,t)\simeq (X_0,\F_0,t)$.

Furthermore, the same argument as above shows that $E'_m\subset E_m$, and hence we have $E_m= E'_m$. By the Koll\'ar base-change condition in the definition of admissible families, and by relative Proj, the anti-adjoint section algebras recover the two families, and since $E_m= E'_m$ we have 
\[
        \X\simeq \Proj_R\bigoplus_m E_m\simeq \Proj_R\bigoplus_m E'_m\simeq\X'.
\]
We now show that the foliations agree. 

We transport $\F'$ via the isomorphism $\phi:\X\to\X'$, and denote the resulting foliation on $\X$ by $\phi^*\F'$. Thus $T_{\phi^*\F'}=d\phi^{-1}(T_{\F'})$. Hence, we may regard both $T_{\F/R}$ and $T_{\phi^*\F'}$ as saturated subsheaves of the same relative tangent sheaf $T_{\X/R}$. These two saturated subsheaves agree over the generic fibre by construction. 

Since $\F, \phi^{*}\F'\subset T_{\X/R}$ are saturated, the
quotients $Q:=T_{\X/R}/\F$ and $Q':=T_{\X/R}/\phi^{*}\F'$ are torsion-free as sheaves on $\X$. In particular, they have no $R$-torsion, hence they are flat over the DVR $R$.

Because $Q$ and $Q'$ agree over the generic fibre and are flat over $R$, they have the same Hilbert polynomial with respect to $\L$. Thus the two quotients $T_{\X/R}\twoheadrightarrow Q$ and $T_{\X/R}\twoheadrightarrow Q'$ define two $R$-points of the same relative Quot scheme $\operatorname{Quot}_{T_{\X/R}/\X/R}^{P}$. The relative Quot scheme is separated over $R$. Since the two $R$-points agree over the generic point $\operatorname{Spec}K$, they are equal. Hence the two quotients $T_{\X/R}\twoheadrightarrow Q$ and $T_{\X/R}\twoheadrightarrow Q'$ are isomorphic compatibly with the quotient maps from $T_{\X/R}$. Their kernels therefore coincide:
\[
T_{\F}=T_{\phi^*\F'}.
\]
It follows that $(\X,\F,t)\simeq (\X',\F',t)$. Finally, uniqueness follows from Lemma \ref{lem:foliation-extends-over-STR}.
\end{proof}

\begin{corollary}[Finiteness of automorphisms]
\label{cor:finite-aut-t-K-stable}
Let $(X,\F,t)$ be a $t$-K-stable adjoint Fano foliated structure. Then $\Aut(X,\F)$ is finite.
\end{corollary}

\begin{proof}
The group $\Aut(X,\F)$ is a linear algebraic group, hence affine. It is enough to prove that it is proper.

Let $C$ be a smooth pointed curve and let $g:C^\circ\to \Aut(X,\F)$ be a morphism. The map $g$ induces an isomorphism over $C^\circ$ between the two trivial admissible families $(X,\F,t)\times C$ and $(X,\F,t)\times C$. By Theorem \ref{thm:stable-separatedness-adjoint-foliated}, this isomorphism extends uniquely over $C$. Therefore $g$ extends to a morphism $C\to \Aut(X,\F)$. Thus $\Aut(X,\F)$ satisfies the valuative criterion of properness. Since it is affine and of finite type over $k$, it is finite over $k$.
\end{proof}

\newcommand{\etalchar}[1]{$^{#1}$}

\end{document}